\documentclass[a4paper,leqno]{amsart}

\newcommand{\ban}[1]{\mathcal{#1}}

\usepackage{amsmath}

\usepackage{mathdots} 

\usepackage{amssymb}

\usepackage{amsbsy}

\usepackage{enumerate}

\usepackage{mathptmx} 

\usepackage{hyperref}

\usepackage{mathrsfs} 

\usepackage{txfonts}  

\usepackage{stmaryrd} 

\usepackage{fancyhdr}  

\usepackage{comment}

\usepackage{tikz}

\usetikzlibrary{arrows,positioning}

\usepackage{makecell}

\usepackage{pgfplots}

\usepackage{graphicx,amssymb}

\usepackage{ifthen}

\usepackage{epsfig}

\newtheorem{theo}{Theorem}

\newtheorem{prop}{Proposition}

\newtheorem{lemm}{Lemma}

\newtheorem{coro}{Corollary}

\newtheorem{defi}{Definition}

\newtheorem{clai}{Claim}

\newtheorem{rema}{Remark}

\newtheorem{exam}{Example}

\newtheorem{theoalpha}{Theorem}

\renewcommand{\le}{\leqslant}
\renewcommand{\leq}{\leqslant}

\renewcommand{\ge}{\geqslant}
\renewcommand{\geq}{\geqslant}

\newcommand{\vep}{\varepsilon}

\newcommand{\bE}{\ensuremath{\mathbb E}}

\newcommand{\bN}{\ensuremath{\mathbb N}}

\newcommand{\bP}{\ensuremath{\mathbb P}}

\newcommand{\bR}{\ensuremath{\mathbb R}}

\newcommand{\sD}{\ensuremath{\mathsf D}}

\newcommand{\sL}{\ensuremath{\mathsf L}}

\newcommand{\sX}{\ensuremath{\mathsf X}}
\newcommand{\sY}{\ensuremath{\mathsf Y}}

\newcommand{\cB}{\ensuremath{\mathcal B}}

\newcommand{\cF}{\ensuremath{\mathcal F}}

\newcommand{\cP}{\ensuremath{\mathcal P}}

\newcommand{\scF}{\ensuremath{\mathscr F}}

\newcommand{\cdist}[1]{\mathsf{c}_{#1}}

\newcommand{\dia}{\ensuremath{\mathsf{D}}}

\newcommand{\dX}{\ensuremath{\mathsf{d_X}}}
\newcommand{\dY}{\ensuremath{\mathsf{d_Y}}}

\newcommand{\sd}{\ensuremath{\mathsf{d}}}

\newcommand{\ham}{\ensuremath{\mathsf{H}}}

\newcommand{\met}[1]{\mathsf{#1}}

\newcommand{\banX}{\ban X}
\newcommand{\banY}{\ban Y}
\newcommand{\metX}{\met X}

\newcommand{\banXn}{(\ban X,\norm{\cdot})}

\newcommand{\metXd}{(\met X,\dX)}
\newcommand{\metYd}{(\met Y,\dY)}

\newcommand{\La}{\mathsf{La}}

\newcommand{\Wa}{\mathsf{Wa}}

\newcommand{\dwas}{\ensuremath{\mathsf{d_{W_{1}}}}}

\newcommand{\eqd}{\stackrel{\mathrm{def}}{=}}

\newcommand{\1}{\mathbf{1}}
\newcommand{\Ex}{\mathbb{E}}

\newcommand{\ksubset}{\!\subset\!}

\newcommand{\abs}[1]{\lvert #1\rvert}
\newcommand{\bigabs}[1]{\big\lvert #1\big\rvert}

\newcommand{\norm}[1]{\| #1\|}
\newcommand{\bnorm}[1]{\Big\| #1\Big\|}

\usepackage[alphabetic]{amsrefs}

\newcommand{\Pk}{\mathsf{P}_k}

\newcommand{\os}{\oslash}
\DeclareMathOperator{\Lip}{Lip}
\DeclareMathOperator{\Per}{Per}
\DeclareMathOperator{\Lin}{\cB}
\DeclareMathOperator{\LF}{LF}

\def \eps {\varepsilon}

\DeclareMathOperator{\med}{med}

\usepackage{apptools}
\AtAppendix{\counterwithin{lemm}{section}}
\AtAppendix{\counterwithin{prop}{section}}

\usepackage{tikz-cd}
\usepackage[normalem]{ulem}
\usepackage{float}

\begin{document}

\title[A tale of two dimensions]{$L_1$-distortion of Wasserstein metrics: a tale of two dimensions}

\author{F.~ Baudier}
\address{F.~ Baudier, Texas A\&M University, College Station, TX 77843, USA}
\email{florent@math.tamu.edu}

\author{C.~ Gartland}
\address{C.~ Gartland, Texas A\&M University, College Station, TX 77843, USA}
\email{cgartland@math.tamu.edu}

\author{Th.~ Schlumprecht}
\address{Th.~Schlumprecht, Department of Mathematics, Texas A\&M University, College Station, TX
  77843-3368, USA, and Faculty of Electrical Engineering, Czech Technical University in Prague,
  Technick\'a 2, 166 27, Prague 6, Czech Republic}
\email{schlump@math.tamu.edu}

\thanks{F. B. was partially supported by the National Science Foundation under Grant Numbers DMS-1800322 and DMS-2055604, C. G. was partially supported by an AMS-Simons travel grant, Th. S. was partially supported by the National Science Foundation under Grant Numbers DMS-1764343 and DMS-2054443.}
	


\begin{abstract}
By discretizing an argument of Kislyakov, Naor and Schechtman proved that the 1-Wasserstein metric over the planar grid $\{0,1,\dots n\}^2$ has $L_1$-distortion bounded below by a constant multiple of $\sqrt{\log n}$. We provide a new ``dimensionality" interpretation of Kislyakov's argument, showing that, if $\{G_n\}_{n=1}^\infty$ is a sequence of graphs whose isoperimetric dimension and Lipschitz-spectral dimension equal a common number $\delta \in [2,\infty)$, then the 1-Wasserstein metric over $G_n$ has $L_1$-distortion bounded below by a constant multiple of $(\log |G_n|)^{\frac{1}{\delta}}$. We proceed to compute these dimensions for $\oslash$-powers of certain graphs. In particular, we get that the sequence of diamond graphs $\{\dia_n\}_{n=1}^\infty$ has isoperimetric dimension and Lipschitz-spectral dimension equal to 2, obtaining as a corollary that the 1-Wasserstein metric over $\dia_n$ has $L_1$-distortion bounded below by a constant multiple of $\sqrt{\log| \dia_n|}$. This answers a question of Dilworth, Kutzarova, and Ostrovskii and exhibits only the third sequence of $L_1$-embeddable graphs whose sequence of 1-Wasserstein metrics is not $L_1$-embeddable.
\end{abstract}

\maketitle

\setcounter{tocdepth}{1}
\tableofcontents

\section{Introduction}
Let $\metXd$ be a finite metric space and $\cP(\metX)$ the set of probability measures on $\metX$. The 1-Wasserstein metric $\dwas$ on $\mathcal{P}(\metX)$ is defined by
\begin{equation*}
    \dwas(\mu,\nu) = \inf_{\gamma} \int_{\metX \times \metX} \dX(x,y) d\gamma(x,y),
\end{equation*}
where the infimum is over all $\gamma \in \cP(\metX \times \metX)$ with marginals $\mu$ and $\nu$. The distance $\dwas(\mu,\nu)$ can be interpreted as the cost of transporting the mass of $\mu$ onto the mass of $\nu$ where cost is directly proportional to the distance moved and to the quantity of mass transported. The metric space $(\cP(\sX), \dwas)$ is referred to as the \emph{1-Wasserstein space} over $\metX$, and we denote it by $\Wa_1(\metX)$. Wasserstein metrics are of high theoretical interest but most importantly they are fundamental in applications in countless areas of applied mathematics, engineering, physics, computer science, finance, social sciences, and more. Indeed, they provide a natural and robust way to measure the (dis)similarity between the numerous objects which can be modeled by probability distributions. We point the interested reader to some of the many monographs discussing Wasserstein metrics and optimal transport in general (\cite{RR98}, \cite{RR98_II}, \cite{Villani03}, \cite{Villani09}, \cite{Santambrogio15}, \cite{ABS21}, \cite{FigalliGlaudo21}).
For both theoretical and practical reasons, the problem of low-distortion embeddings of $\Wa_1(\metX)$ into the Banach space $L_1$ has attracted much interest. We recall here that the distortion of one metric space $\metXd$ into another $\metYd$ is the quantity $\cdist{\sY}(\metX) := \inf_f \Lip(f) \cdot \Lip(f^{-1})$, where the infimum is over all injections $f: \sX \to \sY$ and $\Lip(f)$ is the Lipschitz constant of $f$. Of course, since the embedding $\delta \colon \sX \to W_1(\sX)$ given by $x \mapsto \delta_x$ is isometric, the distortion of $\Wa_1(\metX)$ into $L_1$ is at least as large as that of $\metX$ into $L_1$. Given a sequence of metric spaces $\{\metX_n\}_{n\in \bN}$ such that $\sup_{n\in \bN} \cdist{L_1}(\metX_n)<\infty$, it is a natural and important problem to understand whether or not $\sup_{n\in \bN} \cdist{L_1}(\Wa_1(\metX_n))$ remains finite or not. It has been observed by many that the 1-Wasserstein metric over a tree admits a closed-formula from which isometric embeddability into an $L_1$-space follows immediately (cf. \cite{Charikar02}, \cite{EvansMatsen12}, and the detailed analysis in \cite{VMP}). However, this problem has turned out to be difficult in general, and nontrivial lower bounds for the $L_1$-distortion of Wasserstein metrics are known to exist only in essentially two situations: when the ground space is the $n$-by-$n$ planar grid $[n]^2:=\{0,1,\dots n\}^2$ or the $k$-dimensional Hamming cube $\ham_k$, i.e. $\{0,1\}^k$ equipped with the Hamming metric counting the number of differing corresponding entries. Indeed, by \cite[Theorem 1.1]{NS07} it holds that $\cdist{L_1}(\Wa_1([n]^2)) = \Omega(\sqrt{\log n}) =  \Omega(\sqrt{\log |[n]^2|})$, and by \cite[Corollary 2]{KhotNaor06}, it holds that $\cdist{L_1}(\Wa_1(\ham_k)) = \Omega(k) = \Omega(\log |\ham_k|)$, where $\abs{\cdot}$ denotes cardinality. Note that the fact that $\sup_{k\in \bN}\cdist{L_1}(\Wa_1(\ham_k))=\infty$ was essentially proved by Bourgain \cite{Bourgain86}. The main result of this article is the provision of a third example of a family of spaces $\{\metX_n\}_{n\in\bN}$ which embed into $L_1$ with constant distortion but for which $\{\Wa_1(\metX_n)\}_{n\in \bN}$ does not. Our family is a sequence of generalized diamond graphs $\sD_{k,k}^{\os n}$ equipped with the shortest path metric\footnote{Each graph $\sD_{k,k}^{\os n}$ has $L_1$-distortion bounded above by 14 \cite[Theorem~4.1]{GNRS}.} (see Example~\ref{ex:diamonds} and Definition~\ref{def:slash-power}, also Figures~\ref{fig:diamonds},~\ref{fig:D22-slash-D22}), and the following theorem implies a negative answer to a question of Dilworth-Kutzarova-Ostrovskii \cite[Problem 6.6]{DKO20} about the classical diamond graphs $\{\sD_{2,2}^{\os n}\}_{n\in\bN}$.

\begin{theoalpha} \label{thm:diamondWasserstein}
For each fixed integer $k \geq 2$, $\cdist{L_1}(\Wa_1(\sD_{k,k}^{\os n})) = \Omega_k\left(\sqrt{\log|\sD_{k,k}^{\os n}|}\right)$.
\end{theoalpha}

We deduce Theorem~\ref{thm:diamondWasserstein} from a more general theorem on Wasserstein spaces over graphs with certain dimension estimates (see Theorem~\ref{thm:generalWasserstein} and the sentence following it). Before further discussion, we set notation and introduce the key definitions.

Throughout this article, we adopt the convention that graphs are finite, connected, directed, with at least one edge, and without self-loops or multiple edges between the same pair of vertices. For a graph $G$, we write $V(G)$ for the vertex set and $E(G)$ for the edge set. For a (directed) edge $e = (u,v) \in E(G)$, we write $e^-$ for $u$ and $e^+$ for $v$. Recall that a sequence $\{u_i\}_{i=0}^k \subset V(G)$ is a \emph{path} if, for every $1 \leq i \leq k$, one of $(u_{i-1},u_i)$, $(u_{i},u_{i-1})$ belongs to $E(G)$ (the path is \emph{directed} if always $(u_{i-1},u_i) \in E(G)$). A metric $\sd$ on $V(G)$ is \emph{geodesic} if for any two vertices $x,y \in V(G)$, there exists a path $\{u_i\}_{i=0}^k \subset V(G)$ such that $u_0 = x$, $u_k = y$, and $\sd(x,y) = \sum_{i=1}^k \sd(u_{i-1},u_i)$. 

\begin{rema}
A geodesic metric $\sd$ may be equivalently defined as the shortest path metric with respect to the edge-weights $(\sd(e))_{e\in E(G)}$. Here and in the sequel, we write $\sd(e)$ for $\sd(e^-,e^+)$. 
\end{rema}

When $S$ is a finite set (typically $V(G)$ or $E(G)$), we say that $\nu$ is \emph{a measure on $S$} if $\nu$ is a measure on the measurable space $(S,2^S)$; the domain of $\nu$ is thus the entire power set of $S$. We first define the isoperimetric dimension in the rather general context of graphs equipped with a geodesic metric and probability measures on its edge and vertex sets.

\begin{defi}[Isoperimetric dimension]\label{def:iso-dim}
Let $G$ be a graph, $\delta \in [1,\infty)$, $C_{iso}\in(0, \infty)$, $\mu$ a probability measure on $V(G)$, $\nu$ a probability measure on $E(G)$, and $\sd$ a geodesic metric on $V(G)$. We say that $G$ has $(\mu,\nu, \sd)$-\emph{isoperimetric dimension} $\delta$ with constant $C_{iso}$ if for every $A \subset V(G)$
	\begin{equation*}
		\min\{\mu(A), \mu(A^c)\}^{\frac{\delta-1}{\delta}} \leq C_{iso} \Per_{\nu, \sd}(A),
	\end{equation*}
where $\partial_G A := \{ (x,y)\in E(G) \colon \abs{ \{x, y\} \cap A} =1 \} $ is the edge-boundary of $A$, and the $(\nu, \sd)$-perimeter of $A$ is: 
	\begin{equation*}
		\Per_{\nu, \sd}(A) := \sum_{e \in \partial_G A } \frac{ \nu(e) }{\sd(e)}.
	\end{equation*}
\end{defi}
 
To the best of our knowledge, the second dimensional parameter we define is new. It is inspired by the classical notion of spectral dimension derived from the spectrum of a Laplace operator. We formally introduce the notion of Lipschitz growth function as a nonlinear analogue of the eigenvalue counting function.

\begin{defi}[Lipschitz growth function]
Let $\metXd$ be a metric space. The \emph{Lipschitz growth function} of a family of Lipschitz functions $F = \{f_i\colon \metX \to \bR\}_{i\in I}$ is the function \\ $\gamma_F \colon [0,\infty) \to \bN \cup \{\infty\}$ defined by $\gamma_F(s) = |\{i \in I \colon \Lip(f_i) \leq s\}|$.
\end{defi}

If one can define a Laplace operator $\Delta$ on $\metX$ and if $\{f_i\}_{i\in I}$ is an orthonormal basis of $L_2(\metX,\mu)$, for some probability measure $\mu$ on $\metX$, consisting of eigenfunctions of $\Delta$, then the Lipschitz growth function $\gamma$ coincides with the eigenvalue counting function\footnote{The classical eigenvalue counting function usually counts the eigenvalues of $\Delta$.}, i.e. $N(\lambda):= |\{i \in I\colon \lambda_i \leq \lambda\}| = \gamma(\lambda)$ where $\lambda_i$ is the eigenvalue of $\sqrt{\Delta}$ corresponding to $f_i$.

\begin{defi}[Lipschitz-spectral profile] \label{def:Lip-spec-profile}
Let $C_{1}, C_{\infty}, C_{\gamma}\in(0,\infty)$, $\delta \in [1,\infty)$, and $\beta \in[1,\infty)$. For $G$ a graph, $\mu$ a probability measure on $V(G)$, and $\sd$ a metric on $V(G)$, we say that $G$ has $(\mu, \sd)$-\emph{Lipschitz-spectral profile of dimension $\delta$ and bandwidth $\beta$ with constants $C_{1}, C_{\infty}, C_{\gamma}$} if there exists a collection of functions $F = \{f_i\colon V(G) \to \bR\}_{i\in I}$ satisfying:
	\begin{enumerate}
		\item\label{item:1} $C_{1}^{-1} \le \inf_{i\in I} \norm{f_i}_{L_1(\mu)} \le \sup_{i\in I} \norm{f_i}_{L_\infty(\mu)} \le C_{\infty},$ 
		\item $\{f_i\}_{i\in I}$ is an orthogonal family in $L_2(\mu)$, and
		\item\label{item:3} for every $s \in [1, \beta]$, $\gamma_F(s) \geq C_{\gamma}^{-1} s^\delta$.
	\end{enumerate}
\end{defi}

Our terminology \emph{Lipschitz-spectral dimension} is motivated by the fact that in the special situation mentioned above the estimate $N(\lambda) \gtrsim \lambda^\delta$ says that $(\metX, \mu, \Delta)$ has spectral dimension at least $\delta$. This important concept in spectral geometry  (see \cite{Chavel84} or \cite{Canzani} and the references therein) and in the field of analysis on fractals \cite[Chapter 4]{Kigami01}, originates from the classical Weyl law \cite{Weyl12} (see also \cite[Chapter 1]{Ulm_book}).

\begin{theoalpha} \label{thm:generalWasserstein}
Let $G$ be a graph equipped with geodesic metric $\sd$ on $V(G)$. If there exist probability measures $\mu$ and $\nu$ (on $V(G)$ and $E(G)$ respectively), numbers $\delta \in [2,\infty), \beta \in (0,\infty)$, and constants $C_{iso},C_{1},C_{\infty},C_{\gamma}\in(0,\infty)$ such that $G$ has $(\mu,\nu,\sd)$-isoperimetric dimension $\delta$ with constant $C_{iso}$ and $(\mu,\sd)$-Lipschitz-spectral profile of dimension $\delta$ and bandwidth $\beta$ with constants $C_{1},C_{\infty},C_{\gamma}$, then
    \begin{equation*}
        \cdist{L_1}(\Wa_1(G)) \geq \frac{ 1 }{2C_{iso}C_{1}^2C_{\infty}} \left( \frac{\delta}{C_{\gamma}} \right)^{\frac{1}{\delta}} (\ln \beta)^{\frac{1}{\delta}}.
    \end{equation*}
\end{theoalpha}

\begin{rema}
Note that in Theorem~\ref{thm:generalWasserstein}, it must hold that the dimension $\delta$ is at least 2. For graphs $G$ whose dimensions are strictly between 1 and 2, like the Laakso graphs $\La_1^{\os n}$ of Figure~\ref{fig:Laakso}, we do not know how to prove nontrivial lower bounds for $\cdist{L_1}(\Wa_1(G))$.
\end{rema}

Theorem~\ref{thm:diamondWasserstein} follows immediately from Theorem \ref{thm:generalWasserstein}, the observation that $\log|\dia_{k,k}^{\os n}|= \Theta_k(n)$, and the following theorem.

\begin{theoalpha}[Isoperimetric and Lipschitz-spectral dimensions of generalized diamond graphs] \label{thm:diamonddimensions}
Fix $k,m \in \bN$, and let $\sd$ be the shortest path metric on $\sD_{k,m}^{\os n}$, $\mu$ the degree-probability measure on $V(\sD_{k,m}^{\os n})$, and $\nu$ the uniform probability measure on $E(\sD_{k,m}^{\os n})$. Then $\sD_{k,m}^{\os n}$ has $(\mu,\nu,\sd)$-isoperimetric dimension $1 + \frac{\log m}{\log k}$ with constant $C_{iso} \leq \frac{m}{2}$ and $(\mu,\sd)$-Lipschitz spectral profile of dimension $1 + \frac{\log m}{\log k}$ and bandwidth $k^n$ with constants $C_{1} \leq 6$, $C_{\infty} \leq 1$, and $C_{\gamma} \leq 2k^2m^2$.
\end{theoalpha}

\noindent We refer to Corollary \ref{ex:isop-diamonds} and Corollary \ref{ex:spec-diamonds} for the proof of Theorem~\ref{thm:diamonddimensions}.

Our proof of Theorem~\ref{thm:generalWasserstein} follows the same outline as Naor-Schechtman's proof of $\cdist{L_1}(\Wa_1([n]^2)) = \Omega(\sqrt{\log n})$. The first step is to make the following reduction to linear maps: for $\banX$ a finite-dimensional Banach space, define $\cdist{L_1}^{lin}(\banX) := \inf_T \|T\| \cdot \|T^{-1}\|$, where the infimum is over all $N \in \bN$ and \emph{linear} injections $T: \banX \to \ell_1^N$. By \cite[Lemma 3.1]{NS07} (which is only stated for planar grids, but the proof obviously works for any finite metric space) we have, for any finite metric space $\metXd$,
	\begin{equation} \label{eq:Wa->LF}
    		\cdist{L_1}(\Wa_1(\metX)) = \cdist{L_1}^{lin}(\LF(\metX)),
	\end{equation}
where $\LF(\metX)$ is the \emph{Lipschitz-free space} over $\metX$; in our setting it is the Banach dual to the space $\Lip_0(\metX)$ of real-valued Lipschitz functions on $\metX$ vanishing at a fixed basepoint $x_0 \in \metX$. From there, Naor and Schechtman use a discrete version of an argument by Kislyakov \cite{Kislyakov75} to prove the necessary distortion estimates for an arbitrary linear $T\colon \LF([n]^2) \to \ell_1^N$. In the present work, we identify the precise geometric data of $G$ needed to run Kislyakov's argument, and we are naturally led to isolate the isoperimetric and Lipschitz-spectral dimensions as the key ingredients.

In Section \ref{sec:iso->Sob}, we review Sobolev spaces and prove a general Sobolev inequality on ``measured metric graphs'' with a given isoperimetric dimension (Theorem~\ref{thm:iso->Sob}). The proof technique we use is no different than well-known existing ones (see \cite{FedererFleming60} and the thorough exposition from \cite{BobkovHoudre97} in the smooth setting, or \cite{Chung97}, \cite{CGY00}, and \cite{Ostrovskii05} for the discrete setting) but we include it nonetheless because the general inequality we require does not seem to appear in the literature.

In Section \ref{sec:Kislyakov}, we prove our adaptation of Kislyakov's argument, namely Theorem~\ref{thm:Kislyakov}. Theorem~\ref{thm:generalWasserstein} follows immediately from \eqref{eq:Wa->LF} and Theorem~\ref{thm:Kislyakov}. An important part of the argument is that 1-summing maps from $\ell_\infty^N$ spaces to Banach lattices are order-bounded. In the original proof \cite{Kislyakov75} as well as the discretized one \cite{NS07}, this fact is proved using the Pietsch factorization theorem. In Lemma \ref{lem:orderbndd}, we provide a short, self-contained proof.

In the final sections, we investigate the behavior of isoperimetric and Lipschitz-spectral dimensions under $\os$-products, and we obtain exact computations in the case of $\os$-powers of certain graphs. In Section \ref{sec:slash-review}, we review $\os$-products of graphs and corresponding operations on measures, metrics, and functions. In Section \ref{sec:isoperimetric-inequalities}, we prove general results on isoperimetric inequalities of $\os$-products (Theorem~\ref{thm:slash-isop}) and $\os$-powers (Corollary~\ref{cor:slash-isop}), and in Section \ref{sec:Lip-spec}, we prove a general theorem on the Lipschitz-spectral profiles of $\os$-powers (Corollary~\ref{cor:spec-slashpower}). Theorem~\ref{thm:diamonddimensions} follows from Examples~\ref{ex:isop-diamonds} and \ref{ex:spec-diamonds} of these sections.

\section{Sobolev and isoperimetric inequalities}
\label{sec:iso->Sob}

In this section we recall the definitions of the Sobolev spaces on graphs that will be used in the subsequent sections.

Given a graph $G$ and a geodesic metric $\sd$ on $V(G)$, one can define a linear operator $\nabla_\sd$, which for any function $f \colon V(G) \to \bR$, returns its ``$\sd$-derivative'' as the function $\nabla_\sd f \colon E(G) \to \bR$ defined by 

\begin{equation*}
\nabla_\sd f(e) \eqd \frac{f(e^+) - f(e^-)}{\sd(e)}.
\end{equation*}

The following lemma, which says that the operator $\nabla_\sd$ commutes with integration, will come in handy when the time comes to prove the coarea formula. 

\begin{lemm}
Let $\{f_t \colon V(G) \to [0,\infty)\}_{t \in [0,\infty)}$ be a collection of functions. If for all $x\in V(G)$, the map $t \mapsto f_t(x)$ is integrable, then for all $e\in E(G)$, the map $t \mapsto \nabla_\sd (f_t)(e)$ is integrable and
\begin{equation}\label{eq:commute}
\nabla_\sd F(e) = \int_{0}^\infty \nabla_\sd f_t(e)dt,
\end{equation}
where $F(x) = \int_{0}^\infty f_t(x) dt$.
\end{lemm}

\begin{proof}
The integrability of $t \mapsto \nabla_\sd (f_t)(e)$ follows immediately from the integrability of $t \mapsto f_t(x)$. For all $e \in E(G)$, we have
	\begin{align*}
		\nabla_\sd F(e) = \frac{F(e^+) - F(e^-)}{\sd(e)} & = \frac{1}{\sd(e)} \left( \int_{0}^\infty f_t(e^+) dt - \int_{0}^\infty f_t(e^-) dt \right)\\
			& =  \int_{0}^\infty \frac{ f_t(e^+) - f_t(e^-) }{\sd(e)} dt 
			 = \int_{0}^\infty \nabla_\sd f_t(e)dt.
	\end{align*}
\end{proof}

If $G$ is a graph equipped with a probability measure $\nu$ on $E(G)$, then given a function $f\colon (V(G), \sd) \to \bR$ and $p\in[1,\infty]$, we define the $(1,p)$-Sobolev semi-norm (with respect to $\nu$ and $\sd$) of $f$ by 
	\begin{align*}
		\norm{f}_{W^{1,p}(\nu, \sd)}& \eqd \norm{ \nabla_\sd f}_{L_p(\nu)} = \Ex_{\nu} [\abs{\nabla_\sd f}^p]^{1/p} \\
		& = \left[ \int_{E(G)} \abs{\nabla_\sd f(e)}^p d\nu(e)\right]^{1/p}
	  =  \left[ \sum_{e\in E(G)} \frac{ \abs{ f(e^+) - f(e^-) }^p }{ \sd(e)^p } \nu(e)\right]^{1/p},
	\end{align*}
with the usual convention when $p=\infty$. By the geodesicity assumption, it holds that $\norm{f}_{W^{1,\infty}(\nu, \sd)} \leq \Lip(f)$, with equality if and only if $\nu$ is fully supported.
Note that the Sobolev norms do not depend on the orientation chosen to unambiguously define the derivative.

The following simple additivity property of the $(1,1)$-Sobolev norm will be useful in the ensuing arguments.

\begin{lemm}[Additivity of the $(1,1)$-Sobolev semi-norm] \label{lem:disjointsum}
Let $G$ be a graph equipped with a probability measure $\nu$ on $E(G)$ and a geodesic metric $\sd$ on $V(G)$. If for any $f \colon V(G) \to \bR$, we let $f_+ := \max\{0, f\}$ and $f_- := -\min\{0, f\}$, then 
\begin{equation*}
\norm{ f }_{W^{1,1}(\nu, \sd)} = \norm{f_+}_{W^{1,1}(\nu, \sd)} + \norm{f_-}_{W^{1,1}(\nu, \sd)}.
\end{equation*}
\end{lemm}

\begin{proof}
Let $f \colon V(G) \to \bR$. We need to consider 4 sets of edges:
\begin{itemize}
\item $P = \{e \in E: f(e^-), f(e^+) \geq 0\}$ and  $N = \{e \in E: f(e^-), f(e^+) \leq 0\}$, 
\item $M_1 = \{e \in E: f(e^-) < 0 < f(e^+)\}$ and $M_2 = \{e \in E: f(e^+) < 0 < f(e^-)\}$.
\end{itemize}

We clearly have that $\nabla_\sd f, \nabla_\sd (f_+), \nabla_\sd (f_-)$ vanish on $P \cap N$ and that all other pairwise intersections are empty. Hence, for each $g \in \{f, f_+, f_-\}$,
\begin{equation}\label{eq:Sob_split}
    \norm{g}_{W^{1,1}(\nu, \sd)} = \norm{ \nabla_\sd (g) }_{L_1(\nu)} = \norm{ \nabla_\sd (g ) \1_P }_{L_1(\nu)} + \norm{ \nabla_\sd (g) \1_N} _{L_1(\nu)}  + \norm{ \nabla_\sd (g) \1_{M_1} }_{L_1(\nu)}  + \norm{ \nabla_\sd (g) \1_{M_2} }_{L_1(\nu)}.
\end{equation}
Furthermore, it also clearly holds that:
\begin{enumerate}
\item[$(i_1)$] $\abs{\nabla_\sd (f) \1_P} = \abs{ \nabla_\sd (f_+) \1_P }$ and $\abs{\nabla_\sd (f) \1_N} = \abs{ \nabla_\sd (f_-) \1_N }$,
\item[$(i_2)$] $\abs{\nabla_\sd (f) \1_{M_i}} = \abs{ \nabla_\sd (f_+) \1_{M_i} } + \abs{ \nabla_\sd (f_-) \1_{M_i} }$, for $i \in \{1,2\}$,
\item[$(i_3)$] $\abs{ \nabla_\sd (f_+) \1_{N} } = 0$  and $\abs{ \nabla_\sd (f_-) \1_{P} } = 0$.
\end{enumerate}
Combining everything yields
\begin{align*}
    \norm{ f }_{W^{1,1}(\nu, \sd)} & \stackrel{\eqref{eq:Sob_split}}{=} \norm{ \nabla_\sd (f) \1_P }_{L_1(\nu)} + \norm{ \nabla_\sd (f) \1_N }_{L_1(\nu)} + \norm{ \nabla_\sd (f) \1_{M_1}}_{L_1(\nu)} + \norm{ \nabla_\sd (f) \1_{M_2}}_{L_1(\nu)} \\
    & \hskip -.1in \stackrel{(i_1)\land (i_2)}{=} \norm{ \nabla_\sd (f_+) \1_P}_{L_1(\nu)} + \norm{ \nabla_\sd (f_-) \1_N}_{L_1(\nu)} \\
    &  \hskip .1in + \norm{ \nabla_\sd (f_+) \1_{M_1} }_{L_1(\nu)} + \norm{ \nabla_\sd (f_-) \1_{M_1} }_{L_1(\nu)} + \norm{ \nabla_\sd (f_+) \1_{M_2} }_{L_1(\nu)} + \norm{ \nabla_\sd (f_-) \1_{M_2} }_{L_1(\nu)}\\
    & \hskip -.1in \stackrel{\eqref{eq:Sob_split} \land (i_3)}{=}  \norm{ f_+}_{W^{1,1}(\nu, \sd)} + \norm{ f_- }_{W^{1,1}(\nu, \sd)}.
\end{align*}
\end{proof}

The equivalence between isoperimetric and Sobolev inequalities is well-known, and the following theorem, which will be used in a crucial way in the sequel, is not new. However, because we could not locate a statement with this degree of generality, we give its elementary proof for the convenience of the reader.

\begin{theo}[Sobolev inequality from isoperimetric inequality] \label{thm:iso->Sob}
Let $G$ be a graph, $\mu$ a probability measure on $V(G)$, $\nu$ a probability measure on $E(G)$, and $\sd$ a geodesic metric on $V(G)$. If $G$ has $(\mu,\nu, \sd)$-isoperimetric dimension $\delta$ with constant $C$, then for every map $f \colon (V(G), \sd) \to  \bR$,
	\begin{equation*}
		\norm{ f- \Ex_\mu f }_{L_{\delta'}(\mu)} \leq 2C \norm{ f }_{W^{1,1}(\nu, \sd)},
	\end{equation*}
where $\Ex_\mu f = \int_{V(G)} f(x)d\mu(x)$, and $\delta'$ is the H\"older conjugate exponent of $\delta$, i.e. $\frac{1}{\delta} +\frac{1}{\delta'} =1$.
\end{theo}

The proof of Theorem \ref{thm:iso->Sob} relies on two classical but extremely useful lemmas. The first lemma is sometimes called the layer-cake representation lemma.

\begin{lemm}[Layer-cake representation]
Let $X$ be any set and $f \colon X \to [0,\infty)$ be any function. Then,
	\begin{equation}\label{eq:layer-cake}
		f = \int_{0}^\infty 1_{\{f > t\}}dt.
	\end{equation}
\end{lemm}

\begin{proof}
For $x\in X$ and $t\in[0,\infty)$, simply observe that $\1_{\{f>t\}}(x) = \1_{[0,f(x))}(t)$. Therefore, for every $x\in X$,  $t \mapsto \1_{ \{ f>t \} }(x)$ is measurable, and hence
	\begin{equation*}
		\int_{0}^\infty \1_{\{f > t\}}(x) dt = \int_{0}^\infty \1_{[0,f(x))}(t) dt = f(x).
	\end{equation*}
\end{proof}

The second lemma, known as the coarea formula (originally due to Federer \cite{Federer59}) has been established in various settings (cf. \cite{CGY00}, \cite{Ostrovskii05}). Note that if the metric $\sd$ assigns constant diameter $\sd_0$ to all the edges, then the formula reduces to the classical equality 
	\begin{equation*}
		\int_{E(G)} \abs{\nabla_\sd f(e)} d\nu(e)  = \sd_0^{-1}\int_0^\infty \nu(\partial_G \{f > t\} ) dt .
	\end{equation*}

\begin{lemm}[Coarea formula]
Let $G$ be a graph, $\mu$ a probability measure on $V(G)$, $\nu$ a probability measure on $E(G)$, and $\sd$ a geodesic metric on $V(G)$. Let $f\colon V(G) \to [0,\infty)$ be a function. Then
	\begin{equation*}
		\norm{ f }_{W^{1,1}(\nu, \sd)}  = \int_0^\infty \Per_{\nu, \sd}( \{f > t\} ) dt .
	\end{equation*}
\end{lemm}

\begin{proof}
Given $f\colon V(G)\to [0,\infty)$, we compute
	\begin{align*}
    		\norm{ f }_{W^{1,1}(\nu, \sd)} &= \norm{ \nabla_\sd f }_{L_1(\nu)} \\
    		& \stackrel{\eqref{eq:layer-cake}}{=} \bnorm{ \nabla_\sd \left( \int_0^\infty \1_{ \{f > t\} } dt \right) }_{L_1(\nu)} \\
    		& \stackrel{\eqref{eq:commute}}{=} \bnorm{  \int_0^\infty \nabla_\sd \1_{ \{f > t\} } dt }_{L_1(\nu)} \\
    		&= \sum_{e \in E(G)} \nu(e) \left| \int_0^\infty \nabla_\sd \1_{\{f > t\}}(e)dt \right|.
	\end{align*}
Assuming the following claim:
\begin{clai}
	\begin{equation}\label{eq:coarea}
		\left|\int_0^\infty \nabla_\sd \1_{\{f > t\}}(e)dt \right| = \int_0^\infty \left| \nabla_\sd \1_{\{f > t\}}(e)\right| dt,
	\end{equation}
\end{clai}

we can conclude the proof as follows:
	\begin{align*}
		\sum_{e \in E(G)} \nu(e) \left|\int_0^\infty \nabla_\sd \1_{\{f > t\}}(e)dt \right| &= \sum_{e \in E(G)} \nu(e)  \int_0^\infty \left| \nabla_\sd \1_{\{f > t\}}(e)\right|dt \\
		&= \int_0^\infty \sum_{e \in E(G)} \nu(e)  \left| \nabla_\sd \1_{\{f > t\}}(e)\right|dt \\
		& = \int_0^\infty \sum_{e \in E(G)} \nu(e)  \left| \frac{ \1_{\{f > t\}}(e^+) - \1_{\{f > t\}}(e^-) }{\sd(e)}\right|dt 
		 = \int_0^\infty \sum_{e \in \partial_G \{f > t\} } \frac{ \nu(e) }{\sd(e)} dt.
	\end{align*}

Hence, it remains to prove \eqref{eq:coarea} for each fixed $e \in E(G)$. This will obviously hold if $\nabla_\sd \1_{\{f > t\}}(e) \geq 0$ for a.e. $t \in [0,\infty)$ or if $\nabla_\sd \1_{\{f > t\}}(e) \leq 0$ for a.e. $t \in [0,\infty)$. Let $e \in E(G)$. First suppose $f(e^+) \geq f(e^-)$. Then $\nabla_\sd \1_{\{f > t\}}(e) = \frac{1}{\sd(e)}$ whenever $t \in (f(e^-), f(e^+))$, and $\nabla_\sd \1_{\{f > t\}}(e) = 0$ whenever $t \not\in [f(e^-), f(e^+)]$. This proves \eqref{eq:coarea} in this case. In the other case $f(e^+) \leq f(e^-)$, we have $\nabla_\sd \1_{\{f > t\}}(e) = \frac{-1}{\sd(e)}$ whenever $t \in (f(e^+), f(e^+))$, and $\nabla_\sd \1_{\{f > t\}}(e) = 0$ whenever $t \not\in [f(e^+), f(e^-)]$. Again this proves \eqref{eq:coarea}.
\end{proof}

We are now ready to prove Theorem \ref{thm:iso->Sob}.
\begin{proof}[Proof of Theorem \ref{thm:iso->Sob}]
Assume $G$ has $(\mu, \nu, \sd)$-isoperimetric dimension $\delta$ with constant $C<\infty$, and let $\delta'$ be the H\"older conjugate of $\delta$. First observe that, for any $c \in \bR$,
	\begin{align*}
		\norm{ f- \Ex_\mu f }_{L_{\delta'}(\mu)} & \leq \norm{ f- c}_{L_{\delta'}(\mu)} + \norm{ \Ex_\mu (c-f) }_{L_{\delta'}(\mu)}\\
		& = \norm{ f- c}_{L_{\delta'}(\mu)} + \abs{ \Ex_\mu (c - f) }\\
		& \le \norm{ f - c }_{L_{\delta'}(\mu)} +  \norm{ f - c }_{L_{1}(\mu)}\\
		& \le \norm{ f - c }_{L_{\delta'}(\mu)} +  \norm{ f - c }_{L_{\delta'}(\mu)}
		= 2\norm{ f - c }_{L_{\delta'}(\mu)}.
	\end{align*}
Therefore, it suffices to prove  
	\begin{equation*}
		\norm{ f - \med(f) }_{L_{\delta'}(\mu)} \leq C\norm{ f }_{W^{1,1}(\nu, \sd)},
	\end{equation*}
where $\med(f) \in \bR$ is a \emph{median} of $f$, i.e. any real number $m$ such that $\mu( \{ f > m \} ) \leq \frac{1}{2}$ and $\mu( \{ f < m \} ) \leq \frac{1}{2}$ (which always exists). Set $g := f - \med(f)$. Since $\norm{ g }_{W^{1,1}(\nu, \sd)} = \norm{ f }_{W^{1,1}(\nu, \sd)}$, it suffices to prove
	\begin{equation} \label{eq:gSobolev}
		\norm{ g }_{L_{\delta'}(\mu)} \leq C\norm{ g }_{W^{1,1}(\nu, \sd)}.
	\end{equation}
Note that $\med(g) = 0$. Let $g_+ := \max\{g,0\}$ and $g_- := -\min\{g,0\}$. Then by definition of $\med(g)$, we have
\begin{align*}
    \mu( \{ g_+ > 0 \}) = \mu( \{ g > 0 \}) = \mu( \{ g > \med(g) \} ) \leq \frac{1}{2}, \\
\nonumber \mu( \{ g_- > 0 \}) = \mu( \{ g < 0 \} ) = \mu( \{ g<\med(g) \}) \leq \frac{1}{2},
\end{align*}
and hence by definition of isoperimetric dimension we get
\begin{align*}
    \mu( \{ g_+ > t \})^{\frac{1}{\delta'}} \leq C \Per_{\nu, \sd}(  \{g_+ > t\} ), \\
\nonumber    \mu( \{ g_- > t \})^{\frac{1}{\delta'}} \leq C \Per_{\nu, \sd}(  \{ g_- > t\} ),
\end{align*}
for all $t \geq 0$.

Notice that the left-hand-sides of the above inequalities equal the $L_{\delta'}(\mu)$-norms of the indicator functions of the respective sets, and therefore
\begin{align} \label{eq:isoSob}
    \norm{ \1_{\{g_+ > t\}} }_{L_{\delta'}(\mu)} \leq C\Per_{\nu, \sd}(  \{g_+ > t\} ), \\
\nonumber    \norm{ \1_{\{g_- > t\}} }_{L_{\delta'}(\mu)}  \leq C\Per_{\nu, \sd}( \{g_- > t\} ).
\end{align}
Together with the fact that $g_+,g_-$ have disjoint supports and $g = g_+ - g_-$, we get
\begin{align*}
    \norm{ g }_{L_{\delta'}(\mu)}^{\delta'} &= \norm{ g_+ }_{L_{\delta'}(\mu)}^{\delta'} + \norm{ g_- }_{L_{\delta'}(\mu)}^{\delta'} \\
    &\overset{\eqref{eq:layer-cake}}{=} \bnorm{ \int_{0}^\infty \1_{\{g_+>t\}} dt }_{L_{\delta'}(\mu)}^{\delta'} + \bnorm{ \int_{0}^\infty \1_{\{g_->t\}} dt }_{L_{\delta'}(\mu)}^{\delta'} \\
    &\leq \left(\int_{0}^\infty \bnorm{ \1_{\{g_+>t\}} }_{L_{\delta'}(\mu)} dt\right)^{\delta'} +  \left(\int_{0}^\infty \bnorm{ \1_{\{g_->t\}} }_{L_{\delta'}(\mu)} dt\right)^{\delta'} \\
    &\overset{\eqref{eq:isoSob}}{\leq} \left(\int_{0}^\infty C \Per_{\nu, \sd}( \{g_+ > t\} ) dt\right)^{\delta'} +  \left(\int_{0}^\infty C \Per_{\nu, \sd}( \{g_- > t\} ) dt\right)^{\delta'} \\
    &\overset{\text{coarea}}{=} (C\norm{ g_+ }_{W^{1,1}(\nu, \sd)})^{\delta'} + (C\norm{ g_- }_{W^{1,1}(\nu, \sd)})^{\delta'} \\
    &\leq (C\norm{ g_+ }_{W^{1,1}(\nu, \sd)} + C\norm{ g_- }_{W^{1,1}(\nu, \sd)})^{\delta'} \\
    &\overset{\text{Lem }\ref{lem:disjointsum}}{=} (C\|g\|_{W^{1,1}(\nu, \sd})^{\delta'}.
\end{align*}
Taking the $\delta'$-root of each side proves \eqref{eq:gSobolev}.
\end{proof}

\section{A dimensionality interpretation of Kislyakov's argument} \label{sec:Kislyakov}
In this section, we delve into Naor-Schechtman's discretization of Kislyakov's argument. We pinpoint the crucial role of the two numerical parameters introduced in the introduction : the isoperimetric dimension (Definition \ref{def:iso-dim}) and the Lipschitz-spectral dimension (Definition \ref{def:Lip-spec-profile}). Fix a graph $G$ and a geodesic metric $\sd$ on $V(G)$. In the sequel, for $\mu$ a nonzero measure on $V(G)$ and $\nu$ a fully-supported measure on $E(G)$, we denote by $\Lip_{0, \mu}(V(G),\sd)$ the space of functions $f \colon V(G)\to \bR$ with $\bE_{\mu}[f] = 0$ equipped with the norm $\|f\|_{\Lip} := \|f\|_{W^{1,\infty}(\sd,\nu)}$. It is easily seen that the map $f \mapsto f - \Ex_\mu f$ is an onto isometric isomorphism between $\Lip_0(V(G),\sd)$ and $\Lip_{0,\mu}(V(G),\sd)$. Let $W_{0,\, \mu}^{1,1}(\sd,\nu)$ be the subspace of $W^{1,1}(\sd,\nu)$ consisting of those functions for which $\Ex_\mu f=0$. The map $f \mapsto f - \Ex_\mu f$ is also an onto isometric isomorphism between (the semi-normed space) $W^{1,1}(\sd,\nu)$ and (the normed space) $W_{0,\, \mu}^{1,1}(\sd,\nu)$.

Recall that a bounded linear map $R\colon \banX \to \banY$ between Banach spaces is \emph{1-summing} if there exists $C\in(0,\infty)$ such that
	\begin{equation}\label{eq:1-sum}
		\sum_{i=1}^N \norm{R(x_i)} \leq C \sup_{x^* \in B_{\banX^*}} \sum_{i=1}^N |\langle x^*,x_i \rangle|,
	\end{equation}
for every finite subset $\{x_i\}_{i=1}^N \subset \banX$. We denote the least such constant $C$ such that $\eqref{eq:1-sum}$ holds by $\pi_1(R)$. We begin with two basic facts concerning 1-summing maps. Their elementary proofs can be found in \cite[Chapter 3, Theorem 2.13]{DJT95}.

\goodbreak
\begin{lemm} \label{lem:1summingfacts} ~
\begin{enumerate}
    \item For any probability measure $\bP$, let $\iota_1$ be the formal identity from $L_\infty(\bP)$ to $L_1(\bP)$ and  $\banX$ be a subspace of $L_\infty(\bP)$, then $\iota_{1\restriction \banX}\colon \banX \to \iota_1(\banX)$ is 1-summing with $\pi_1(\iota_{1\restriction \banX}) = 1$.
    \item If $Q\colon \ban{W} \to \banX$, $R\colon \banX \to \banY$, and $S\colon \banY \to \ban{Z}$ are bounded linear maps between Banach spaces with $R$ 1-summing, then $S \circ R \circ Q$ is 1-summing with $\pi_1(S \circ R \circ Q) \leq \|S\|\pi_1(R)\|Q\|$.
\end{enumerate}
\end{lemm}

The next lemma already follows from \cite[Proof of Theorem 3]{Kislyakov75} (see also \cite[Lemmas 3.4, 3.5]{NS07}). We provide a shorter proof for the sake of completeness.

Recall that a Banach lattice is a Banach space $\banXn$ equipped with a partial order $\preceq$ satisfying, for all $\alpha \in [0,\infty)$ and $x,y,z \in \banX$,
\begin{itemize}
    \item $x \preceq y \implies x + z \preceq y + z$,
    \item $x \preceq y \implies \alpha x \preceq \alpha y$,
    \item there exists a supremum $x \lor y$ of $x,y$, and
    \item $\abs{x} \le \abs{y} \implies \norm{ x} \le \norm{y}$, where $\abs{x} = x \lor (-x)$.
\end{itemize}
The main examples of Banach lattices concerning us are the spaces $\ell_p(J)$, where $p \in [1,\infty]$, $J$ is some indexing set, and $a \preceq b$ if and only if $a_j \leq b_j$ for all $j \in J$.

\begin{lemm} \label{lem:orderbndd}
Let $N \in \bN$. For any Banach lattice $\banX$ and 1-summing linear map $R\colon \ell^N_\infty \to \banX$, there exists $x \in \banX$ with $\norm{x} \leq \pi_1(R)$ and $\abs{R(v)} \preceq x$ for every $v \in B_{\ell^N_\infty}$.
\end{lemm}

\begin{proof}
Let $\banX$ be a Banach lattice and $R\colon \ell^N_\infty \to \banX$ a 1-summing linear map. Define $x \in \banX$ by $x := \sum_{i=1}^N \abs{R(e_i)}$, where $\{e_i\}_{i=1}^N$ is the standard basis of $\ell^N_\infty$. Then we have
\begin{align*}
    \|x\| = \left\| \sum_{i=1}^N \abs{R(e_i)} \right\|
    \leq \sum_{i=1}^N \norm{R(e_i)}
    \leq \pi_1(R) \sup_{b \in B_{\ell_1^N}} \sum_{i=1}^N |\langle b,e_i \rangle|
    = \pi_1(R) \sup_{b \in B_{\ell_1^N}} \sum_{i=1}^N |b_i|
    = \pi_1(R),
\end{align*}
and for every $v \in B_{\ell^N_\infty}$,
\begin{align*}
    \abs{R(v)} = \left|R\left(\sum_{i=1}^N v_ie_i\right)\right|
    = \left|\sum_{i=1}^N v_i R\left(e_i\right)\right| 
    \preceq \sum_{i=1}^N|v_i|\left|R\left(e_i\right)\right| 
    \preceq \sum_{i=1}^N\left|R\left(e_i\right)\right|
    = x.
\end{align*}
\end{proof}

\begin{theo}\label{thm:Kislyakov}
Let $G$ be a graph, $C_{iso},C_{1},C_{\infty},C_{\gamma}$ constants in $(0, \infty)$, $\mu$ a probability measure on $V(G)$, $\nu$ a probability measure on $E(G)$, and $\sd$ a geodesic metric on $V(G)$. Let $\delta_{iso}\in [2, \infty)$ and $\delta_{spec}\in [1, \infty)$. If $G$ has $(\mu, \nu, \sd)$-isoperimetric dimension $\delta_{iso}$ with constant $C_{iso}$, and Lipschitz-spectral profile of dimension $\delta_{spec}$, bandwidth $\beta$, and constants $C_{1},C_{\infty},C_{\gamma}$, then any $D$-isomorphic embedding from the Lipschitz-free space $\mathsf{LF}(V(G),\sd)$ into a finite-dimensional $L_1$-space $\ell_1^N$ satisfies 
	\begin{equation*}
		D \ge \frac{ 1 }{2C_{iso}C_{1}^2C_{\infty}} \left( \frac{\delta_{iso}}{C_{\gamma}} \right)^{\frac{1}{\delta_{iso}}} \left( \int_{1}^\beta s^{ \delta_{spec} - \delta_{iso} - 1}  ds \right)^{\frac{1}{\delta_{iso}}}.
	\end{equation*}
\end{theo}

\begin{proof}
Assume that there exist $N \in \bN$ and a $D$-isomorphic embedding $T \colon \mathsf{LF}(V(G),\sd) \to \ell_1^N$. By scaling, we may assume that for all $x \in \mathsf{LF}(V(G),\sd)$, $\norm{x}_{\LF} \le \norm{Tx}_1 \le D \norm{x}_{\LF}$. The dual map $T^* \colon \ell_\infty^N \to \Lip_{0}(V(G),\sd)\equiv \Lip_{0,\mu}(V(G),\sd)$ is an onto linear map satisfying $\norm{ T^*}\le D$ and, importantly, 
	\begin{equation} \label{eq:T*quotient}
		T^*(B_{\ell^N_\infty}) \supset B_{\Lip_{0,\mu}(V(G),\sd)},
	\end{equation}
which follows from $\|x\|_{\LF} \le \|Tx\|_1$ and the Hahn-Banach theorem. Denote by $\iota_{sob} \colon W_{0,\, \mu}^{1,1}(\nu, \sd) \to L_{\delta_{iso}'}(\mu)$ the formal identity. It follows from Theorem \ref{thm:iso->Sob} and the condition $W^{1,1}_{0,\, \mu}(\nu,\sd) \subset \ker(\bE_\mu)$ that $\norm{\iota_{sob} }\le 2C_{iso}$. Let $\{f_j\}_{j \in J}$ be a collection of pairwise orthogonal functions realizing the $(\mu,\sd)$-Lipschitz-spectral profile of dimension $\delta$ and bandwidth $\beta$, with constants $C_{1},C_{\infty},C_{\gamma}$. We define a linear map $\cF \colon L_1(\mu) \to \bR^J$ by 
	\begin{equation*}
		\cF(g) \eqd (\bE_{\mu}[g f_j])_{j \in J}.
	\end{equation*}
Since $f_j \in L_\infty(\mu)$, for all $j\in J$, $\cF(g)$ is well-defined for all $g\in L_p(\mu)$ and $p\in [1,\infty]$. Moreover, since $\sup_{j \in J}\norm{ f_j }_{L_\infty(\mu)} \le C_{\infty}$, it follows that $\norm{ \cF}_{L_{1}(\mu)\to \ell_\infty(J)}\le C_{\infty}$.
By orthogonality of the collection $\{f_j\}_{j \in J}$ and because $\sup_{j \in J} \norm{ f_j }_{L_2(\mu)} \le \sup_{j \in J}\norm{ f_j }_{L_\infty(\mu)} \le C_{\infty}$, we have that $\norm{ \cF}_{L_{2}(\mu)\to \ell_2(J)}\le C_{\infty}$. Since $\delta_{iso}'\le 2$, the Riesz-Thorin interpolation\footnote{Riesz-Thorin interpolation theorem is valid for $\sigma$-finite measures and can be applied in our situation since $J$ is countable.} theorem tells us that $\cF \colon L_{\delta'_{iso}}(\mu) \to \ell_{\delta_{iso}}(J)$ is well-defined and $\norm{ \cF}_{L_{\delta_{iso}'}(\mu)\to \ell_{\delta_{iso}}(J)}\le C_{\infty}$. We thus have a chain of linear maps 

$$ \ell^N_\infty \stackrel{T^*}{\to} \Lip_{0,\mu}(V(G),\sd) \stackrel{\iota}{\to} W_{0,\, \mu}^{1,1}(\nu, \sd) \stackrel{\iota_{sob}}{\to} L_{\delta_{iso}'}(\mu) \stackrel{\cF}{\to}  \ell_{\delta_{iso}}(J),$$
where $\iota$ is the formal identity from $\Lip_{0,\mu}(V(G),\sd)$ into $W_{0,\, \mu}^{1,1}(\nu, \sd)$. Note that the gradient operator $\nabla_\sd$ defines a contractive linear map $\Lip_{0, \mu}(V(G),\sd) \to L_\infty(\nu)$ and a linear isometric embedding $W_{0,\, \mu}^{1,1}(\nu, \sd) \to L_1( \nu )$, and that we have the following commutative diagram:
\[
\begin{tikzcd}
 \Lip_{0, \mu}(V(G),\sd) \arrow{ rr }{ \iota } \arrow[swap]{d}{\nabla_\sd} & &  W_{0,\mu}^{1,1}(\nu, \sd)  \\
\hskip 1cm X \subset L_\infty(\nu) \arrow{rr}{\iota_{1} }  &  &  L_1(\nu) \supset Y \arrow[swap]{u}{\nabla_\sd^{-1}} \\
\end{tikzcd}
\]
Here, $X = \nabla_\sd( \Lip_{0,\mu}(V(G),\sd))$, $\iota_{1}$ is the formal identity, $Y = \iota_{1}(X)$, and $\iota = \nabla_\sd ^{-1}\circ \iota_{1\restriction_X}\circ \nabla_\sd$. Since $\nu$ is a probability measure, the above factorization and Lemma \ref{lem:1summingfacts} implies $\iota$ is 1-summing with $\pi_1(\iota) \leq 1$.	Similarly, by Lemma \ref{lem:1summingfacts} again,
\begin{equation*}
    \pi_1(\cF \circ \iota_{sob} \circ \iota \circ T^*) \leq \|\cF\| \cdot \|\iota_{sob}\| \cdot \|T^*\| \leq 2C_{iso}C_{\infty}D.
\end{equation*}
The above inequality together with Lemma \ref{lem:orderbndd} implies that there exists $b \in \ell_{\delta_{iso}}(J)$ with
	\begin{align}
		\|b\|_{\delta_{iso}} &\leq 2C_{iso}C_{\infty}D \label{eq:b1} \\
		\bigvee_{a \in B_{\ell_\infty^N}}|\cF \circ \iota_{sob} \circ \iota \circ T^*(a)| &\preceq b. \label{eq:b2}
	\end{align}

It follows from the definition of $\cF$ and from Definition \ref{def:Lip-spec-profile}-\eqref{item:1} that, for all $j\in J$,
\begin{equation} \label{eq:ej's}
    |\cF( f_j )| \succeq C_{1}^{-2} e_j,
\end{equation}
where $\{e_j\}_{j\in J}$ is the canonical basis of $\ell_{\delta_{iso}}(J)$. Therefore,
	\begin{equation*}
		\abs{ b } \overset{\eqref{eq:b2}}{\succeq} \bigvee_{a \in B_{\ell_\infty^N}}|\cF \circ \iota_{sob} \circ \iota \circ T^*(a)| \overset{\eqref{eq:T*quotient}}{\succeq} \bigvee_{j \in J } \frac{ \abs{ \cF( f_j )} }{ \Lip(f_j) } \overset{\eqref{eq:ej's}}{\succeq} \frac{1}{C_{1}^2} \bigvee_{j \in J } \frac{e_j }{ \Lip(f_j) } = \frac{1}{C_{1}^2} \sum_{j \in J } \frac{ e_j }{ \Lip(f_j) }.
	\end{equation*}
By taking the norm on both sides we get 
	\begin{equation*}
		\frac{1}{C_{1}^2} \left( \sum_{j \in J } \frac{ 1 }{ \Lip(f_j)^{\delta_{iso}} } \right)^{1/\delta_{iso}} \le \|b\|_{\delta_{iso}} \overset{\eqref{eq:b1}}{\leq} 2C_{iso}C_{\infty}D,
	\end{equation*}
and hence
	\begin{equation}\label{eq:aux1}
		D \geq \frac{1}{2C_{iso}C_{1}^2C_{\infty}} \left( \sum_{j \in J } \frac{ 1 }{ \Lip(f_j)^{\delta_{iso}} } \right)^{1/\delta_{iso}}. 
	\end{equation}
From here, we calculate the sum applying the classical formula
$$\int_{\Omega} \abs{ h}^p d\sigma = p \int_{0}^\infty t^{p-1} \sigma( \{ h >t\} ) dt$$
with $\Omega = J$ and $\sigma$ the counting measure:
	\begin{align}
	\nonumber	\sum_{j \in J} \frac{ 1 }{ \Lip(f_j)^{ \delta_{iso}  } } & = \delta_{iso} \int_{0}^\infty t^{ \delta_{iso} - 1}\bigabs{ \Big\{ j \in J \colon \frac{ 1 }{ \Lip(f_j) } > t \Big \} } dt \\ 
	\nonumber	& = \delta_{iso} \int_{0}^\infty  \frac{1}{ s^{ \delta_{iso} - 1}}\bigabs{ \Big\{ j \in J \colon \frac{ 1 }{ \Lip(f_j) } >  \frac{1}{s} \Big \} }  \frac{1}{s^2} ds \\
	\nonumber	& = \delta_{iso} \int_{0}^\infty  \frac{1}{ s^{ \delta_{iso} + 1}}\bigabs{ \Big\{ j \in J \colon  \Lip(f_j) < s \Big \} } ds \\
		\label{eq:aux2} & \stackrel{\eqref{item:3}}{\geq}  \delta_{iso} \int_{1}^\beta  \frac{1}{ s^{ \delta_{iso} + 1} } \frac{ s^{\delta_{spec}} }{ C_{\gamma} } ds.
	\end{align}
Combining \eqref{eq:aux1} and \eqref{eq:aux2} gives us
	\begin{equation*}
		D \geq \frac{ 1 }{2C_{iso}C_{1}^2C_{\infty}} \Big( \frac{\delta_{iso}}{C_{\gamma}} \Big)^{\frac{1}{\delta_{iso}}} \Big( \int_{1}^\beta s^{ \delta_{spec} - \delta_{iso} - 1}  ds \Big)^{\frac{1}{\delta_{iso}}}.
	\end{equation*}
\end{proof}

\section{Brief review of graph measures and $\os$-products}
\label{sec:slash-review}
The graphs we are interested in are graphs built by taking $\os$-product of various $s$-$t$ graphs. In the first subsection we define measures on the vertex set of general graphs induced by measures on their edge set, and in the following  subsection we recall basic properties of the $\os$-product operation relevant to the ensuing arguments.

\subsection{Edge-induced vertex measures}
Let $G$ be a graph and $\alpha=(\alpha(e))_{e\in E(G)} \subset (0,1)$. When $\nu$ is a measure on $E(G)$, we get an \emph{induced measure} $\mu_\alpha(\nu)$ on $V(G)$ defined for $x \in V(G)$ by
    \begin{equation} \label{eq:induced-measure}
		\mu_{\alpha}(\nu)(x) \eqd \sum_{e\in E(G) \atop e^+=x} \nu(e)\alpha(e) + \sum_{e\in E(G) \atop e^-=x} \nu(e) (1 - \alpha(e)).
	\end{equation}
It can be easily checked that $\mu_\alpha(\nu)$ is the unique measure on $V(G)$ satisfying
\begin{equation} \label{eq:induced-integral}
    \int_{V(G)} f d\mu_\alpha(\nu) = \int_{E(G)} \alpha(e)f(e^+) + (1-\alpha(e))f(e^-) d\nu(e)
\end{equation}
for all $f\colon V(G) \to \bR$.

\begin{rema}
Whenever $\nu$ is a probability measure, so is $\mu_\alpha(\nu)$. If $\alpha \equiv \tfrac{1}{2}$, we will often suppress notation and write $\mu(\nu)$ for $\mu_{\frac{1}{2}}(\nu)$. If $\nu$ is the uniform probability measure on $E(G)$, we call $\mu(\nu)$ the \emph{degree-probability measure} on $V(G)$ because, for all $x\in V(G)$, we have
\begin{equation*}
    \mu(\nu)(x) = \frac{ \deg(x) }{ 2\abs{E(G)} } = \frac{ \deg(x) }{ \sum_{y \in V(G)} \deg(y)}.
\end{equation*}
\end{rema}

\subsection{$\os$-products} 
\label{sec:slash-products}
In the sequel, an \emph{$s$-$t$ graph} will be a graph $G$ equipped with two distinguished and distinct vertices: a \emph{source} vertex $s(G)$ and a \emph{sink} or \emph{target} vertex $t(G)$, and an orientation of the edges such that every vertex in $V(G)$ belongs to a directed path from $s(G)$ to $t(G)$.

\begin{exam} \label{ex:Pk}
Let $k \ge 2$ be an integer. Let $\Pk$ denote the \emph{path graph} of length $k$ with the following concrete labelling: $V(\Pk) := \{ \frac{i}{k} \colon 0 \leq i \leq k\}$ and $E(\Pk) := \{(\tfrac{i-1}{k},\tfrac{i}{k}) \colon 1 \leq i \leq k\}$. The graph $\Pk$ has $k+1$ vertices and $k$ edges directed from the source $s(\Pk) := 0$ to the sink $t(\Pk) := 1$, thus turning $\Pk$ into an $s$-$t$ graph. The graph $\Pk$ is typically equipped with the normalized geodesic metric induced by the weights $\sd_{\Pk}(e) := \frac{1}{k}$ for every $e \in E(\Pk)$.
\end{exam}

The next example supplies the class of graphs to which our main theorems on dimensions of $\os$-powers apply.

\begin{exam}[Generalized diamond graphs] \label{ex:diamonds}
Let $k,m \geq 2$ be integers. The $m$-branching diamond graph of depth $k$, denoted $\dia_{k,m}$, is the $s$-$t$ graph with vertex set:
	$$V(\dia_{k,m}) := V(\Pk) \times \{1,\dots m\} / \sim,$$
where $(u,i) \sim (v,j)$ if and only if $(u,i) = (v,j)$, or $u = v = 0$, or $u = v = 1$,\\
and directed edge set:
	$$E(\dia_{k,m}) := \{([(e^-,i)],[(e^+,i)]): e \in E(\Pk), i \in \{1, \dots m\} \},$$
with source $s(\dia_{k,m}) := [(0,i)]$ and sink $t(\dia_{k,m}) := [(1,i)]$. 
\begin{figure}[H]
    \centering
   \includegraphics[trim=250 300 200 100,clip,scale=.425]{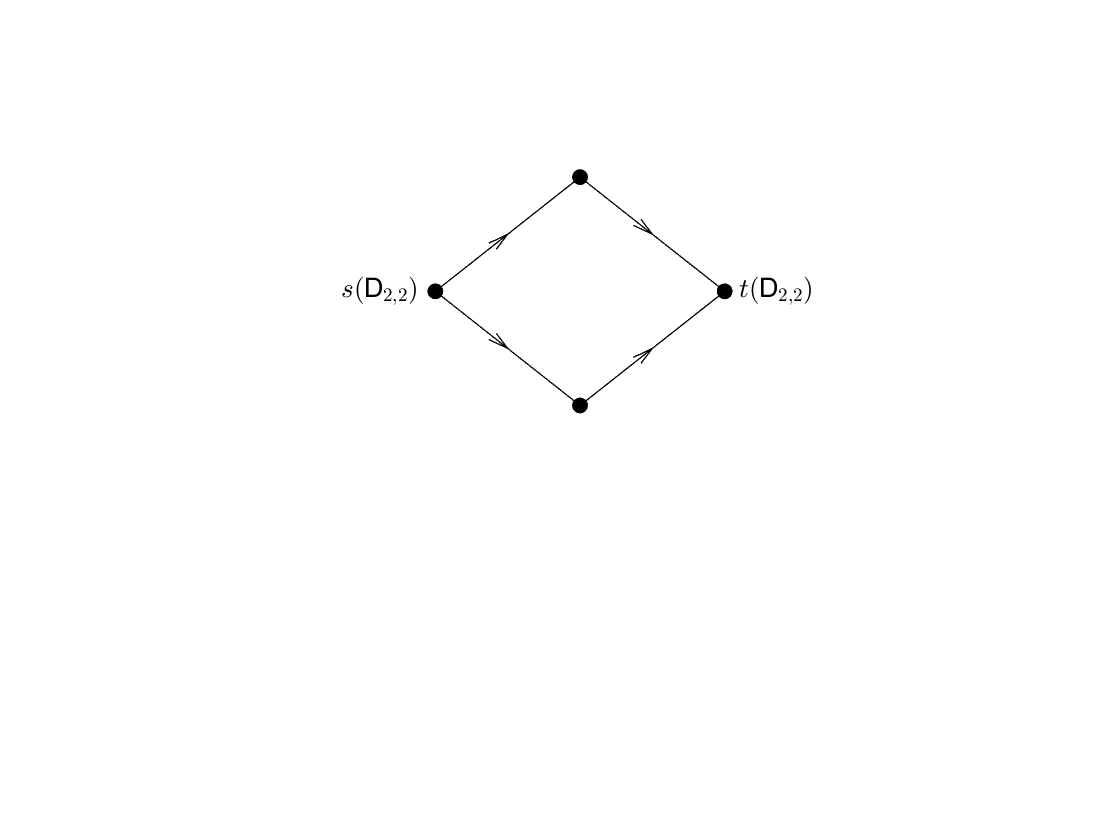}
   \includegraphics[trim=200 300 200 100,clip,scale=.425]{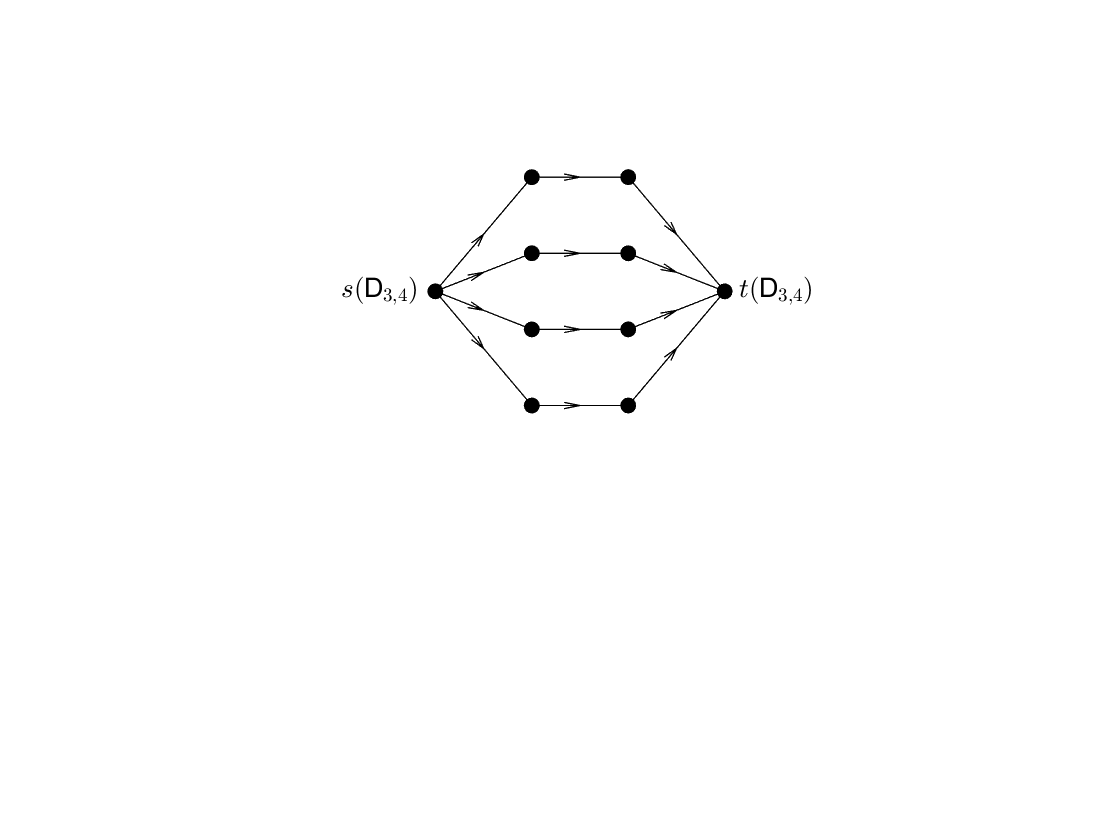}
       \caption{The diamond graphs $\sD_{2,2}$ and $\sD_{3,4}$.}
    \label{fig:diamonds}
\end{figure}
We typically equip $V(\sD_{k,m})$ with the normalized geodesic metric induced by the weights $\sd_{\sD_{k,m}}(e) := \frac{1}{k}$ and $E(\sD_{k,m})$ with the uniform probability measure $\nu_{\sD_{k,m}}(e) := \frac{1}{km}$ for every $e \in E(\sD_{k,m})$.
\end{exam}

It seems that the first formal definition of $\os$-product appeared in \cite{LeeRag10}.
\begin{defi}[$\os$-product]\label{def:slash-product}
Let $H$ be a graph and $G$ an $s$-$t$ graph. We define the graph \emph{$\os$-product} of $H$ by $G$, denoted $H \os G$, as follows.
\begin{itemize}
    \item The vertex set $V(H\os G)$ is defined to be $E(H) \times V(G)/\sim$, where $(e_1,u_1) \sim (e_2,u_2)$ if and only if
    \begin{itemize}
        \item $(e_1,u_1) = (e_2,u_2)$, or
        \item $e_1^+ = e_2^-$, $u_1 = t(G)$, and $u_2 = s(G)$, or
        \item $e_1^+ = e_2^+$, $u_1 = t(G)$, and $u_2 = t(G)$, or
        \item $e_1^- = e_2^-$, $u_1 = s(G)$, and $u_2 = s(G)$.        
    \end{itemize}
    For $(e,u) \in E(H) \times V(G)$, its equivalence class in $V(H \os G)$ is denoted by $e \os u$.
    \item The directed edge set $E(H \os G)$ is defined to be $\{(e\os f^-,e\os f^+): (e,f) \in E(H) \times E(G)\}$. We denote the edge $(e\os f^-,e\os f^+)$ by $e\os f$.
\end{itemize}
\end{defi}

\begin{rema}
The assignment $(e, f) \mapsto e\os f$ defines a bijection $E(H) \times E(G) \to E(H\os G)$. With our choice of notation, it obviously holds that $(e\os f)^\pm = e\os f^\pm$.
\end{rema}

It is routine to check that $H\os G$ satisfies our standing assumptions on graphs (finite, connected, directed, with at least one edge, and without self-loops or multiple edges between the same pair of vertices) since $H$ and $G$ do.

There is a canonical injection $V(H) \hookrightarrow V(H\os G)$ given by $e^+ \mapsto e\os t(G)$ and $e^- \mapsto e\os s(G)$ for every $e \in E(H)$. The domain of this map is all of $V(H)$ since every vertex is an endpoint of at least one edge, and it is well-defined by the definition of the equivalence relation $\sim$ defining $V(H\os G)$. We treat $V(H)$ as a subset of $V(H\os G)$ under this identification. If $H$ is an $s$-$t$ graph, then $H\os G$ inherits an $s$-$t$ structure under the choice $s(H\os G) := s(H)$, $t(H\os G) := t(H)$. 

\begin{figure}[H]
    \centering
   \includegraphics[trim=200 175 150 150,clip,scale=.425]{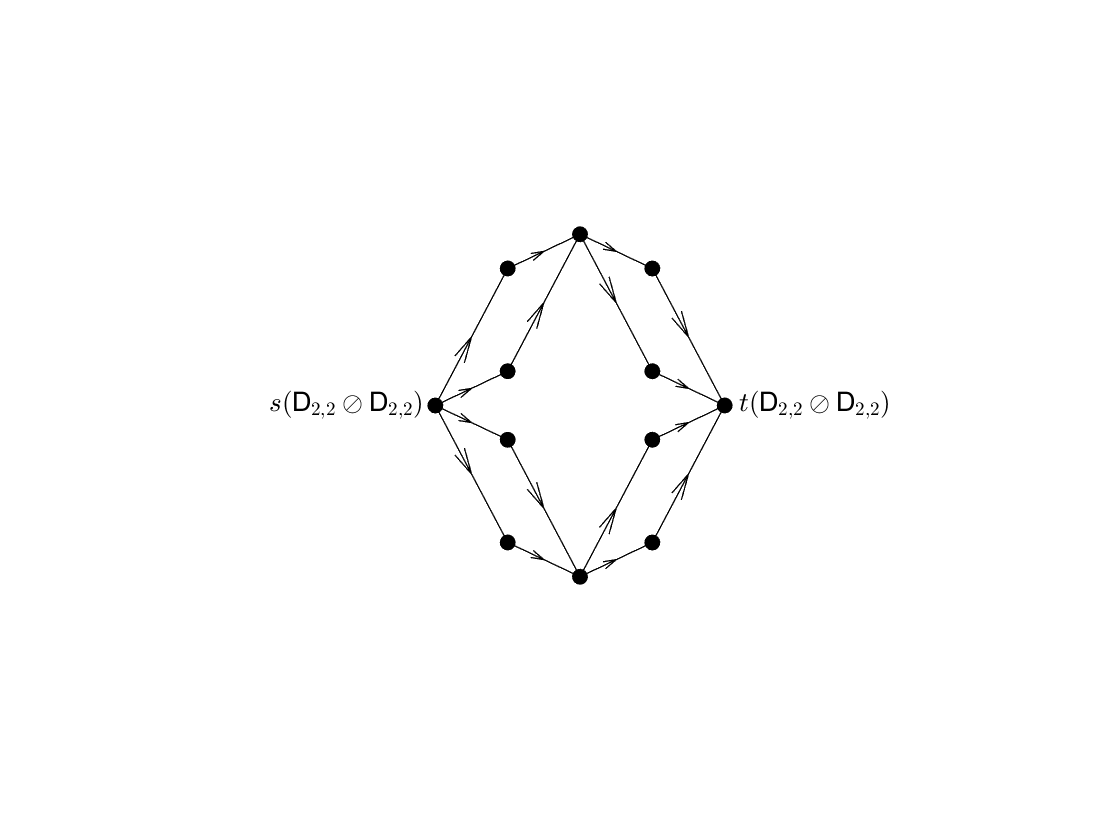}
       \caption{The $\os$-product $\sD_{2,2} \os \sD_{2,2} = \sD_{2,2}^{\os 2}$.}
    \label{fig:D22-slash-D22}
\end{figure}

Let $H$ and $H'$ be graphs. Recall that a \emph{graph morphism} is a map $\theta\colon V(H) \to V(H')$ that preserves directed edges, i.e. $(\theta(e^-),\theta(e^+)) \in E(H')$ for every $e \in E(H)$ (we adopt the convention that all graph morphisms are directed). In this case $\theta$ induces a well-defined map (still denoted $\theta$) from $E(H)$ to $E(H')$ satisfying $\theta(e)^\pm = \theta(e^\pm)$. Let $G$ and $G'$ be $s$-$t$ graphs and $\theta\colon V(G) \to V(G')$ a graph morphism. If $\theta(s(G)) = s(G')$ and $\theta(t(G)) = t(G')$, then $\theta$ is an \emph{$s$-$t$ graph morphism}. Let $\theta_H\colon V(H) \to V(H')$ be a graph morphism and $\theta_G\colon V(G) \to V(G')$ an $s$-$t$ graph morphism. We define the \emph{$\os$-morphism} $\theta_H\os\theta_G\colon V(H\os G) \to V(H'\os G')$ by
	\begin{equation*}
		(\theta_H\os\theta_G)(e\os u) := \theta_H(e)\os\theta_G(u).
	\end{equation*}
It can be easily verified that $\theta_H\os\theta_G$ is a well-defined graph morphism.

\subsection{$\os$-measures on $\os$-products}
\label{sec:os-meas}
Let $H$ be a graph and $G$ an $s$-$t$ graph. When $\nu_H$ and $\nu_G$ are measures on $E(H)$ and $E(G)$, respectively, we define the \emph{$\os$-measure} $\nu_H\os\nu_G$ on $E(H\os G)$ by
\begin{equation}\label{eq:prod-measure}
	(\nu_H\os \nu_G)(e\os f) \eqd \nu_H(e) \cdot \nu_G(f).
\end{equation}

\begin{rema}
Obviously, under the identification $E(H\os G) = E(H) \times E(G)$, the $\os$-measure is simply the product measure.
\end{rema}

Combining \eqref{eq:induced-measure} and\eqref{eq:prod-measure}, we obtain a simple identity below that will be used repeatedly in the sequel. For $e_0 \in E(H)$, we define the \emph{contractions} along $e_0$ of $S \subset V(H\os G)$ and $\alpha=(\alpha(e\os f))_{e\os f\in E(H\os G)}\subset (0,1)$ by $S_{e_0} \eqd \{x \in V(G): e_0\os x \in S\}$ and $\alpha_{e_0} \eqd (\alpha_{e_0}(f))_{f\in E(G)} \eqd (\alpha(e_0 \os f))_{f\in E(G)}$. Then, for all $S\subset V(H\os G)$ and $(\alpha(e\os f))_{e\os f\in E(H\os G)}\subset (0,1)$ we have
    \begin{equation}\label{eq:decomposition}
		\mu_\alpha(\nu_H\os \nu_G)(S)= \sum_{e\in E(H)} \nu_H(e)\mu_{\alpha_e}(\nu_G)(S_e).
	\end{equation}

Given measures $\nu_H$ and $\mu_G$ on $E(H)$ and $V(G)$, respectively, Riesz's representation theorem guarantees that there exists a unique \emph{$\os$-measure} $\nu_H\os\mu_G$ on $V(H\os G)$ satisfying
\begin{equation} \label{eq:slash-measure}
    \int_{V(H\os G)} f d(\nu_H\os\mu_G) = \int_{E(H)} \left(\int_{V(G)} f(e \os x) d\mu_G(x)\right)d\nu_H(e)
\end{equation}
for all $f\colon V(H\os G) \to \bR$.

Using \eqref{eq:decomposition} with $\alpha \equiv \frac{1}{2}$, we see that \eqref{eq:slash-measure} implies
	\begin{equation}\label{eq:decomposition2}
		\mu(\nu_H\os \nu_G)= \nu_H\os\mu(\nu_G)
	\end{equation}
whenever $\nu_H$ and $\nu_G$ are measures on $E(H)$ and $E(G)$, respectively.

\subsection{$\os$-metrics on $\os$-products}
\label{sec:os-met}
Let $H$ be a graph and $G$ an $s$-$t$ graph. Let $\sd_H$ and $\sd_G$ be geodesic metrics on $V(H)$ and $V(G)$, respectively. We define the \emph{$\os$-geodesic metric} $\sd_H\os\sd_G$ to be the unique geodesic metric on $V(H\os G)$ satisfying
	\begin{equation}\label{eq:prod-metric}
		(\sd_H\os\sd_G)(e\os f) = \sd_H(e)\cdot\sd_G(f),
	\end{equation}
for all $e\os f \in E(H\os G)$.

Observe that for any $u,v \in V(H) \subset V(H\os G)$, it holds that
	\begin{equation*}
		(\sd_H\os\sd_G)(u,v) = \sd_H(u,v)\cdot\sd_G(s(G),t(G)).
	\end{equation*}
Hence, if the geodesic metric on $G$ is normalized, i.e. $\sd_G(s(G),t(G)) = 1$, then the canonical inclusion of $(V(H),\sd_H)$ in $(V(H\os G),\sd_H\os\sd_G)$ is an isometric embedding. Note also that for any $e \in E(H)$ and $u,v \in V(G)$, it clearly holds that
\begin{equation*}
    (\sd_H\os\sd_G)(e\os u,e\os v) = \sd_H(e)\cdot\sd_G(u,v).
\end{equation*}

\section{Isoperimetric dimension of $\os$-products and $\os$-powers}
\label{sec:isoperimetric-inequalities}

The main goal of this section is to compute the isoperimetric dimension of $\os$-powers of graphs. This is accomplished with Theorem~\ref{theo:isoper-inequality-for-slashpowers}. To prove this theorem, we study the behavior of isoperimetric ratios under $\os$-products. In Definition \ref{def:iso-dim} and Section \ref{sec:iso->Sob} we considered measures on the edge and vertex sets that were independent of each other. In our study of the isoperimetric dimension of $\os$-products we require a certain compatibility condition between the two measures. In some sense the measure on the vertex set is governed by the measure on the edge set.

For $G$ a graph, a probability measure $\nu$ on $E(G)$, a geodesic metric $\sd$ on $V(G)$, $\delta \in [1, \infty)$, and $\alpha=(\alpha(e))_{e\in E(G)}\subset (0,1)$, we define the {\em isoperimetric ratio of $S\subset V(G)$} by
    \begin{equation}
		q_{\sd,\nu ,\alpha, \delta}(S) \eqd \frac {\Per_{\sd,\nu}(S)}{\min\{ \mu_{\alpha}(\nu)(S), \mu_{\alpha}(\nu)(S^c)\}^{\frac{\delta-1}\delta}  }.
	\end{equation}
Thus, $G$ has $(\mu_{\alpha}(\nu),\nu, \sd)$-isoperimetric dimension $\delta$ with constant $C$ if $q_{\sd,\nu ,\alpha, \delta}(S)\ge 1/C$ for all $S\subset V(G)$.

Since obviously $\partial(S) = \partial(S^c)$ and thus $\Per_{\sd,\nu}(S) = \Per_{\sd,\nu}(S^c)$, it is certainly true that $$q_{\sd,\nu ,\alpha, \delta}(S) = \max\{\tilde{q}_{\sd,\nu ,\alpha, \delta}(S), \tilde{q}_{\sd,\nu ,\alpha, \delta}(S^c)\}$$ where for all $\emptyset\neq S\subset V(G)$,
	\begin{equation}\label{eq:tilde-iso}
		\tilde{q}_{\sd,\nu ,\alpha, \delta}(S) \eqd \frac {\Per_{\sd,\nu}(S)}{\mu_{\alpha}(\nu)(S)^{\frac{\delta-1}\delta}  }.
	\end{equation}
Note here that for all $S\neq \emptyset$, $\mu_{\alpha}(\nu)(S)\neq 0$ since $\nu$ is fully supported and $\alpha(e)>0$ for all $e\in E(G)$.
We will conveniently referred to \eqref{eq:tilde-iso} as the $\sim$-isoperimetric ratio. First we prove a general lemma showing that, in order to lower bound isoperimetric ratios, it suffices to consider only connected subsets. Recall that a subset $S$ of $V(G)$ is \emph{connected} if any two vertices $x,y$ in $S$ can be connected by a path made of vertices in $S$. If $S \subset V(G)$, a \emph{connected component} of $S$ is a maximal connected subset of $S$.

\begin{prop}\label{prop:connected-components}
Let $G$ be a graph and $\nu$ a probability measure on $E(G)$. Let $\alpha=(\alpha(e))_{e\in E(G)}\subset (0,1)$ and $S\subset V(G)$ with $\mu_{\alpha}(\nu)( S)\le \mu_{\alpha}(\nu)(S^c)$ and  let $S_1, S_2, \ldots, S_n$ be its connected components, then 
	\begin{equation*}
		q_{\sd,\nu,\delta,\alpha}(S)\ge \min_{1\le j \le n} q_{\sd,\nu,\delta,\alpha}(S_j).
	\end{equation*}
\end{prop}
\begin{proof} Since 
the boundaries of
$S_j$, $j=1, 2, \ldots, n$, are pairwise disjoint,  it follows that  
$$q_{\sd,\nu,\delta,\alpha}(S)=\frac{\sum_{j=1}^n \Per_{\sd,\nu}(S_j)}{\Big(\sum_{j=1}^n \mu_{\alpha}(\nu)(S_j)\Big)^{\frac{\delta-1}\delta}}.$$
Thus the proposition follows from the following claim.
\begin{clai}\label{claim:connected-components}  Let $0<r\le 1$, $n\in \bN$, $a_1, \dots, a_n$ non-negative numbers, and $b_1, \dots, b_n$ positive numbers. Then 
	\begin{equation}\label{eq:connected-components}
		\frac{\sum_{j=1}^na_j}{\Big(\sum_{j=1}^n b_j\Big)^r}\ge \min_{1 \le j \le n} \frac{a_j}{b_j^r}.
	\end{equation}
\end{clai} 
\begin{proof}[Proof of Claim \ref{claim:connected-components}]
	\begin{align*}
		\min_{1 \le i \le n} \frac{a_i}{b_i^r} \Big(\sum_{j=1}^n b_j\Big)^r & =  \Big(\sum_{j=1}^n b_j \min_{1 \le i \le n} \frac{a_i^{1/r}}{b_i}  \Big)^r \le  \Big( \sum_{j=1}^n a_j^{1/r} \Big)^r \le  \sum_{j=1}^n a_j,
	\end{align*}
where in the last inequality we used that $r\in (0,1]$.
\end{proof}
\end{proof}

\subsection{Behavior of isoperimetric ratios under $\os$-products}
\label{sec:general-iso}
Throughout this subsection, fix an isoperimetric exponent $\delta \in [1,\infty)$, a graph $H$, an $s$-$t$ graph $G$, probability measures $\nu_H$ and $\nu_G$ on $E(H)$ and $E(G)$, respectively, and geodesic metrics $\sd_H$ and $\sd_G$, on $V(H)$ and $V(G)$ respectively. We assume that $\sd_G$ is \emph{normalized}, meaning $\sd_G(s(G),t(G)) = 1$.

First let us introduce some convenient and simplified notation. We will simply write $\Per_{H},~\Per_{G}$, and $\Per_{H\os G}$ for $\Per_{\nu_{H},\sd_{H}}$, $\Per_{\nu_{G},\sd_{G}}$, and $ \Per_{ \nu_H\os \nu_G, \sd_{H\os G }}$, respectively. Similarly, for $(\alpha(e))_{e\in E(H\os G)}$, $(\beta(e))_{ e\in E(H)}$, $(\gamma(e))_{e\in E(G)} \subset (0,1)$, we will omit references to the (fixed) metrics, measures, and isoperimetric exponent, and we will abbreviate the isoperimetric ratios $q_{\sd_H\os \sd_G,\nu_H\os \nu_G, \delta,\alpha}$, $q_{\sd_H,\nu_H, \delta,\beta}$, and $q_{\sd_G,\nu_G, \delta,\gamma}$, by $q_{H\os G,\alpha},~q_{H,\beta}$, and $q_{G,\gamma}$, respectively. We apply the same rules for the $\sim$-isoperimetric ratios. The induced measures $\mu_\alpha(\nu_H\os \nu_G)$, $\mu_{\beta}(\nu_H)$, and $\mu_{\gamma}(\nu_G)$, will be shorten to $\mu_{H\os G,\,\alpha}$, $\mu_{H,\,\beta}$, and $\mu_{G,\,\gamma}$, respectively. 

We start with a first intuitive lemma which says that the $\sim$-isoperimetric ratio of a nonempty subset $S$ of $H\os G$ contained entirely inside a copy of $G$ (and not containing the end vertices) is up to some natural scaling factors and appropriate weights the $\sim$-isoperimetric ratio of $S$ considered in $G$. For $e \in E(H)$ and $S \subset V(G)$, we define the \emph{lift} of $S$ in the $e$-th copy of $G$ by $e\os S := \{e\os x\colon x\in S\} \subset V(H\os G)$.

\begin{lemm} \label{lem:insideGformula}
For every $(\alpha(e))_{e \in E(H\os G)}\subset (0,1)$, $e_0 \in E(H)$, and $S \subset V(G) \setminus \{s(G),t(G)\}$ with $S\neq \emptyset$,
	\begin{equation*}
		\tilde{q}_{H\os G,\alpha}(e_0\os S) = \frac{\nu_H(e_0)^{\frac{1}{\delta}}}{\sd_H(e_0)} \tilde{q}_{G,\alpha_{e_0}}(S).
	\end{equation*}
\end{lemm}

\begin{proof} Since $S$ does not contain the endpoints we have $\partial_{H\oslash G} (e_0\os S)=e_0\os \partial_G(S)$, and thus 
\begin{align*} 
\Per_{H\oslash G}(e_0\os S)= \sum_{e\in \partial_G(S)} \frac{\nu_H(e_0)\nu_G(e)}{\sd_H(e_0)d_G(e)}=\frac{\nu_H(e_0)}{\sd_H(e_0)} \Per_G(S).
   \end{align*}
Equation \eqref{eq:decomposition} tells us that $\mu_{H\os G,\alpha}(e_0\os S)=\nu_H(e_0) \mu_{G,\alpha_{e_0}}(S)$,
which yields 
\begin{align*}  \tilde{q}_{H\os G,\alpha}(e_0\os S) = \frac{\frac{\nu_H(e_0)}{\sd_H(e_0)} \Per_G(S)}{[\nu_H(e_0)\mu_{G,\alpha_{e_0}}(S)]^{\frac{\delta-1}{\delta}}}
=\frac{\nu_H(e_0)^{\frac1\delta}}{\sd_H(e_0)}\tilde q_{G,\alpha_{e_0}}(S).
\end{align*} 
\end{proof}

The next lemma is our main technical observation for isoperimetric ratios of subsets containing both endpoints of at least an edge in $H$. We need one more piece of notation pertaining to lifts of edges of $G$. For $e \in E(H)$, we define the \emph{lift} of $F \subset E(G)$ in the $e$-th copy of $G$ by $e\os F := \{e\os f\colon f \in F\} \subset E(H\os G)$. 

\begin{lemm}\label{lem:dichotomy} 
Let $\alpha=(\alpha(e\os f))_{e\os f\in E(H\os G)}\subset (0,1)$
and $S \ksubset V( H\os G)$, with $\mu_{ {H\os G}, \alpha}(S)\le \frac12$. If there exists $e_0 \in E(H)$ such that $\{e_0^-, e_0^+\} \subset S$, then at least one of the following conditions (a) and (b) hold.
	\begin{enumerate}
		\item[(a)]$S \cup \big(e_0\os S_{e_0}^c\big)\not= V(H\os G)$ and $q_{H \os G, \alpha}(S) \ge q_{H \os G, \alpha}\big(S \cup( e_0 \os S_{e_0}^c)\big)$,
		\item[] or
		\item[(b)] $q_{H \os G, \alpha}(S) \ge \frac{ \nu_H(e_0)^{ \frac{1}{\delta} } }{ \sd_H(e_0) } \tilde{q}_{G,\alpha_{e_0}}(S_{e_0}^c).$
	\end{enumerate}
Note that in (a) the complement is taken in $V(G)$, i.e. $S^c_{e_0} := V(G)\setminus S_{e_0} = \{x\in V(G)\colon e_0 \os x \not\in S\}$.
\end{lemm}

\begin{proof}
Assume that there exists $e_0 \in E(H)$ such that $\{e_0^-,e_0^+\} \subset S$. We will prove that if (a) does not hold then (b) holds.

In the case that $S \cup \big(e_0\os S_{e_0}^c\big) = V(H\os G)$, and thus  $S^c= e_0\os S^c_{e_0}\subset e_0\os (V(G)\setminus\{s(G), t(G)\})$, it follows that
	\begin{align*}
		q_{H \os G, \alpha}(S) =q_{H \os G, \alpha}(S^c)=q_{H \os G, \alpha}(e_0\os S_{e_0}^c) \overset{\text{Lem }\ref{lem:insideGformula}}{\ge} \frac{ \nu_H(e_0)^{ \frac{1}{\delta} } }{ \sd_H(e_0) } \tilde{q}_{G, \alpha_{e_0}}(S_{e_0}^c),
	\end{align*}
yielding (b).

So we assume that $S \cup \big(e_0\os S_{e_0}^c\big) \neq V(H\os G)$ and $q_{H \os G, \alpha}(S) < q_{H \os G, \alpha}\big(S \cup( e_0 \os S_{e_0}^c)\big)$. Necessarily $S_{e_0}^c \neq \emptyset$. Letting $\tilde{S} := S \cup (e_0\os S_{e_0}^c)$, we can also assume without loss of generality that 
	\begin{equation}\label{eq:wlog}
		\mu_{H \os G, \alpha }(\tilde{S}) > \frac12\ge  \mu_{ H \os G, \alpha }((\tilde{S})^c)
	\end{equation}
since otherwise
	\begin{align*}
		  q_{ H \os G, \alpha }(\tilde S) & = \frac{ \Per_{ { H \os G } }(\tilde S) }{ \mu_{{ H \os G }, \alpha} (\tilde S)^{ \frac{\delta-1}{\delta} } } = \frac{  \Per_{ H \os G }(S) - \frac{ \nu_{H}(e_0) }{ \sd_H(e_0) } \Per_{ G }(S_{e_0}) }{ \mu_{ H \os G, \alpha}( \tilde S)^{ \frac{\delta-1}{\delta} } }\\
		  & \le \frac{  \Per_{H \os G }(S) }{ \mu_{ H \os G, \alpha}(S)^{ \frac{\delta-1}{\delta} } } 
		   = q_{ H \os G, \alpha }(S),
	\end{align*}
contradicting our assumption. If we let 
	\begin{align*}
	&a_1  :=  \Per_{  H \os G  }(S) - \frac{ \nu_{H}(e_0) }{ \sd_H(e_0) } \Per_{ G }(S_{e_0}^c),
	\ a_2  := \frac{ \nu_{H}(e_0) }{ \sd_H(e_0) } \Per_{G}(S_{e_0}^c),\\
	&b_1  :=  \mu_{ H \os G, \alpha}(S^c) - \nu_H(e_0) \mu_{ {G, \alpha_{e_0}} } (S_{e_0}^c),
	\ b_2  := \nu_H(e_0) \mu_{ G, \alpha_{e_0} } (S_{e_0}^c),
	\end{align*}
then	
	\begin{align*}
		 \frac{ a_1 + a_2 }{ (b_1 + b_2)^{ \frac{\delta -1}{\delta} } } & =  \frac{  \Per_{H \os G  }(S) }{ \mu_{ H \os G, \alpha}(S^c) ^{ \frac{\delta-1}{\delta} } } \le \frac{  \Per_{ { H \os G } }(S) }{ \min \{ \mu_{ H \os G, \alpha}(S^c), \mu_{ H \os G, \alpha}(S) \} ^{ \frac{\delta-1}{\delta} } } = q_{ H \os G, \alpha }(S),
	\end{align*}
and
	\begin{align*}
		 \frac{ a_1 }{ b_1^{ \frac{\delta -1}{\delta} } } & =  \frac{ \Per_{ H \os G }(S) - \frac{ \nu_{H}(e_0) }{ \sd_H(e_0) } \Per_{ G }(S_{e_0}^c) }{ \big(\mu_{ H \os G, \alpha}(S^c) - \nu_H(e_0) \mu_{ G, \alpha_{e_0} } (S_{e_0}^c) \big)^{\frac{\delta -1}{\delta}} } = \frac{ \Per_{{ H \os G } }(\tilde S) }{ \mu_{ H \os G, \alpha}((\tilde S)^c)^{\frac{\delta -1}{\delta}} } \stackrel{\eqref{eq:wlog}}{=} q_{ H \os G, \alpha }( \tilde S).
	\end{align*}
By our assumption, we have 
	\begin{equation}
		 \frac{ a_1 + a_2 }{ (b_1 + b_2)^{ \frac{\delta -1}{\delta} } } < \frac{ a_1 }{ b_1^{ \frac{\delta -1}{\delta} } }.
	\end{equation}
Then by Claim~\ref{claim:connected-components} in the proof of Proposition~\ref{prop:connected-components}, it follows that
	\begin{equation}
		 \frac{ a_2 }{ b_2^{ \frac{\delta -1}{\delta} } } \le \frac{ a_1 + a_2 }{ (b_1 + b_2)^{ \frac{\delta -1}{\delta} } },
	\end{equation}
which gives (b) after substitution.
\end{proof}

The next theorem is our main result on isoperimetric inequalities. It relates isoperimetric rations of $H \os G$ in terms of geometric parameters of $H$ and $G$ and their isoperimetric ratios.

\begin{theo}\label{thm:slash-isop} 
For $\delta \in [1,\infty)$, a graph $H$, an $s$-$t$ graph $G$, probability measures $\nu_H$ and $\nu_G$ on $E(H)$ and $E(G)$ respectively, and geodesic metrics $\sd_H$ and $\sd_G$, on $V(H)$ and $V(G)$, respectively, with $\sd_G$ \emph{normalized}, we have
	\begin{align*}
		\min_{S \subsetneq V(H\os G) \atop S \neq \emptyset} \inf_{ \alpha\in (0,1)^{E(H\os G)} } & q_{H \os G,\alpha}(S) \ge\\
		&  \min\left\{ \min_{e\in E(H)} \frac{ \nu_H(e)^{ \frac1\delta} }{ \sd_H(e) } \cdot \min_{ S \cap \{s(G), t(G) \} = \emptyset \atop S\not=\emptyset} \inf_{ \alpha \in (0,1)^{E(G)} }  \tilde{q}_{G,\alpha}(S)\right., \\
		&\hskip.8in\left. \min_{S \subsetneq V(H) \atop S \neq \emptyset}\inf_{ \alpha\in (0,1)^{E(H)} q_{H, \alpha}(S)} \cdot \min_{ \abs{ S \cap \{s(G), t(G) \} } =1 } \Per_{G}(S) \right\}
	\end{align*}
\end{theo}

\begin{proof}
For convenience let us introduce the following parameters for $H$, $G$, and $H\os G$:
	\begin{align*}
		p_{G} & \eqd  \min_{ \abs{ S \cap \{s(G), t(G) \} } =1 } \Per_{G}(S), \\
		\tilde q^\circ_{G} &\eqd \min_{ S \cap \{s(G), t(G) \}  = \emptyset \atop S\not=\emptyset} \inf_{ \alpha \in (0,1)^{E(G)} }  \tilde{q}_{G,\alpha}(S), \\ \rho_{H} &\eqd \min_{e\in E(H)} \frac{ \nu_H(e)^{ \frac1\delta} }{ \sd_H(e) }, \\
		q_{K} & \eqd \min_{S \subsetneq V(K) \atop S \neq \emptyset\}}\inf_{ \alpha \in(0,1)^{E(K)} } q_{K,\alpha}(S), \hskip.3in \text{ for } K \in \{H,H\os G\}.
	\end{align*}
Let $\alpha = (\alpha(e))_{e \in E( H \os G)} \subset(0,1)$ be arbitrary. For each $S \subset V(H \os G)$ with $S\not\in\{\emptyset, V(H\os G)\}$, define
	\begin{align*}
		N(S) \eqd &|\{e \in E(H): \{e^-,e^+\} \cap S = \emptyset \text{ but } (e\os V(G)) \cap S \neq \emptyset\}| \\
		&\hskip.1in+ |\{e \in E(H): \{e^-,e^+\} \subset S \text{ but } (e\os V(G)) \not\subset S\}|.
	\end{align*}
We will prove that
\begin{equation} \label{eq:slash-isop}
	q_{H \os G,\alpha}(S) \ge \min \{ \rho_{H} \cdot \tilde q_G^\circ, q_{H} \cdot p_{G} \}
\end{equation}
by induction on $N(S) \in \bN\cup\{0\}$. As we will see, the base case $N(S) = 0$ requires as much work as the inductive step. Note that $N(S) = N(S^c)$, and hence by passing to $S^c$ if necessary, we may assume that $\mu_{H\os G,\alpha}(S) \leq \frac{1}{2}$ without changing the value of $q_{H\os G,\alpha}(S)$ or $N(S)$ (which are the only two quantities that matter).

Assume that $N(S) = 0$. Noting that $|S_e\cap \{s(G),t(G)\}|=1 \iff e\in \partial_H(S\cap V(H))$ and because $N(S)=0$ we have 
	\begin{equation*}
		\partial_{H\os G}(S)=\bigcup_{e\in \partial_H (S\cap V(H))}\big( e\os \partial_G(S_e)\big), 
	\end{equation*}
and thus         
	\begin{equation}\label{eq:Per}
		\Per_{H\os G}(S) = \sum_{e\in \partial_H (S\cap V(H))}\frac{\nu_H(e)}{\sd_H(e)} \Per_{G}(S_e).
	\end{equation}
It follows from \eqref{eq:decomposition} that
	\begin{align*}
		\mu_{H\os G, \alpha} (S) & \stackrel{\eqref{eq:decomposition}}{=} \sum_{e\in E(H)} \nu_H(e) \mu_{G, {\alpha_e}} (S_e) \\
		&\overset{\eqref{eq:induced-measure}}{=} \sum_{e\in E(H)\atop e^-,e^+\in S} \nu_H(e)+ \sum_{e\in E(H)\atop e^+\in S, e^-\not\in S} \nu_H(e) \mu_{G, {\alpha_e}} (S_e)+ \sum_{e\in E(H)\atop e^-\in S, e^+\not\in S} \nu_H(e) \mu_{G, {\alpha_e}} (S_e),
	\end{align*}
where the last equality follows from the assumption that $N(S)=0$, and thus that
$e^+,e^-\in S$ implies that $S_e=V(G)$, and $e^+,e^-\not\in S$ implies that $S_e=\emptyset$. 
Hence after defining
\begin{align*}
    \beta(e) \eqd \begin{cases} \frac12 &\text{if   $e^+,e^-\in S$, or  $e^+,e^-\not\in S$,}\\
                                                \mu_{G,{\alpha_e}}(S_e)  &\text{if  $e^+\in S$, and $e^-\not\in S$,}\\
                                                1-\mu_{G,\alpha_e}(S_e)  &\text{if   $e^+\not\in S$, and  $e^-\in S$,}\end{cases}
\end{align*}
we get
\begin{equation}\label{eq:mu} 
    \mu_{H\os G,\alpha} (S)= \sum_{e\in E(H),\atop e^+\in S\cap V(H)} \nu_H(e) \beta(e)+ \sum_{e\in E(H), \atop e^-\in S\cap V(H)} \nu_H(e)(1- \beta(e)) \overset{\eqref{eq:induced-measure}}{=} \mu_{H,\beta} (S\cap V(H)).
\end{equation}       
Combining \eqref{eq:Per} and \eqref{eq:mu}, 
we obtain 
\begin{align*}
    q_{H\os G, \alpha}(S) &= \frac{\Per_{H\os G}(S)}{\mu_{H\os G,\,\alpha}(S)^{\frac{\delta-1}\delta}} \\
    &\overset{\eqref{eq:Per},\eqref{eq:mu}}{=} \frac{\sum_{e\in \partial_H (S\cap V(H))} \frac{\nu_H(e)}{\sd_H(e)} \Per_{G}(S_e) }{\mu_{H,\,\beta} (S\cap V(H))^{\frac{\delta-1}\delta}} \\
    &\geq \frac{\Per_H(S\cap V(H)) \cdot p_G}{\mu_{H,\,\beta} (S\cap V(H))^{\frac{\delta-1}\delta}} \\
    &= q_{H,\,\beta}(S\cap V(H)) \cdot p_G 
    \geq q_H \cdot p_G,
\end{align*}
where in the last equality, we've used the fact that
$\mu_{H,\beta} (S\cap V(H)) \overset{\eqref{eq:mu}}{=} \bar{\mu}_{H\os G,\alpha}(S) \leq \frac{1}{2}$. Inequality \eqref{eq:slash-isop} follows in the base case $N(S) = 0$.

Now we prove the inductive step. Assume $N(S) > 0$ and that \eqref{eq:slash-isop} holds for all $S' \subset V(H\os G)$ with $S' \not\in \{\emptyset,V(H\os G)\}$ and $N(S') < N(S)$. Of course, in this situation we have two cases: (I) there exists $e \in E(H)$ with $\{e^-,e^+\} \subset S$ but $e\os V(G) \not\subset S$, and (II) there exists $e \in E(H)$ with $\{e^-,e^+\} \cap S = \emptyset$ but $(e\os V(G)) \cap S \neq \emptyset$.

Assume that (I) holds. Let $e_0 \in E(H)$ such that $\{e_0^-,e_0^+\} \subset S$ but $e_0\os V(G) \not\subset S$. Then, setting $S' := S \cup (e_0\os S_{e_0}^c)$ (and recalling that we may assume $\mu_{H\os G,\alpha}(S) \leq \frac{1}{2}$), Lemma~\ref{lem:dichotomy} implies that one of the following holds:
\begin{itemize}
    \item[(a)] $S' \not\in \{\emptyset,V(H\os G)\}$ and $q_{H\os G,\alpha}(S) \geq q_{H\os G,\alpha}(S')$.
    \item[(b)] $q_{H\os G,\alpha}(S) \geq \rho_H \cdot\tilde q_G^\circ$.
\end{itemize}
Since $N(S') = N(S) - 1$, if (a) holds, then we get \eqref{eq:slash-isop} by the inductive hypothesis. If (b) holds then we get \eqref{eq:slash-isop} automatically. This completes the proof for case (I).

Now assume that (II) holds. Let $B$ be a connected component of $S$ with $q_{H\os G,\alpha}(S) \geq q_{H\os G,\alpha}(B)$, which exists by Proposition~\ref{prop:connected-components}. Note that $\mu_{H\os G,\alpha}(B) \leq \mu_{H\os G,\alpha}(S) \leq \frac{1}{2}$. Consider the set $F:=\{e \in E(H): \{e^-,e^+\} \cap B = \emptyset \text{ but } (e\os V(G)) \cap B \neq \emptyset\}$. Since $B$ is a connected component, necessarily $\abs{F}\in \{0,1\}$. Therefore exactly one of the following two subcases must hold: (i) $\abs{F}=1$, i.e. there exist $e\in E(H)$ and $B' \subset V(G) \setminus \{s(G),t(G)\}$ such that $B = e\os B'$ , or (ii) $\abs{F}=0$, i.e. $\{e \in E(H): \{e^-,e^+\} \cap B = \emptyset \text{ but } (e\os V(G)) \cap B \neq \emptyset\} = \emptyset$ . Assume that (i) holds. Then Lemma~\ref{lem:insideGformula} implies $q_{H\os G,\alpha}(B) \geq \rho_H \cdot\tilde q_G^\circ$, and \eqref{eq:slash-isop} follows. Finally, assume that (ii) holds. Then the following is true.
\begin{itemize}
    \item Since $B \subset S$ is a connected component, $$\{e \in E(H): \{e^-,e^+\} \subset B \text{ but } (e\os V(G)) \not\subset B\} \\ \subset \{e \in E(H): \{e^-,e^+\} \subset S \text{ but } (e\os V(G)) \not\subset S\}.$$
    \item Since we are in case (II), $$\{e \in E(H): \{e^-,e^+\} \cap S = \emptyset \text{ but } (e\os V(G)) \cap S \neq \emptyset\} \neq \emptyset.$$
    \item Since we are in subcase (ii), $$\{e \in E(H): \{e^-,e^+\} \cap B = \emptyset \text{ but } (e\os V(G)) \cap B \neq \emptyset\} = \emptyset.$$
\end{itemize}
These three items together with the definition of $N(S),N(B)$ imply $N(B) < N(S)$. Hence \eqref{eq:slash-isop} holds by the inductive hypothesis. This completes the proof of the inductive step in all cases.
\end{proof}

\subsection{The isoperimetric dimension of $\os$-powers}\label{sec:slash-powers}
In this subsection, we again fix an isoperimetric exponent $\delta \in [1,\infty)$, an $s$-$t$ graph $G$,  a normalized geodesic metric $\sd_G$ on $V(G)$, and a fully supported probability measure $\nu_G$ on $E(G)$. We retain the same notation from the previous subsection.

\begin{defi}[$\os$-powers] \label{def:slash-power}
Given an $s$-$t$ graph $G$, we define its $n$-th $\os$-power $G^{\os n}$ for $n \in \bN$ recursively as follows: $G^1 :=G$ and $G^{\os n+1}:=G^{\os n}\os G$. 
\end{defi}

\begin{rema}\label{rem:prod}
It holds that $E(G^{\os n})=\{ \os_{j=1}^n e_j: \{e_j\}_{j=1}^n \in E(G)\}$, where $\os_{j=1}^n e_j$ is defined in the obvious way.
\end{rema}
Recall the following notation from the previous subsection:
	\begin{align*}
		q_{G^{\os n},\alpha} &= \min_{S \subsetneq V(G^{\os n}) \atop S \neq \emptyset} \frac{\Per_{\nu_G^{\os n},\sd_G}(S)}{\min\left\{\mu_{\alpha}(\nu_G^{\os n})(S), \mu_{\alpha}(\nu_G^{\os n})(S^c)\right\}^{\frac{\delta-1}{\delta}}} \\
		q_{G^{\os n}} &= \inf_{\alpha\in (0,1)^{E(G^{\os n})} } q_{G^{\os n},\alpha} \\
		\tilde q^\circ_{G} & = \min_{ S \cap \{s(G), t(G) \}  = \emptyset \atop S\not=\emptyset} \inf_{ \alpha \in (0,1)^{E(G)} }  \tilde{q}_{G,\alpha}(S) .
	\end{align*}
We characterize precisely when a $\os$-power admits a uniform lower bound on the isoperimetric ratio.

\begin{theo}\label{theo:isoper-inequality-for-slashpowers} Let $G$ be an $s$-$t$ graph and assume that $|V(G)| > 2$. Then the following conditions are equivalent:
	\begin{enumerate}
  		\item There exists $c>0$ such that for all $n\in\bN$, $$\min_{S \subsetneq V(G^{\os n}) \atop S \neq \emptyset} \frac{\Per_{\nu_G^{\os n},\sd_G^{\os n}}(S)}{\min\left\{\mu(\nu_G^{\os n})(S),\mu(\nu_G^{\os n})(S^c)\right\}^{\frac{\delta-1}{\delta}}}\ge c.$$
 		\item $$  \rho_G \eqd \min_{e\in E(G)} \frac{\nu_G^{\frac1\delta}(e)}{\sd_G(e)} \ge 1 \text{ and } p_G \eqd \min_{S\subset V(G) \atop |S\cap \{s(G),t(G)\}|=1} \Per_{\nu_G,\sd_G}(S) \geq 1.$$
 		\item There exists $c>0$ such that for all $n\in\bN$, $$\inf_{\alpha \in (0,1)^{E(G^{\os n})}} \min_{S \subsetneq V(G^{\os n}) \atop S \neq \emptyset} \frac{\Per_{\nu_G^{\os n},\sd_G}(S)}{\min\left\{\mu_{\alpha}(\nu_G^{\os n})(S),\mu_{\alpha}(\nu_G^{\os n})(S^c)\right\}^{\frac{\delta-1}{\delta}}}\ge c.$$
	\end{enumerate} 
Moreover, in both (1) and (3), $c$ can be taken to be $$\min \{ \Per_{\nu_G,\sd_G} (S)\colon \emptyset \neq S\subsetneq V(G)\}.$$
\end{theo}

\begin{proof}
We retain the notational conventions from the previous subsection, e.g., $\Per_{G^{\os n}}$ means $\Per_{\nu_G^{\os n},\sd_G^{\os n}}$ and $\mu_{G^{\os n},\alpha}$ means $\mu_{\alpha}(\nu_G^{\os n})$.

The implication (3) $\implies$ (1) is immediate, and the implication (2) $\implies$ (3) holds by induction, using Theorem~\ref{thm:slash-isop}.
Indeed, let $$c \eqd \min \{ \Per_{\nu_G,\sd_G} (S)\colon \emptyset \neq S\subsetneq V(G)\}.$$  
Then clearly $q_{G}\ge c$. Moreover, $\tilde q^\circ_{G}\ge c$ and  assuming that $q_{G^{\os n}}\ge c$, for some $n\in\bN$,
we first observe that 
   $$ \rho_{G^{\os n}} \eqd \min_{e\in E(G^{\os n})} \frac{\nu_{G^{\os n}}^{\frac1\delta}(e)}{\sd_{G^{\os n}}(e)}  \stackrel{(\text{Rem. }\ref{rem:prod}) \land \eqref{eq:prod-measure}\land \eqref{eq:prod-metric}}{=} \min_{e_1,e_2,\ldots e_n\in E(G)} \frac{\prod_{j=1}^n\nu_G^{\frac1\delta}(e_j)}{\prod_{j=1}^n \sd_G(e_j)}\ge 1$$
and then  by  applying Theorem ~\ref{thm:slash-isop} to $H=G^{\os n}$ and $G$, we obtain from the induction hypothesis that
$$q_{G^{\os (n+1)}} \ge \min \{ \rho_{G^{\os n}} \cdot \tilde q^\circ_{G},q_{G^{\os n}}\cdot p_G\}\ge \min\{  \tilde q^\circ_{G},q_{G^{\os n}} \}\ge c. $$

It remains to prove that (1) implies (2), which we do by contraposition. 
Assume that (2) does not hold, so that $\rho_G < 1$ or $p_G < 1$. Assume first that $\rho_G < 1$. Let $e\in E(G)$ such that $\frac{\nu_G(e)^{\frac1\delta}}{\sd_G(e)}< 1$. Set $S := V(G)\setminus \{s(G),t(G)\} \not=\emptyset$, and $S_n := e^{\os n} \os S \subset V(G^{\os n}\os G)$ for $n\in\bN$. Since $\nu_G$ is fully supported and $E(G)$ has more than two edges, $\nu_G(e) < 1$. From this we get
	\begin{equation*}
		\mu_{G^{\os n+1}}(S_n) =\mu_{G^{\os n+1}}(e^{\os n} \os S) \overset{\eqref{eq:decomposition2}}{=} \nu_G^{\os n}(e^{\os n}) \mu_G(S) = \nu_G(e)^n \mu_G(S) \underset{n\to\infty}{\to} 0.
	\end{equation*}
Therefore, for $n$ sufficiently large, $\mu_{G^{\os n+1}}(S_n) \leq \frac{1}{2}$. Using this we get, for all $n$ sufficiently large,
\begin{align*}
    q_{G^{\os n+1},\frac12}(S_n) = \tilde{q}_{G^{\os n+1},\frac12}(S_n) \overset{\text{Lem }\ref{lem:insideGformula}}{=} \frac{\nu_G^{\os n}(e^{\os n} )^{\frac1\delta} }{\sd_{G^{\os n}}(e^{\os n})} \frac{\Per_G(S)}{\mu_G(S)^{\frac{\delta-1}\delta}}=  \Bigg(\frac{\nu_G(e)^{\frac1\delta}}{\sd_G(e)}\Bigg)^n \frac{\Per_G(S)}{\mu_{G}(S)^{\frac{\delta-1}\delta}} \underset{n\to\infty}{\to} 0.
\end{align*}
This shows (1) in the case $\rho_G < 1$.

Now assume that $p_G < 1$. Choose any $S\subset V(G)$ with $|S\cap\{s(G),t(G)\}|=1$ and
\begin{equation*}
    \Per_G(S)=\min \{ \Per_G(B)\colon B\subset V(G),  |B\cap\{s(G),t(G)\}|=1\}<1.
\end{equation*}
Without loss of generality we may assume that $s(G) \in S$ and $t(G) \not\in S$. The set $S$ must be connected, because otherwise the connected component of $S$ containing $s(G)$ would have strictly smaller perimeter. By the same reasoning, since $\Per_G(S^c)=\Per_G(S)$, $S^c$ must also be connected. We consider 2 cases: either $S$ or $S^c$ is a singleton, and neither $S$ nor $S^c$ is a singleton.

{\bf Case 1:} Either $S$ or $S^c$ is a singleton. \\
We will only treat the case that $S$ is a singleton, since the argument in the other case is identical. Then we have that $S = \{s(G)\}$, and hence by our convention of the orientation of an $s$-$t$ graph,
\begin{align} \label{eq:Per(s(G))<1}
    1 > \Per_G(\{s(G)\})=\sum_{e\in E(G) \atop e^-=s(G)} \frac{\nu_G(e)}{\sd_G(e)}.
\end{align}

Let $e_0\in E(G)$ such that $e_0^-=s(G)$ and such that there exists $e_1 \in E(G)$ with $e_1^- = e_0^+$ (such an edge $e_0$ must exist since $V(G)$ has more than $2$ elements). We define, for $n\in\bN$, $n\ge 2$,
\begin{equation*}
    S_n := e_0 \os V(G^{\os n-1}) \subset V(G\os G^{\os (n-1)})=V(G^{\os n}).
\end{equation*}
It holds that
\begin{align*}
        \partial_{G^{\os n}} (S_n)= \{&f_1\os f_2\os\cdots\os f_{n} \in E(G^{\os n}): f_1^-=e_0^+, f_j^-=s(G) \text{ for $2 \leq j \leq n$} \}\\
        &\cup \{f_1\os f_2\os\cdots\os f_{n}\in E(G^{\os n}): f_1\not= e_0, f_j^-=s(G) \text{ for $1 \leq j\leq n$} \}.
\end{align*}
Note that at least the first of above two sets cannot be empty by choice of $e_0$. It follows that 
	\begin{align*}
		\Per_{G^{\os  n}}  (S_n) &= \left(\sum_{e\in E(G) \atop e^-=e^+_0} \frac{\nu_G(e)}{\sd_G(e)} \right)\left(\sum_{f\in E(G) \atop f^-=s(G)} \frac{\nu_G(f)}{\sd_G(f)}\right)^{n-1} + \left(\sum_{e\in E(G)\setminus\{e_0\}\atop e^-=s(G)} \frac{\nu_G(e)}{\sd_G(e)}\right) \left(\sum_{f\in E(G) \atop f^-=s(G)} \frac{\nu_G(f)}{\sd_G(f)}\right)^{n-1}\\
		&=\left(\sum_{e\in E(G)\setminus\{e_0\}\atop e^-=s(G)} \frac{\nu_G(e)}{\sd_G(e)}+\sum_{e\in E(G) \atop e^-=e^+_0} \frac{\nu_G(e)}{\sd_G(e)}  \right) \left(\sum_{f\in E(G) \atop f^-=s(G)} \frac{\nu_G(f)}{\sd_G(f)}\right)^{n-1} \\
		&=\left(\sum_{e\in E(G)\setminus\{e_0\}\atop e^-=s(G)} \frac{\nu_G(e)}{\sd_G(e)}+\sum_{e\in E(G) \atop e^-=e^+_0} \frac{\nu_G(e)}{\sd_G(e)}  \right) \Per_G(\{s(G)\})^{n-1} \overset{\eqref{eq:Per(s(G))<1}}{\underset{n\to\infty}{\to}} 0.
	\end{align*} 
Thus, the proof that (1) is not satisfied, is complete in this case if we can verify that
	\begin{equation*}
		\inf_{n\in\bN} \min\{\mu_{G^{\os n}}(S_n),\mu_{G^{\os n}}(S_n^c)\}>0.
	\end{equation*}
First note that
	\begin{equation*}
		\mu_{G^{\os n}}(S_n)=\mu_{G^{\os n}}(e_0 \os V(G^{\os n-1}))  \overset{\eqref{eq:decomposition2}}{=} \nu_G(e_0)\mu_{G^{\os(n-1)}}( V(G^{\os n-1})) = \nu_G(e_0)\ge \min_{e\in E(G)} \nu_G(e)>0.
	\end{equation*}
Secondly, there must exist $e_2\in E(G)$ with $e_2^+=t(G) $ and $e_2^- \neq s(G)$, from which it follows that $e_2\os e_2 \os V(G^{\os n-2})$ is a subset of $V(G\os G\os G^{\os (n-2)})=V(G^{\os n})$ disjoint from $S_n$, and thus
	\begin{align*}
		\mu_{G^{\os n}}(S_n^c) \ge \mu_{G^{\os n}}(e_2\os e_2 \os V(G^{\os n-2})) \overset{\eqref{eq:decomposition2}}{=} \nu_G^{\os 2} (e_2\os e_2) \ge \min_{e\in E(G)} \nu_G(e)^2>0.
	\end{align*}

{\bf Case 2:} $\min\{ |S|, |S^c|\} \geq 2$. \\
Since $S$ and $S^c$ are connected, there exist $e_1,e_2\in E(G)$ such that $e_1^-,e_1^+\in S$ and $e_2^-,e_2^+\in S^c$. We define $S_n\subset V(G^{\os n})$ recursively for all $n\in\bN$. Let  $S_1:=S$ and
	\begin{equation*}
		S_{n+1}:= \bigcup_{e\in E(G)} e\os S'_{n,e}\subset V(G\os G^{\os n}),
	\end{equation*}
where 
	\begin{align*}
		S'_{n,e}:=
			\begin{cases}  V(G^{\os n}) & e^-,e^+\in S \\
				\emptyset & e^-,e^+\not\in S \\
				S_n & e^-\in S, e^+\not\in S \\
				S_n^c & e^+\in S, e^-\notin S 
			\end{cases}.
	\end{align*}
In particular we have $S'_{n,e_1}= V(G^{\os n})$ and $S'_{n,e_2}= \emptyset$.
For all $n\in\bN$, we have
$$\partial_{G^{\os (n+1)}}(S_{n+1})= \bigcup_{e\in\partial_G S} e\oslash \partial_{G^{\os n}} S'_{n,e},$$
and thus 
	\begin{align*}
		\Per_{G^{\os n+1}} (S_{n+1}) = \sum_{e\in\partial_G (S)} \frac{\nu_G(e)}{\sd_G(e)} \Per_{G^{\os n} }(S'_{n,e})= 
\sum_{e\in\partial_G (S)} \frac{\nu_G(e)}{\sd_G(e)} \Per_{G^{\os n} }(S_{n}) = \Per_G(S)\Per_{G^{\os n} }(S_{n}),
	\end{align*}
from which we deduce by induction that
	\begin{equation*}
		\Per_{G^{\os n}}(S_{n}) = \Per_G(S)^n \underset{n\to\infty}{\to} 0.
	\end{equation*}
On the other hand,
	\begin{equation*}
		\mu_{G^{\os n}}(S_n) = \mu_{G^{\os n}}\Big(\bigcup_{e\in E(G)} e\os S'_{n-1,e} \Big) \ge  \mu_{G^{\os n}}(e_1\os S'_{n-1,e_1}) = \mu_{G^{\os n}}(e_1\os V(G^{\os (n-1)})) \stackrel{ \eqref{eq:decomposition2}}{=} \nu_G(e_1),
	\end{equation*}
and observing that $S_{n}^c = \bigcup_{e\in E(G)} e\os [S'_{n,e}]^c$, a similar argument shows that $\mu_{G^{\os n}}(S_n^c)\ge \nu(e_2)$. Combining these last three estimates yields
	\begin{equation*}
		q_{G^{\os n},\frac{1}{2}}(S_n) \underset{n\to\infty}{\to} 0,
	\end{equation*}
proving the negation of (1).
\end{proof}

We immediately get the next corollary, which characterizes the isoperimetric dimension of $G^{\os n}$ in terms of easily verifiable conditions on $G$.

\begin{coro} \label{cor:slash-isop}
If an $s$-$t$ graph $G$ satisfies $|V(G)| > 2$, then the following are equivalent:
\begin{enumerate}
    \item For all $n \in \bN$, $G^{\os n}$ has $(\mu(\nu_G^{\os n}),\nu_G^{\os n},\sd_G^{\os n})$-isoperimetric dimension $\displaystyle \delta = \max_{e\in E(G)}\frac{\log (\nu_G(e))}{\log(\sd_G(e))}$ with constant $\displaystyle C \leq \max_{S\subsetneq V(G) \atop S\neq \emptyset}\Per_{\nu_G,\sd_G}(S)^{-1}$.
    \item There exist $\delta \in [1,\infty)$ and $C\in (0,\infty)$ such that, for all $n \in \bN$, $G^{\os n}$ has $(\mu(\nu_G^{\os n}),\nu_G^{\os n},\sd_G^{\os n})$-isoperimetric dimension $\delta$ with constant $C$.
    \item $\displaystyle \min_{S\subset V(G) \atop |S\cap \{s(G),t(G)\}|=1}\Per_{\nu_G,\sd_G}(S)\ge 1$.
\end{enumerate}
\end{coro}

\subsection{Applications}
In this section we show how to apply the results in Section \ref{sec:isoperimetric-inequalities} to two important sequences of graphs.

\subsubsection{Isoperimetric dimensions of diamond graphs}
Let $k,m \geq 2$ be integers. Recall from Example~\ref{ex:diamonds} that the $m$-branching diamond graph of depth $k$, $\sD_{k,m}$, is equipped with $\nu_{\sD_{k,m}}$ the uniform probability measure on $E(\sD_{k,m})$ and $\sd_{\sD_{k,m}}$ the normalized geodesic metric on $V(\sD_{k,m})$. It can be easily verified that $\Per_{\nu_{\sD_{k,m}},\sd_{\sD_{k,m}}}(S) \geq 1$ for every $S \subset V(\sD_{k,m})$ with $|S \cap \{s(\sD_{k,m}),t(\sD_{k,m})\}| = 1$. Indeed, by symmetry and the fact that connected components of $S$ have smaller perimeter than $S$, it suffices to check the inequality assuming that $S$ is connected, $s(\sD_{k,m}) \in S$, and $t(\sD_{k,m}) \not\in S$. It is clear that any such set $S$ must be a union of directed paths $\{P_i\}_{i=1}^j$, $1 \leq j \leq m$, starting at the common vertex $s(\sD_{k,m})$ and ending at non-neighboring vertices. It is easily seen that $\Per_{\nu_{\sD_{k,m}},\sd_{\sD_{k,m}}}(S) = 1$ in this case. It is also clear that $\displaystyle \max_{\emptyset \neq S\subsetneq V(G)}\Per_{\nu_G,\sd_G}(S)^{-1} \leq \frac{m}{2}$, and thus by Corollary~\ref{cor:slash-isop}, we get that:

\begin{coro} \label{ex:isop-diamonds}
For all, $k,m\ge 2$ and all $n\in \bN$, $\sD_{k,m}^{\os n}$ has $(\mu(\nu_{\sD_{k,m}}^{\os n}),\nu_{\sD_{k,m}}^{\os n},\sd_{\sD_{k,m}}^{\os n})$-isoperimetric dimension $1+\frac{\log m}{\log k}$ with constant $C \leq \frac{m}{2}$. In particular, the classical binary diamond graph $\dia_n$ has
 $(\mu(\nu_{ \sD_{1}^{\os n} }), \nu_{ \sD_{1}^{\os n} }, \sd_{ \sD_{1}^{\os n} } )$-isoperimetric dimension $2$ with constant $C \leq 1$ 
\end{coro}

\subsubsection{Isoperimetric dimensions of Laakso graphs}
Let $\La_1$ denote the level 1 Laakso graph (originally studied by Lang  and Plaut \cite[Theorem~2.3]{LangPlaut01}) depicted in Figure~\ref{fig:Laakso}. We give labels to the vertices as $V(\La_1) = \{s(\La_1) = u_0, u_{1/4}, u_{1/2+}, u_{1/2-}, u_{3/4}, u_1=t(\La_1)\}$ so that the edge set is 
$$E(\La_1) = \{(u_0,u_{1/4}), (u_{1/4},u_{1/2+}), (u_{1/4},u_{1/2-}), (u_{1/2+},u_{3/4}), (u_{1/2-},u_{3/4}), (u_{3/4},u_1)\}.$$ 
\begin{figure}[H]
    \centering
   \includegraphics[trim=200 325 175 150,clip,scale=.5]{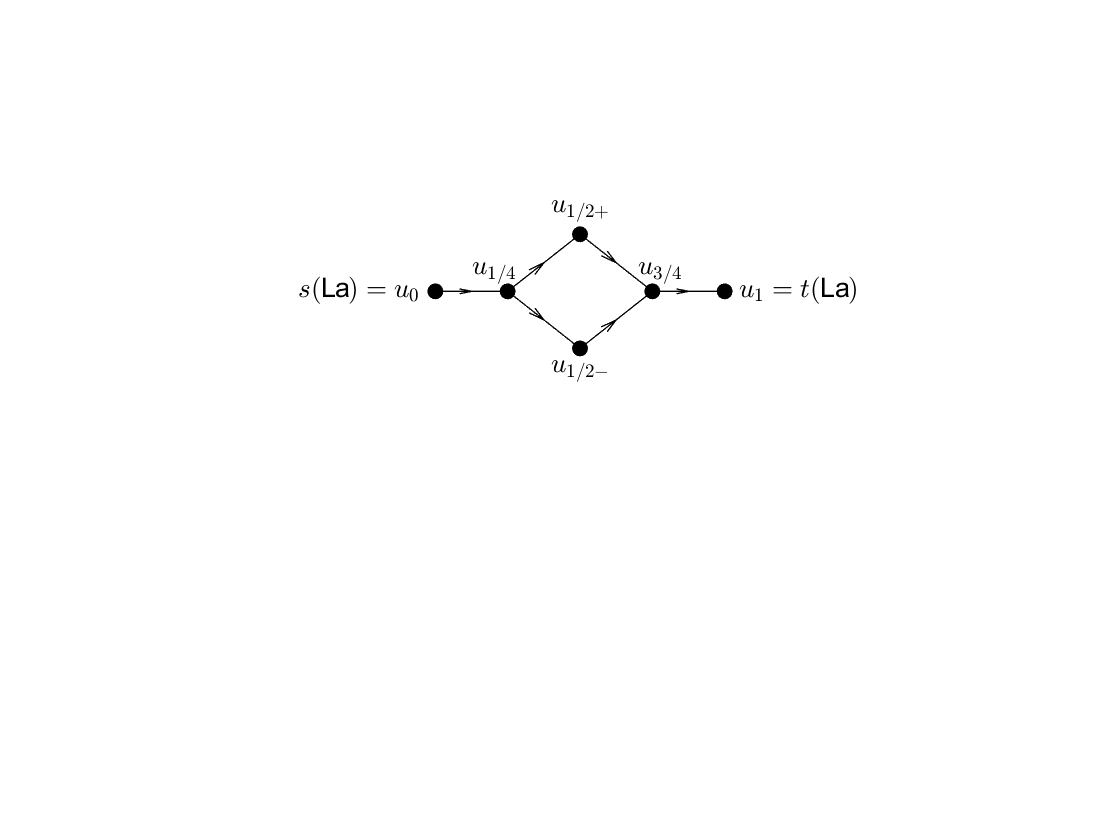}
       \caption{The Laakso graph $\La_1$.}
    \label{fig:Laakso}
\end{figure}
Equip $V(\La_1)$ with the normalized geodesic metric $\sd_{\La_1}(e) := \frac{1}{4}$ for every $e \in E(\La_1)$. If $\nu_{\La_1,u}$ is the uniform probability measure on $E(\La_1)$, then $\Per_{\nu_{\La_1,u}, \sd_{\La_1}}(\{s(\La_1)\}) = \frac23$ and thus by Corollary~\ref{cor:slash-isop}, there is no $\delta < \infty$ such that $\La_n:=\La_1^{\os n}$ has $(\mu(\nu_{\La_1,u}^{\os n}), \nu_{\La_1,u}^{\os n}, \sd_{\La_1}^{\os n})$-isoperimetric dimension $\delta$ with a fixed constant $C\in(0,\infty)$.
However, if $\nu_{\La_1,p}$ is the probability measure on  $E(\La_1)$ defined by $\nu_{\La_1,p}(e) := \frac{1}{4}$ if $e \in \{(u_0,u_{1/4}), (u_{3/4},u_{1})\}$ and $\nu_{\La_1,p}(e) := \frac{1}{8}$ otherwise, then it is easy to check that $\Per_{\nu_{\La_1,p}, \sd_{\La_1}}(S)\ge 1$ for every $\emptyset \neq S\subsetneq V(\La_1)$. Therefore, since $\frac{\log(1/8)}{\log(1/4)}=\frac32$, Corollary~\ref{cor:slash-isop} gives:

\begin{coro}
There is $C<\infty$ such that, for every $n\in \bN$, $\sL_{n}$ has $(\mu(\nu_{\La_1,p}^{\os n}), \nu_{\La_1,p}^{\os n}, \sd_{\La_1}^{\os n})$-isoperimetric dimension $\frac32$ with constant $C<\infty$.
\end{coro}

\section{Lipschitz-spectral profile of $\os$-products and $\os$-powers}
\label{sec:Lip-spec}

The main goal of this section is to compute the Lipschitz-spectral profile of $\os$-powers of $s$-$t$ graphs $G$ when $E(G)$ is equipped with the uniform probability measure (Corollary \ref{cor:spec-slashpower}). This result will be obtained as a particular case of a more general study of the Lipschitz-spectral profile of $\os$-products (Theorems \ref{thm:stronglyorthog},\ref{thm:Lpnorms},\ref{thm:Lipgrowth}). Throughout this section, fix an integer $k \ge 2$.

\begin{rema}
We remark that none of the results of this section require the vertices in an $s$-$t$ graph $G$ to lie on a directed edge path from $s(G)$ to $t(G)$; the results apply to more general graphs.
\end{rema}

\subsection{Operators between function spaces}
We introduce various operators between function spaces that we use to build orthogonal sets of Lipschitz functions on $\os$-products. The first two operators are $\os$-products and barycentric extensions of functions. These operators are defined whenever the relevant graphs are $s$-$t$.

\begin{defi}[$\os$-products of functions]
Given a graph $H$, $s$-$t$ graph $G$, and functions $h: E(H) \to \bR$, $g_1: V(G) \to \bR$, $g_2: E(G) \to \bR$ with $g_1(s(G)) = g_1(t(G)) = 0$, we define $h \os g_1: V(H \os G) \to \bR$ and $h \os g_2: E(H \os G) \to \bR$ by $(h \os g_1)(e \os u) := h(e) \cdot g_1(u)$ and $(h \os g_2)(e \os e') := h(e) \cdot g_2(e')$. Note that $h \os g_1$ is well-defined because $g_1(s(G)) = g_1(t(G)) = 0$. 
\end{defi}

Given a real-valued function $f$ on $V(H)$, a \emph{barycentric extension} operator will return a function on $V(H \os \Pk)$ by taking a natural barycentric combination of the values of $f$ at the two corresponding vertices of $H$ where each copy of $\Pk$ is attached.

\begin{defi}[Barycentric extension]
Given a graph $H$ and function $f \colon V(H) \to \bR$, we define its \emph{barycentric extension} $\Lin(f) \colon V(H \os \Pk) \to \bR$ by
	\begin{equation*}
		\Lin(f)(u) := (1-\tfrac{i}{k})f(e^-) +\tfrac{i}{k}f(e^+)
	\end{equation*}
for all $u = e \os \tfrac{i}{k} \in V(H \os \Pk)$.
\end{defi}

The next two operators, pullbacks and conditional expectations, require a graph morphism $\theta: V(G) \to V(G')$ and a measure $\mu_G$ on $V(G)$ rather than an $s$-$t$ structure.

Let $G,H$ be graphs and $\theta: V(G) \to V(G')$ a graph morphism. We define $\sigma(\theta)$ to be the $\sigma$-algebra on $V(G)$ or $E(G)$ generated by $\theta$. That is, the atoms of $\sigma(\theta)$ are preimages of singleton subsets of $V(G')$ or $E(G')$ under $\theta$, and $\sigma(\theta)$ is generated by these atoms.

\begin{defi}[Pullbacks induced by graph morphisms]
Let $G,H$ be graphs and $\theta: V(G) \to V(G')$ a graph morphism. For a given function $f \in \bR^{V(G')}$, we define its \emph{pullback} by $\theta^*(f) := f \circ \theta \in \bR^{V(G)}$. Since $\theta$ is a graph morphism, it induces a well-defined map $\theta: E(G) \to E(G')$, and thus we get a pullback operator $\theta^*: \bR^{E(G')} \to \bR^{E(G)}$ given by the same formula.
\end{defi}

\begin{rema}
A function on $V(G)$ or $E(G)$ is $\sigma(\theta)$-measurable if and only if it is in the image of $\theta^*$.
\end{rema}

\begin{defi}[Conditional expectations induced by graph morphisms]
Let $G,G'$ be graphs, $\theta: V(G) \to V(G')$ a graph morphism, and $\nu_G$ a measure on $E(G)$. We define $\bE_{\nu_G}^\theta$ to be the conditional expectation with respect to the measure space $(E(G),\sigma(\theta),\nu_G)$. That is, for every $g: E(G) \to \bR$, $\bE_{\nu_G}^\theta(g): E(G) \to \bR$ is $\sigma(\theta)$-measurable and satisfies
\begin{equation*}
    \int_{E(G)} h\cdot\bE_{\nu_G}^\theta(g)d\nu_G = \int_{E(G)} h\cdot gd\nu_G
\end{equation*}
for every $\sigma(\theta)$-measurable $h: E(G) \to \bR$.
\end{defi}

\subsection{Strongly orthogonal sets of Lipschitz functions on $\os$-products}
The following definition is the crucial strengthening of orthogonality needed to study Lipschitz-spectral profile of $\os$-products.

\begin{defi}[Strong orthogonality]
When $H$ is a graph and $f: V(H) \to \bR$ is a function, we define the \emph{induced edge-functions} $f_-,f_+: E(H) \to \bR$ by $f_-(e) := f(e^-)$ and $f_+(e) := f(e^+)$. For $\nu_H$ a measure on $E(H)$ and $f,g: V(H) \to \bR$, we say that $f,g$ are \emph{strongly $\nu_H$-orthogonal} if $f_{\eps_1},g_{\eps_2}$ are orthogonal in $L_2(E(H),\nu_H)$ for all $\eps_1,\eps_2 \in \{-,+\}$. We say that a set of functions $F \subset \bR^{V(H)}$ is strongly $\nu_H$-orthogonal if $f,g$ are strongly $\nu_H$-orthogonal for all $f \neq g \in F$.
\end{defi}

\begin{rema} \label{rem:strongorthog->orthog}
It follows easily from \eqref{eq:induced-integral} that strong $\nu_H$-orthogonality of $f,g: V(H) \to \bR$ implies orthogonality of $f,g$ in $L_2(V(H),\mu(\nu_H))$, and this is all that is needed as far as Lipschitz-spectral is concerned. However, for our inductive argument to close in the proof of Theorem~\ref{thm:stronglyorthog}, we need to consider strongly orthogonal sets of functions.
\end{rema}

The main goal of this subsection is to extend a given strongly $\nu_H$-orthogonal set of functions on $V(H)$ to a strongly $\nu_H \os \nu_G$-orthogonal set of functions on $V(H \os G)$ with control on the $L_1$, $L_\infty$, and Lipschitz norms of the functions. To do this, we must work with a special class of graphs $G$.

For $G$ an $s$-$t$ graph, a graph morphism $\pi: V(G) \to V(\Pk)$ is called a \emph{$\Pk$-collapsing map} if $\pi^{-1}(\{0\}) = \{s(G)\}$ and $\pi^{-1}(\{1\}) = \{t(G)\}$.

\begin{exam}[$\Pk$-collapsing map for diamonds] \label{ex:pidiamonds}
Let $k,m \geq 2$ be integers. Recall from Example~\ref{ex:diamonds} the diamond graph $\dia_{k,m}$ with vertex set $V(\dia_{k,m}) := V(\Pk) \times \{1,\dots m\} / \sim$, where $(u,i) \sim (v,j)$ if and only if $(u,i) = (v,j)$, $u = v = 0$, or $u = v = 1$. The map $\pi: V(\dia_{k,m})  \to V(\Pk)$ defined by $\pi([(u,i)]) := u$ is a $\Pk$-collapsing map. See Figure \ref{fig:pidiamonds}.
\end{exam}

\begin{figure}[H]
    \centering
   \includegraphics[trim=250 200 200 100,clip,scale=.425]{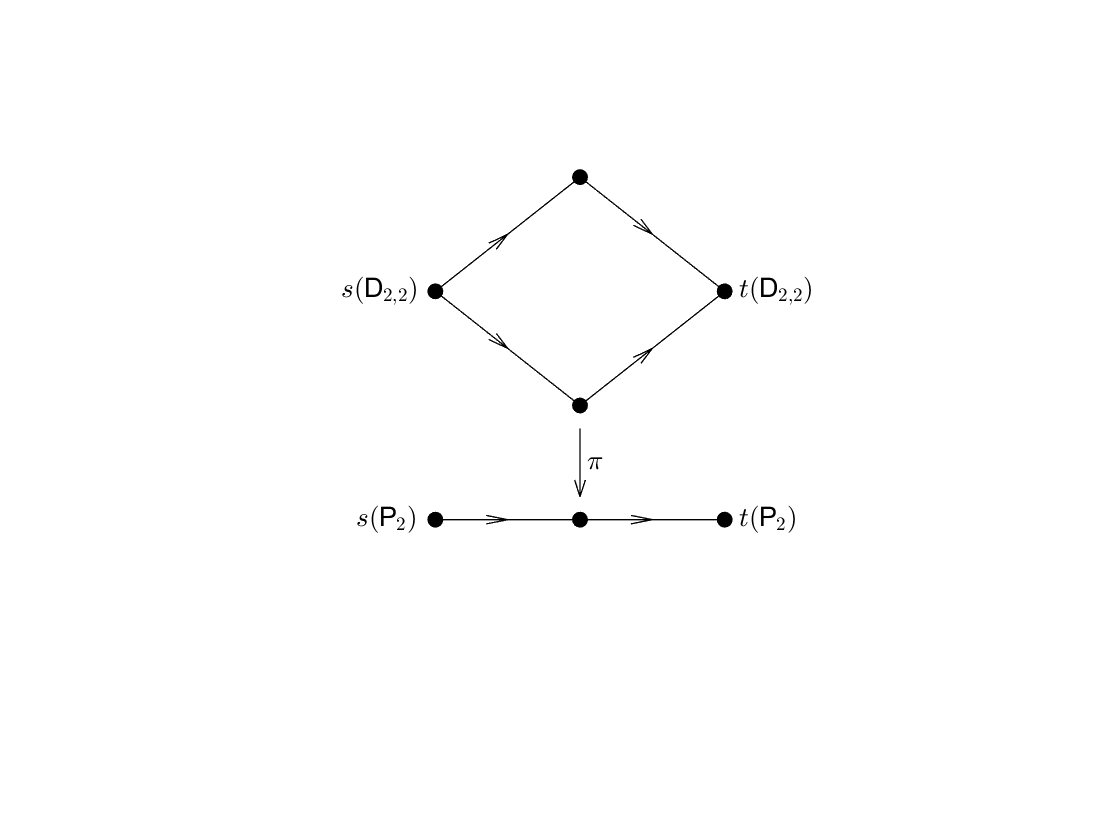}
   \includegraphics[trim=200 200 200 100,clip,scale=.425]{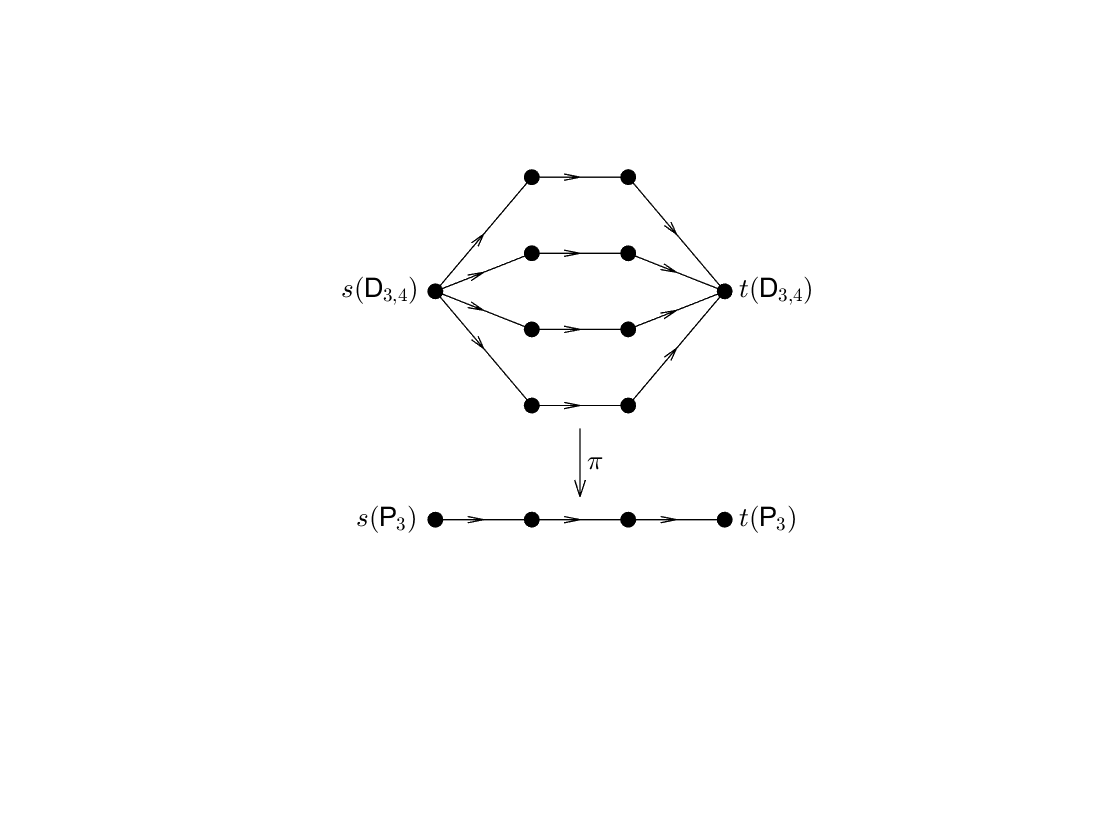}
       \caption{The $s$-$t$ graphs $\sD_{2,2}$ and $\sD_{3,4}$ and their $\Pk$-collapsing maps $\pi$.}
    \label{fig:pidiamonds}
\end{figure}

\begin{defi}
Let $H$ be a graph and $G$ an $s$-$t$ graph with $\Pk$-collapsing map $\pi$. Let $F_1 \subset \bR^{V(H)}$, $F_2 \subset \bR^{E(H)}$, $F_3 \subset \bR^{V(G)}$ be sets of functions with $f_3(s(G)) = f_3(t(G)) = 0$ for every $f_3 \in F_3$. Then we define the collection of functions $\scF(F_1,F_2,F_3) \subset \bR^{V(H \os G)}$ by
$$\scF(F_1,F_2,F_3) := ((id_{H} \os \pi)^* \circ \Lin)(F_1) \: \cup \: (F_2 \os F_3).$$
\end{defi}

\begin{figure}
    \centering
    \includegraphics[trim=300 100 250 125,clip,scale=.375]{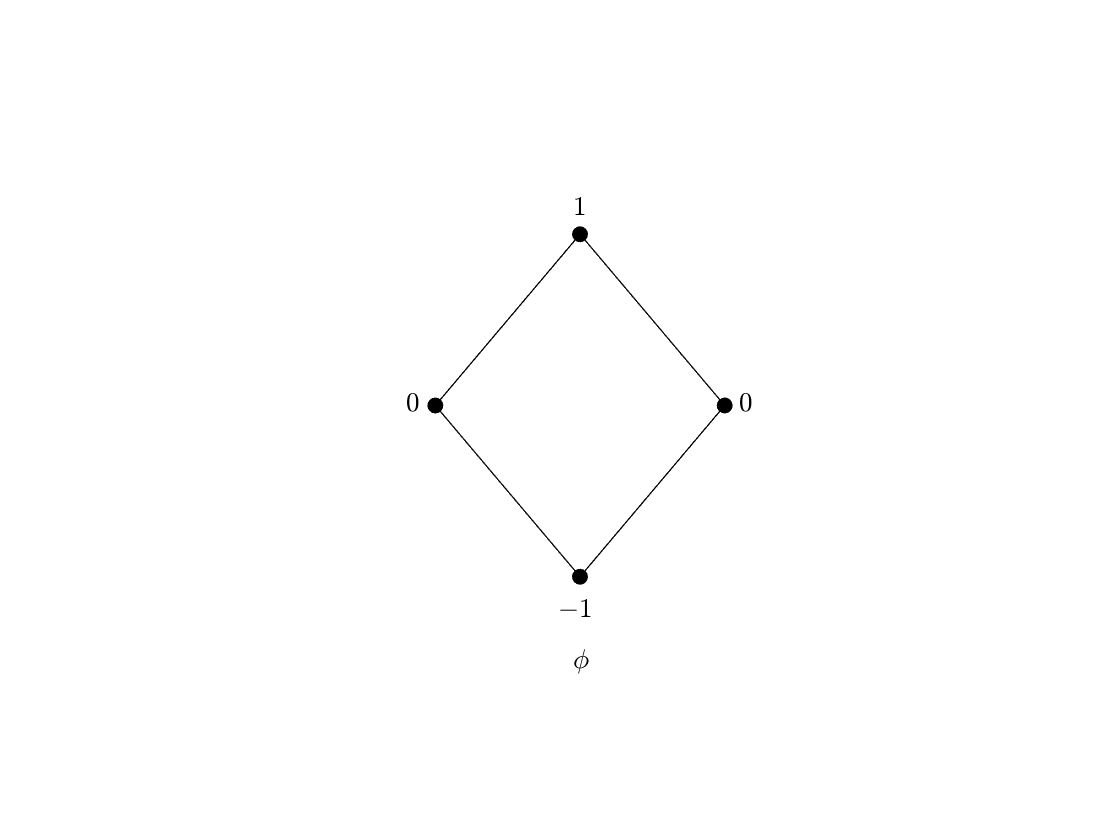}
   \includegraphics[trim=275 100 250 125,clip,scale=.375]{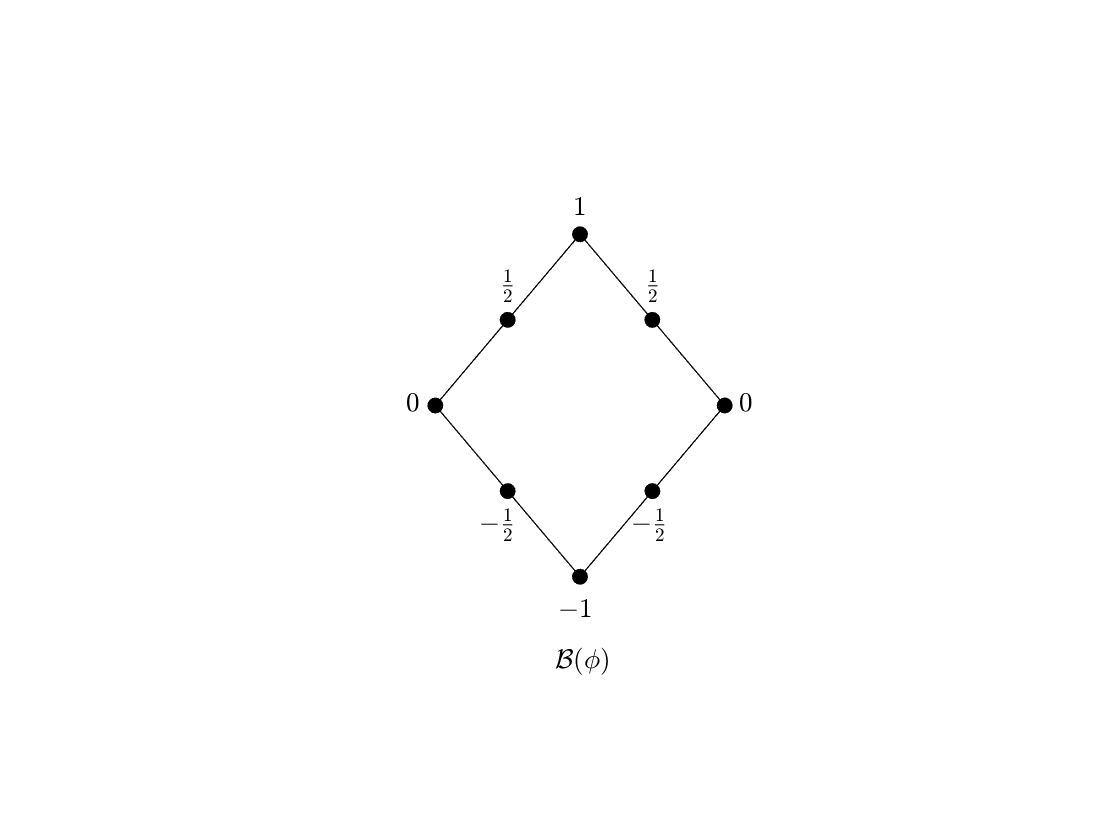}
    \includegraphics[trim=275 100 275 125,clip,scale=.375]{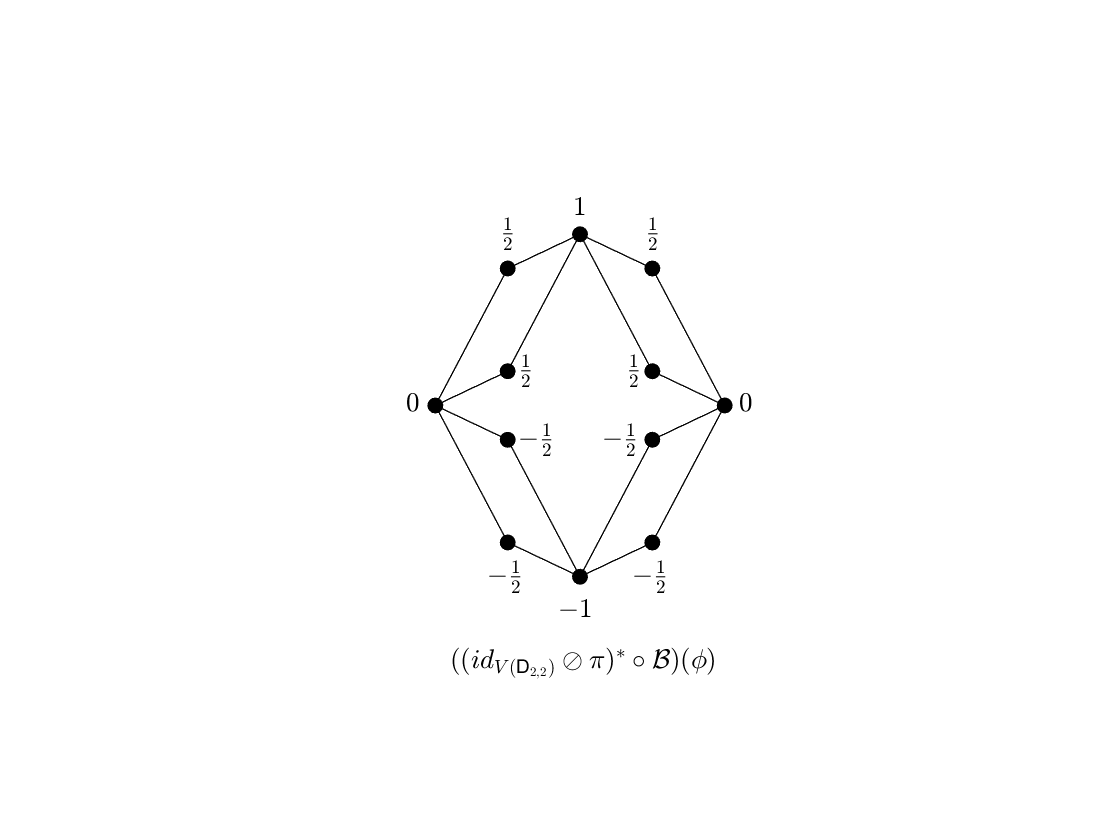} \\
    
    \includegraphics[trim=300 100 275 125,clip,scale=.33]{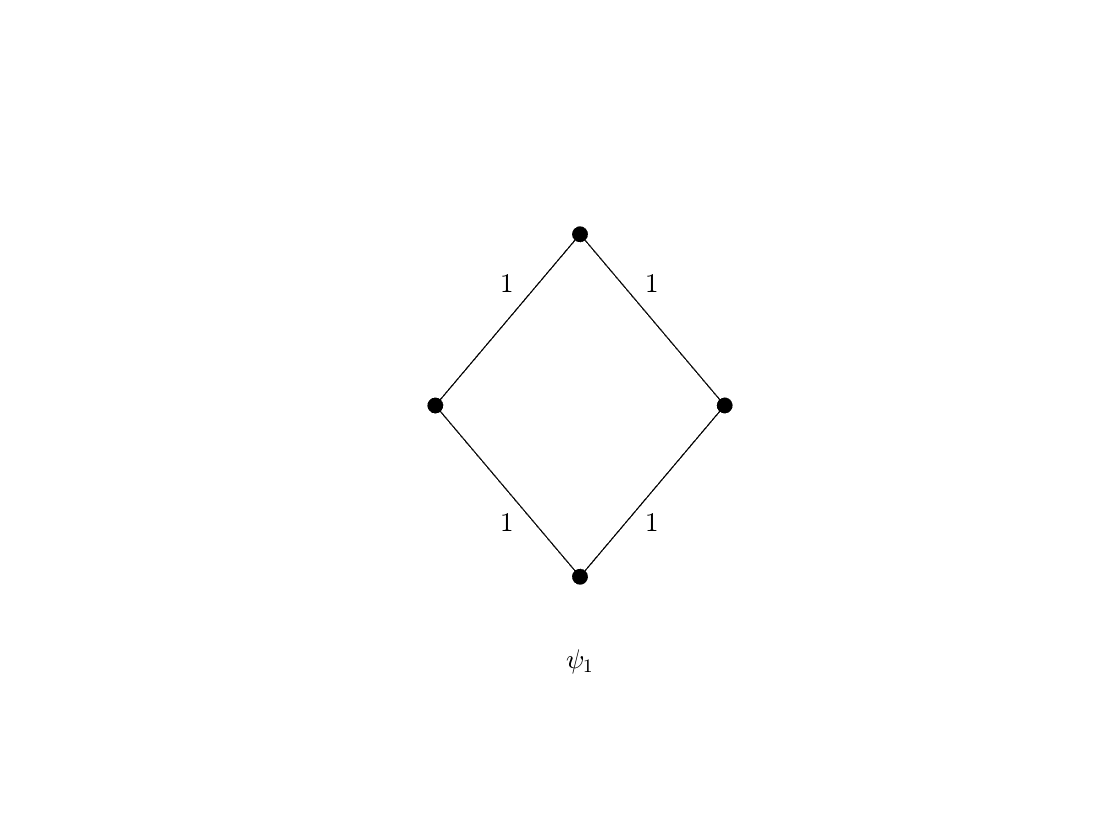}
    \includegraphics[trim=300 100 275 125,clip,scale=.33]{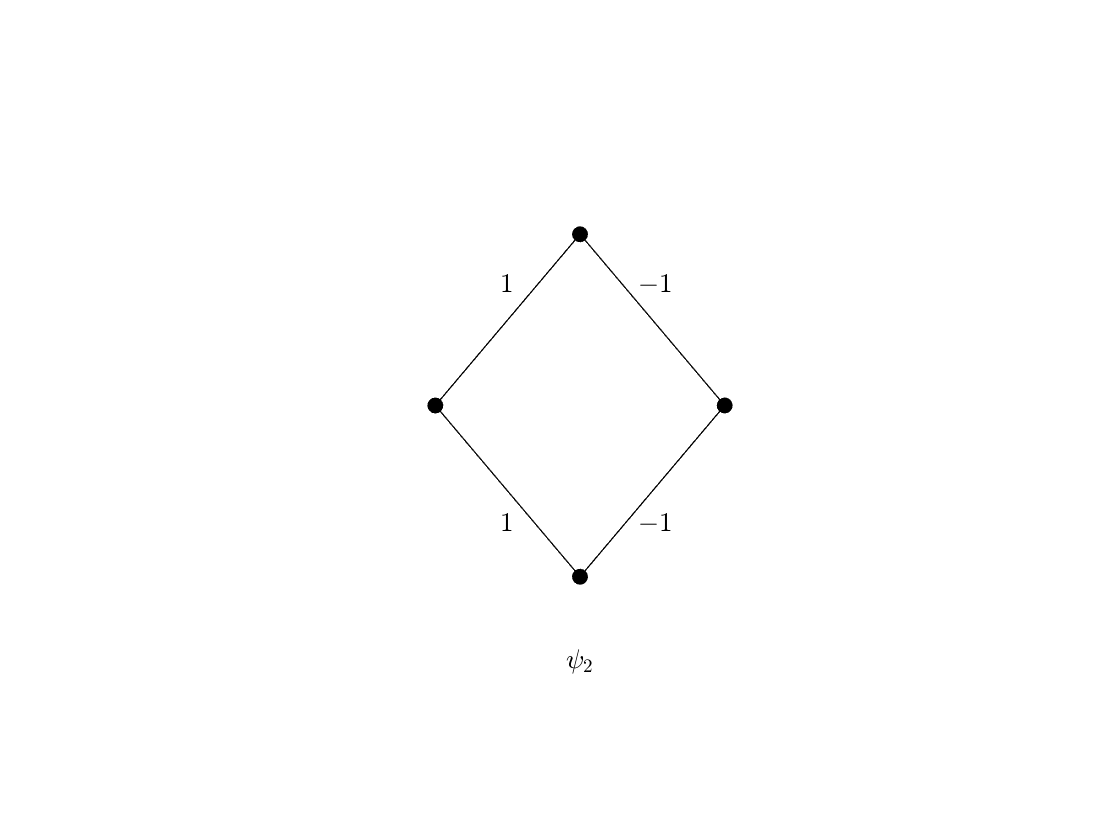}
    \includegraphics[trim=300 100 275 125,clip,scale=.33]{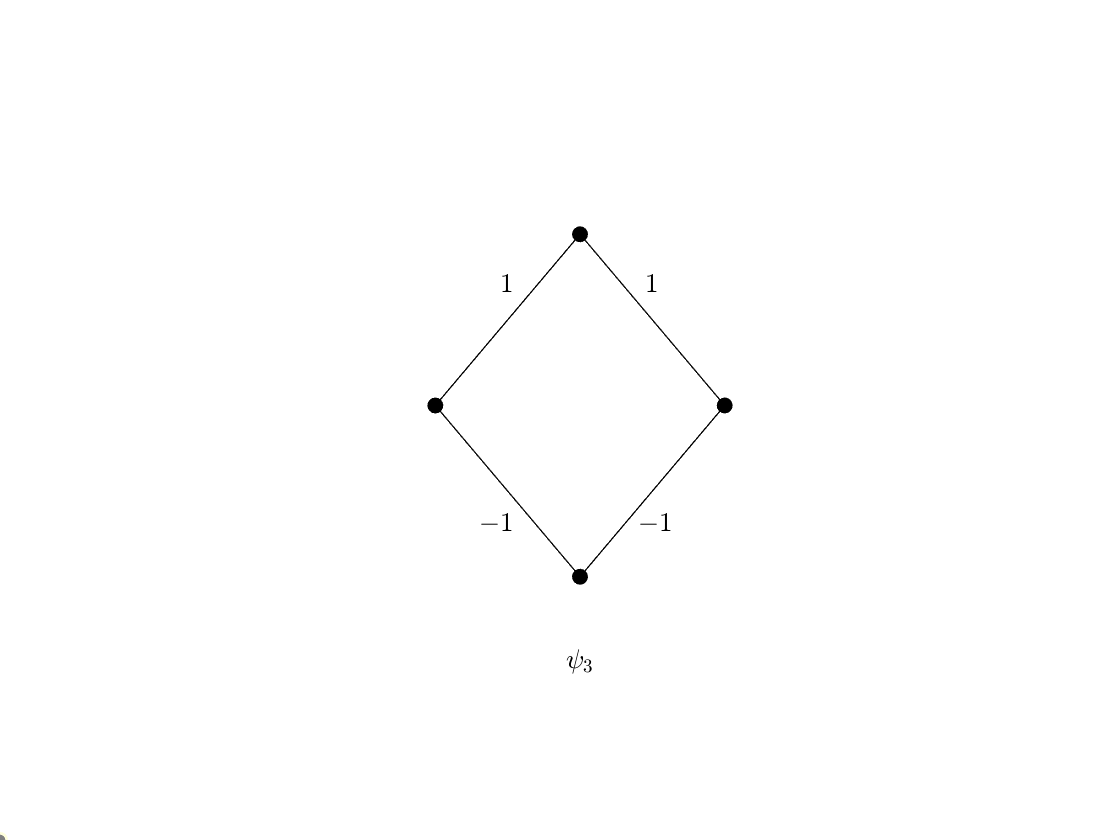}
    \includegraphics[trim=300 100 275 125,clip,scale=.33]{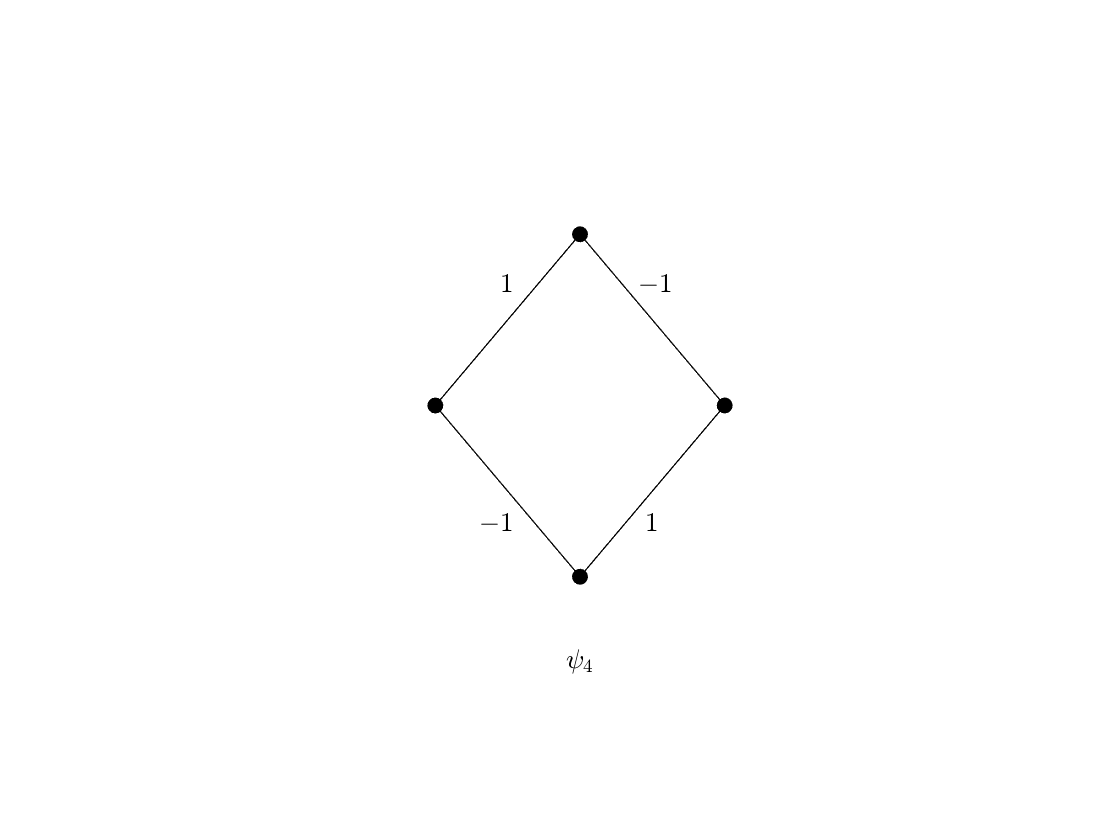} \\
    
    \includegraphics[trim=300 100 275 125,clip,scale=.33]{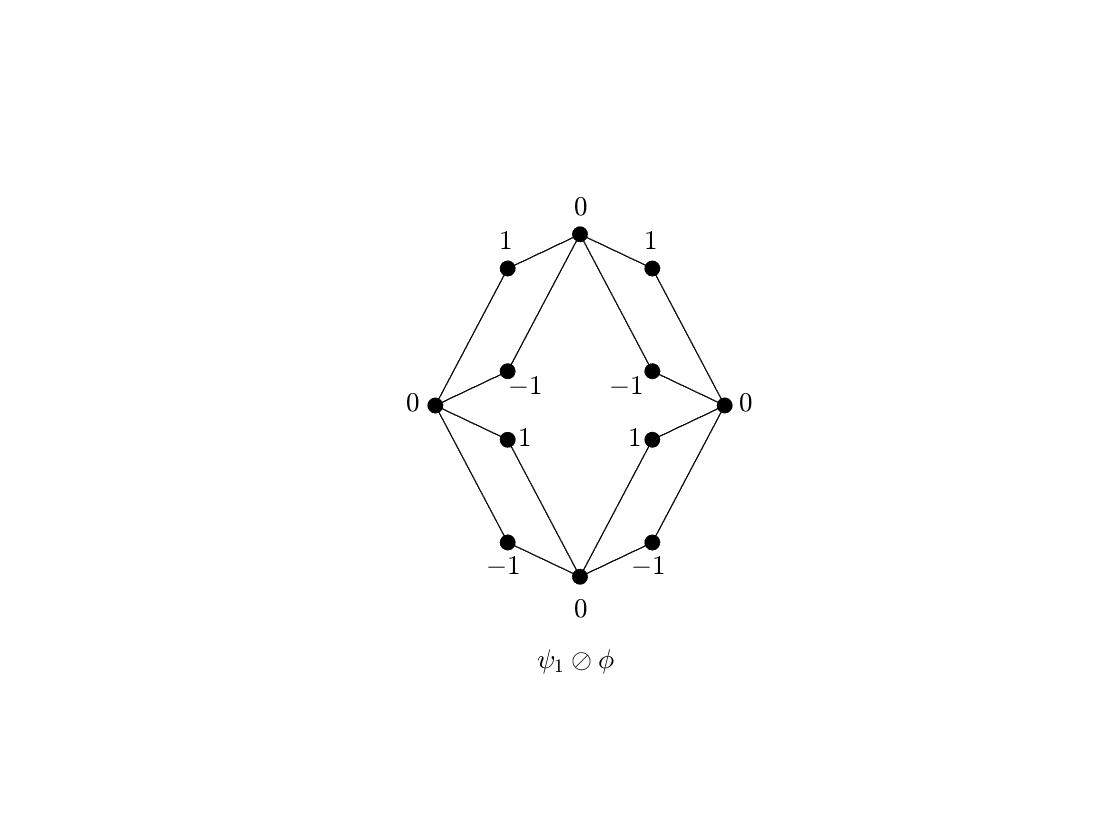}
    \includegraphics[trim=300 100 275 125,clip,scale=.33]{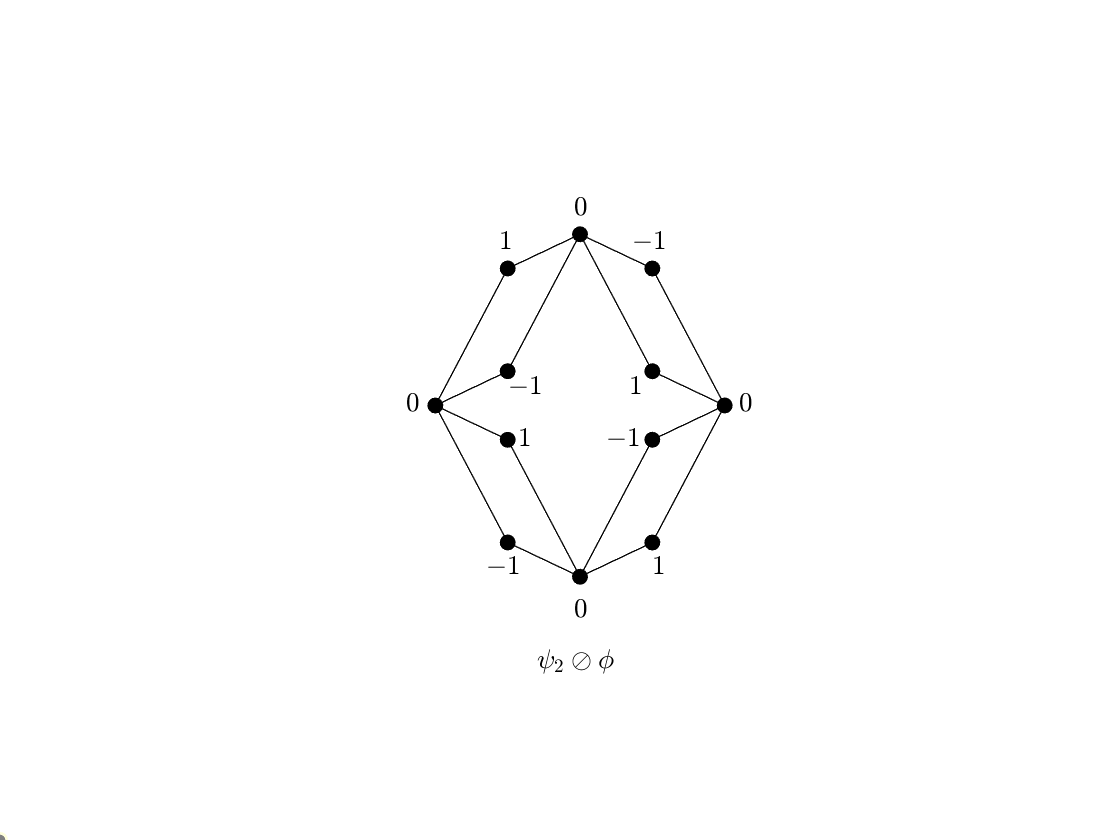}
    \includegraphics[trim=300 100 275 125,clip,scale=.33]{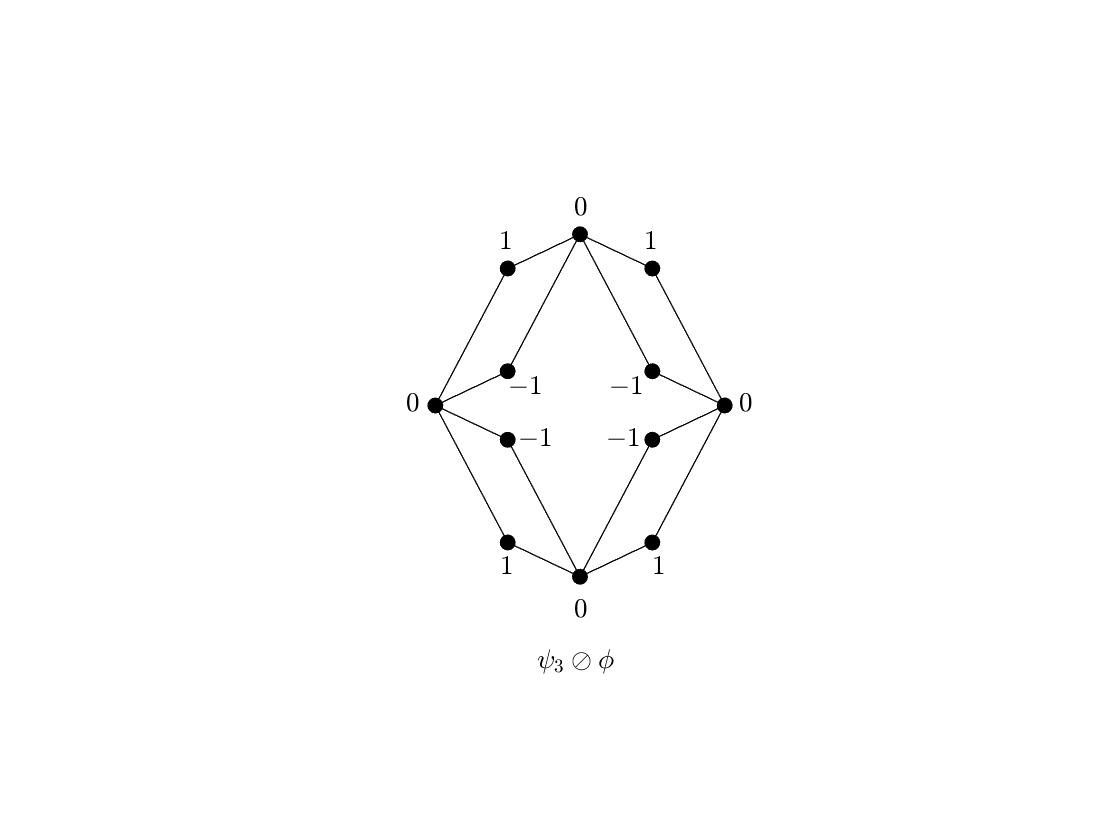}
    \includegraphics[trim=300 100 275 125,clip,scale=.33]{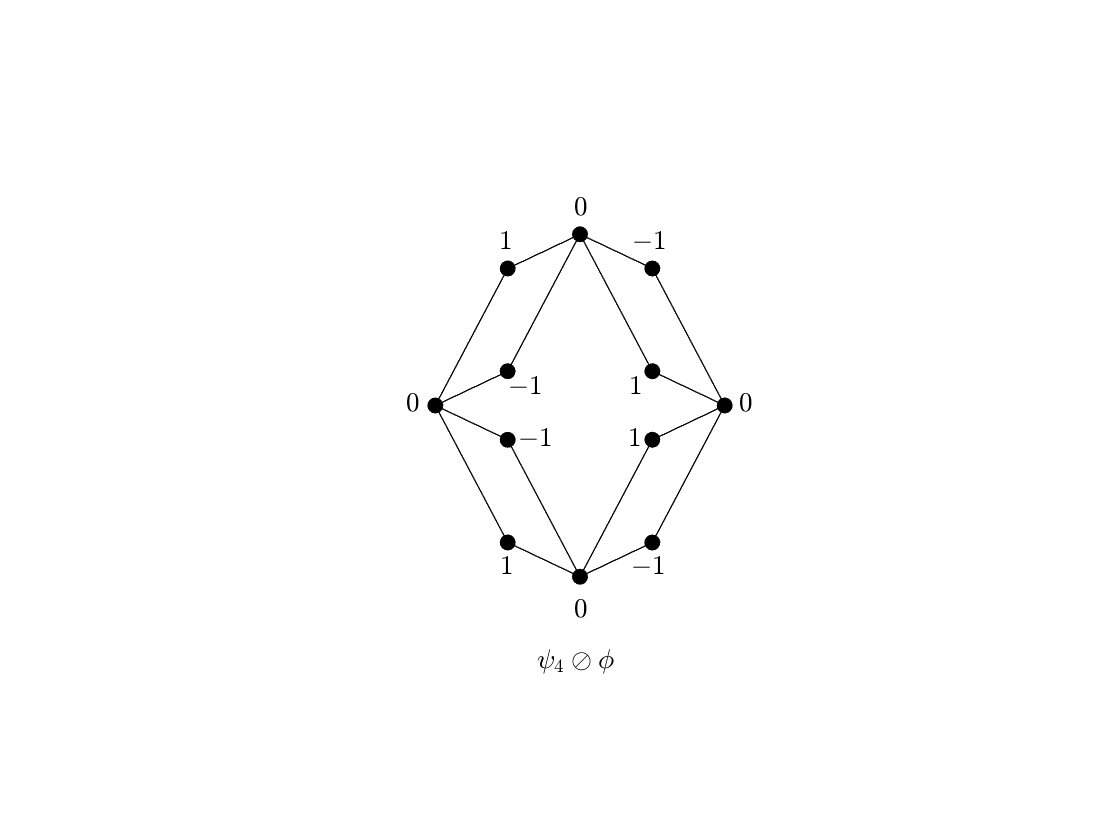}
    \caption{Construction of the set $\scF(F_1,F_2,F_3)$ from the sets $F_1 = \{\phi\}, F_2 = \{\psi_1,\psi_2,\psi_3,\psi_4\}$ and $F_3 = \{\phi\}$. The set $\scF(F_1,F_2,F_3)$ consists of the top right function $((id_{\sD_{2,2}}\os\pi)^*\circ\Lin)(\phi)$ and the bottom row of functions $\psi_1\os\phi,\psi_2\os\phi,\psi_3\os\phi,\psi_4\os\phi$.}
    \label{fig:slashfunctions}
\end{figure}

In the above definition, $(id_{H} \os \pi)^* \circ \Lin$ should be thought to transfer the set of functions $F_1$ on $V(H)$ to the set of functions $((id_{H} \os \pi)^* \circ \Lin)(F_1)$ on $V(H \os G)$ in a natural way that preserves $L_1$, $L_\infty$, and Lipschitz norms and also $(\Lin,\nu_H,\nu_{\Pk})$-orthogonality. The second set $F_2 \os F_3$ will be strongly orthogonal if $F_2$ and $F_3$ are each strongly orthogonal, and $F_2 \os F_3$ will be strongly orthogonal to $((id_{H} \os \pi)^* \circ \Lin)(F_1)$ if $F_3 \subset \ker(\bE_{\mu_\alpha(\nu_G)}^{\pi})$ for all $\alpha \in \Delta$. See Figure \ref{fig:slashfunctions} for an example when $H=G=\sD_{2,2}$. By repeating the procedure demonstrated in this figure, one may obtain orthogonal sets of functions $F_n \subset L_2(\sD_{2,2}^{\os n})$ witnessing the Lipschitz-spectral profile of $\sD_{2,2}^{\os n}$ having dimension 2, bandwidth $2^k$, and uniform control on the constants. Readers who are interested only in the diamond graphs $\sD_{2,2}^{\os n}$ may wish to provide these simpler details for themselves and avoid the technicalities presented in the remainder of the section (which are necessary for our result on general $\os$-products).

The next three theorems establish the precise facts needed to calculate the Lipschitz-spectral profile.
We save the proofs until the ensuing subsection. After stating the theorems, we give as a corollary a lower bound on the Lipschitz-spectral profile of $\os$-powers for certain graphs, such as the diamond graphs. 

\begin{theo}[Preservation of strong orthogonality] \label{thm:stronglyorthog}
Let $H$ be a graph, $G$ an $s$-$t$ graph with $\Pk$-collapsing map $\pi$, and $\nu_H,\nu_G$ measures on $E(H),E(G)$. Suppose that
\begin{itemize}
    \item $F_1 \subset \bR^{V(H)}$ is strongly $\nu_H$-orthogonal,
    \item $F_2 \subset \bR^{E(H)}$ is orthogonal in $L_2(E(H),\nu_H)$, and
    \item $F_3 \subset \bR^{V(G)}$ is strongly $\nu_G$-orthogonal and $(F_3)_{-},(F_3)_{+}, \subset \ker(\bE_{\nu_G}^{\pi})$.
\end{itemize}
Then $\scF(F_1,F_2,F_3) \subset \bR^{V(H \os G)}$ is strongly $\nu_{H}\os\nu_{G}$-orthogonal.
\end{theo}

Let $\vep_j := (\frac{j-1}{k}, \frac{j}{k})$ be the $j$th edge of $\Pk$. We say that a measure $\nu_{\Pk}$ on $E(\Pk)$ is \emph{reflection invariant} if $\nu_{\Pk}(\vep_j) = \nu_{\Pk}(\vep_{k-j+1})$ for every $1 \leq j \leq k$. It is easy to see that, if $\nu_{\Pk}$ is reflection invariant, then the induced measure $\mu(\nu_{\Pk})$ on $V(\Pk)$ is also reflection invariant in the sense that $\mu(\nu_{\Pk})(\tfrac{j}{k}) = \mu(\nu_{\Pk})(\tfrac{k-j}{k})$ for every $0 \leq j \leq k$.

For $H$ a graph and $F \subset \bR^{V(H)}$, we say that $F$ has the \emph{edge-sign property} if $f(e^-) \cdot f(e^+) \geq 0$ for every $f \in F$ and $\{e^-,e^+\} \in E(H)$. Whenever $\mu$ is a measure on a set $S$ and $p \in [1,\infty]$, we write $\inf \|F\|_{L_p(\mu)}$ and $\sup \|F\|_{L_p(\mu)}$ to denote $\inf_{f \in F} \|f\|_{L_p(\mu)}$ and $\sup_{f \in F} \|f\|_{L_p(\mu)}$, respectively.

\begin{theo}[Preservation of edge-sign property and $L_1,L_\infty$-norms] \label{thm:Lpnorms}
Suppose that $H,G,\pi,\nu_H,\nu_G$ are as in Theorem \ref{thm:stronglyorthog}. Suppose
\begin{itemize}
    \item $F_1 \subset \bR^{V(H)}$ is any subset,
    \item $F_2 \subset \bR^{E(H)}$ is any subset, and
    \item $F_3 \subset \bR^{V(G)}$ satisfies $F_3(s(G)) = F_3(t(G)) = \{0\}$.
\end{itemize}
Then
\begin{align}
\label{eq:thm2Linfty}
    \sup\|\scF(F_1,F_2,F_3)\|_{L_\infty(\nu_H\os\mu(\nu_G))} &= \max\{\sup\|F_1\|_{L_\infty(\mu(\nu_H))},\sup\|F_2\|_{L_\infty(\nu_H)} \cdot \sup\|F_3\|_{L_\infty(\mu(\nu_G))}\}.
\end{align}
Additionally, if $\pi_\#\nu_G$ is reflection invariant and if $F_1,F_3$ have the edge-sign property, then
\begin{align}
\label{eq:thm2edgesign}
    \scF(F_1,F_2,F_3) &\text{ has the edge-sign property,} \\
\label{eq:thm2L1}
    \inf\|\scF(F_1,F_2,F_3)\|_{L_1(\nu_H\os\mu(\nu_G))} &= \min\{\inf\|F_1\|_{L_1(\mu(\nu_H))},\inf\|F_2\|_{L_1(\nu_H)} \cdot \inf\|F_3\|_{L_1(\mu(\nu_G))}\}.
\end{align}

\end{theo}

\begin{theo}[Increase of Lipschitz growth function] \label{thm:Lipgrowth}
Suppose that $H,G,\pi$ are as in Theorem \ref{thm:stronglyorthog} and that $V(H),V(G)$ are equipped with geodesic metrics $\sd_H,\sd_G$. Equip $V(H \os G)$ with the $\os$-geodesic metric $\sd_{H}\os\sd_{G}$. Suppose that
\begin{itemize}
    \item $F_1 \subset \bR^{V(H)}$ is any subset,
    \item $F_2 \subset \bR^{E(H)}$ is any subset, and
    \item $F_3 \subset \bR^{V(G)}$ is any subset with $F_3(s(G)) = F_3(t(G)) = \{0\}$.
\end{itemize}
Then, for $s\ge 0$
\begin{equation*}
    \gamma_{\scF(F_1,F_2,F_3)}(s) \geq \gamma_{F_1}(s) + |F_2|\cdot\gamma_{F_3}\left(\frac{s}{\sup_{f_2 \in F_2}\sup_{e \in E(H)} \frac{|f_2(e)|}{\sd_H(e)}}\right).
\end{equation*}
\end{theo}

\begin{defi}[Base functions]
Let $G$ be an $s$-$t$ graph with $\Pk$-collapsing map $\pi$, and let $\nu_G$ be the uniform probability measure on $E(G)$. We say that $\phi: V(G) \to \bR$ is a \emph{base function of $G$} if $\phi$ has the edge-sign property and $\phi_-,\phi_+ \in \ker(\bE_{\nu_G}^\pi)$.
\end{defi}

\begin{rema}
Let $G,\pi,\nu_G$ be as in the previous definition. Although it won't be needed for our purposes, it can be checked that a nonzero base function on $G$ exists if and only if one of the following hold.
\begin{itemize}
    \item $|\pi^{-1}(\tfrac{j}{k})| \geq 3$ for some $\tfrac{j}{k} \in V(\Pk)$.
    \item $\pi^{-1}(\tfrac{j}{k}) = \{u_1,u_2\}$ for some $\tfrac{j}{k} \in V(\Pk)$ and $u_1 \neq u_2 \in V(G)$ with $\dfrac{\deg^+(u_1)}{\deg^+(u_2)} = \dfrac{\deg^-(u_1)}{\deg^-(u_2)}$, where $\deg^{\pm}(u) = |\{e \in E(G): e^{\pm} = u\}|$.
\end{itemize}
\end{rema}

The next corollary and the example following it are our main applications of the tools developed in this section.

\begin{coro}[Lipschitz-spectral profile of $\os$-powers] \label{cor:spec-slashpower}
Suppose that $G$ is an $s$-$t$ graph with $\Pk$-collapsing map $\pi$, $\nu_G$ is a probability measure on $E(G)$, and $\sd_G$ is a geodesic metric on $V(G)$. If
\begin{enumerate}
    \item $\nu_G$ is uniform,
    \item $\nu_{\Pk} := \pi_\# \nu_G$ is reflection invariant,
    \item $G$ admits a nonzero base function $\phi$, and
    \item $\sd_G(e) = \frac{1}{k}$ for every $e \in E(G)$,
\end{enumerate}
then, for every $n \geq 1$, $G^{\os n}$ has $(\sd_{G}^{\os n},\mu(\nu_{G}^{\os n}))$-Lipschitz-spectral profile of dimension $\frac{\log{|E(G)|}}{\log{k}}$ and bandwidth $k^n$, with constants $C_{L_1} \leq 2\frac{\|\phi\|_{L_\infty(\mu(\nu_G))}}{\|\phi\|_{L_1(\mu(\nu_G))}}$, $C_{L_\infty} \leq 1$, $C_{\gamma} \leq 2|E(G)|^2$.
\end{coro}

Note that, in the conclusion of the corollary, the dimension and constants $C_{L_1},C_{L_\infty},C_{\gamma}$ are independent of $n$ and that the bandwidth grows exponentially with $n$.

\begin{proof}
Assume $\nu_G,\nu_{\Pk},\sd_G$ are as above, and let $\tilde{\phi}$ be a nonzero base function. We prove the following stronger statement by induction: For every $n \geq 1$, there exists a set of functions $F_1^n \subset \bR^{V(G^{\os n})}$ satisfying:
\begin{enumerate}
    \item\label{item:slashpower1} $F_1^n$ has the edge-sign property.
    \item\label{item:slashpower2} $F_1^n$ is strongly $\nu_{G}^{\os n}$-orthogonal.
    \item\label{item:slashpower3}  $\left(2\frac{\|\tilde{\phi}\|_{L_\infty(\mu(\nu_G))}}{\|\tilde{\phi}\|_{L_1(\mu(\nu_G))}}\right)^{-1} \leq \inf\|F_1^n\|_{L_1(\mu(\nu_G^{\os n}))} \leq \sup\|F_1^n\|_\infty \leq 1$.
    \item\label{item:slashpower4}  $\gamma_{F_1^n}(k^m) \geq (2|E(G)|)^{-1}(k^m)^\frac{\log{|E(G)|}}{\log{k}}$ for every $1 \leq m \leq n$.
\end{enumerate}
By Remark~\ref{rem:strongorthog->orthog}, to prove the desired estimates on the Lipschitz-spectral profile, only orthogonality in $L_2(V(G^{\os n}),\mu(\nu_G^{\os n}))$ and not the full force of \eqref{item:slashpower2} is needed, and \eqref{item:slashpower1} is not needed at all. However, for the induction to close, we do need \eqref{item:slashpower1} and \eqref{item:slashpower2}.

Define $\phi := \frac{\tilde{\phi}}{\|\tilde{\phi}\|_\infty}$. Then we have
\begin{itemize}
    \item $\|\phi\|_\infty = 1$,
    \item $\phi_-,\phi_+ \in \ker(\bE_{\nu_G}^{\pi})$,
    \item $\|\phi\|_{L_1(\mu(\nu_G))} = \frac{\|\tilde{\phi}\|_{L_1(\mu(\nu_G))}}{\|\tilde{\phi}\|_\infty}$, and
    \item $\phi$ has the edge-sign property.
\end{itemize}
Note that the edge-sign property, $\|\phi\|_\infty = 1$, and $\sd_G(e) = \frac{1}{k}$ for all $e \in E(G)$ together imply
\begin{itemize}
    \item $\Lip(\phi) \leq k$.
\end{itemize}

We now begin the inductive proof. The base case $n=1$ is satisfied by $F_1^1 = \{\phi\}$. Let $n \geq 2$, and assume that the statement holds for $n-1$. Let $F_1^{n-1} \subset \bR^{V(G^{\os n-1})}$ be a set of functions satisfying \eqref{item:slashpower1}-\eqref{item:slashpower4} given by the induction hypothesis. Let $F_2 \subset L_2(E(G^{\os n-1}),\nu_G^{\os n-1})$ be an orthogonal subset such that $\sup\|F_2\|_{\infty} \leq 1$, $\inf\|F_2\|_{1} \geq \frac{1}{2}$, and $|F_2| \geq \frac{1}{2}|E(G^{\os n-1})|$. Such a set exists by uniformity of $\nu_G^{\os n-1}$ and by Sylvester's construction of Hadamard matrices (see Lemma~\ref{lem:Hadamard}). Then we define $F_1^n := \scF(F_1^{n-1},F_2,\{\phi\}) \subset \bR^{V(G^{\os n-1} \os G)}$. By Theorem \ref{thm:stronglyorthog} and the inductive hypothesis, $F_1^n$ is strongly $\nu_{G}^{\os n}$-orthogonal, verifying \eqref{item:slashpower2}. By Theorem \ref{thm:Lpnorms} and the inductive hypothesis, $F_1^n$ has the edge-sign property, verifying \eqref{item:slashpower1}. We now verify \eqref{item:slashpower3}-\eqref{item:slashpower4}.

By \eqref{eq:decomposition2}, Theorem \ref{thm:Lpnorms}, and the inductive hypothesis,
\begin{align*}
    &\inf \|F_1^n\|_{L_1(\mu(\nu_G^{\os n}))} = \inf \|F_1^n\|_{L_1(\nu_G^{\os n-1}\os\mu(\nu_G))} \\
    &\hspace{.5in} = \min\{\inf\|F_1^{n-1}\|_{L_1(\mu(\nu_G^{\os n-1}))}, \inf\|F_2\|_{L_1(\nu_G^{\os n-1})} \cdot \|\phi\|_{L_1(\mu(\nu_G))}\} \geq \left(2\frac{\|\tilde{\phi}\|_\infty}{\|\tilde{\phi}\|_{L_1(\mu(\nu_G))}}\right)^{-1} \\
    &\sup \|F_1^n\|_\infty = \max\{\sup\|F_1^{n-1}\|_{\infty},\sup\|F_2\|_{\infty} \cdot \|\phi\|_{\infty}\} = 1,
\end{align*}
verifying \eqref{item:slashpower3}.

Finally, we verify \eqref{item:slashpower4}. By Theorem \ref{thm:Lipgrowth}, the facts that $|F_2| \geq \frac{1}{2}|E(G^{\os n-1})| = \frac{1}{2}|E(G)|^{n-1}$,  $\sup\|F_2\|_\infty \leq 1$, and $\Lip(\phi) \leq k$, and the inductive hypothesis applied to \eqref{item:slashpower4} for $F_1^{n-1}$, we get, for any $1 \leq m \leq n$,
\begin{align*}
    \gamma_{F_1^{n}}(k^m) &\geq \gamma_{F_1^{n-1}}(k^m) + |F_2|\cdot\gamma_{\{\phi\}}\left(\frac{k^m}{k^{n-1}\sup\|F_2\|_\infty}\right) \\
    &\geq \left\{\begin{matrix} \gamma_{F_1^{n-1}}(k^m) & m \leq n-1 \\ \frac{1}{2}|E(G)|^{n-1} & m = n \end{matrix}\right. \\
    &\geq \left\{\begin{matrix} 
    (2|E(G)|)^{-1}(k^m)^{\frac{\log{|E(G)|}}{\log{k}}} & m \leq n-1 \\ \frac{1}{2}|E(G)|^{n-1} & m=n \end{matrix}\right. \\
    &= (2|E(G)|)^{-1}(k^m)^{\frac{\log{|E(G)|}}{\log{k}}}.
\end{align*}
\end{proof}

We now apply our machinery to compute the Lipschitz-spectral profile of diamond graphs.
Let $k,m \geq 2$ be integers. Recall from Example~\ref{ex:diamonds} the definition of the diamond graph $\sD_{k,m}$, with $\nu_{\sD_{k,m}}$ the uniform probability measure on $E(\sD_{k,m})$ and $\sd_{\sD_{k,m}}$ the normalized geodesic metric on $V(\sD_{k,m})$, and recall the $\Pk$-collapsing map $\pi: V(\sD_{k,m}) \to V(\Pk)$ from Example~\ref{ex:pidiamonds}. It is clear that $\nu_{\Pk} := \pi_\#\nu_{\sD_{k,m}}$ is the uniform probability measure on $E(\Pk)$, and hence is reflection-invariant. Furthermore, we may define a base function $\phi: V(\sD_{k,m}) = V(\Pk) \times \{1,\dots m\}/\sim \to \bR$ by
\begin{align*}
    \phi([(u,i)]) \eqd \begin{cases} 1 & i \:\: \mathrm{odd,} \:\: i<m, \:\: u \not\in\{s(\sD_{k,m}), t(\sD_{k,m})\} \\ -1 & i \:\: \mathrm{even,} \:\: u \not\in\{s(\sD_{k,m}), t(\sD_{k,m})\} \\ 0 & \mathrm{otherwise}\end{cases}.
\end{align*}
See a picture of $\phi$ for $\sD_{2,2}$ in the top left corner of Figure~\ref{fig:slashfunctions}.

The following can be directly computed.
\begin{itemize}
    \item $\phi$ has the edge-sign property.
    \item $\phi_{-},\phi_+ \in \ker(\bE_{\nu_{\sD_{k,m}}}^{\pi})$.
    \item $\|\phi\|_{\infty} = 1$.
    \item $\|\phi\|_{1} = \frac{2(k-1)\cdot2\lfloor \frac{m}{2}\rfloor}{2km} \geq \frac{1}{3}$.
\end{itemize}
Hence, by Corollary~\ref{cor:spec-slashpower}, we obtain:

\begin{coro}\label{ex:spec-diamonds}
The diamond graph $\dia_{k,m}^{\os n}$ has $(\sd_{\dia_{k,m}}^{\os n},{\nu_{\dia_{k,m}}^{\os n}})$-Lipschitz-spectral profile of dimension $1+\frac{\log{m}}{\log{k}}$, bandwidth $k^n$, and constants $C_{L_1} \leq 6$, $C_{L_\infty} \leq 1$, $C_{\gamma} \leq 2k^2m^2$.
\end{coro}

\subsection{Supporting propositions and lemmas}
In this subsection, we prove a host of supporting lemmas and propositions. Each proposition is directly used in the next subsection to prove the main theorems (Theorems~\ref{thm:stronglyorthog}~,~\ref{thm:Lpnorms}~,~\ref{thm:Lipgrowth}), and each lemma is used in the proof of one of the propositions. These results illustrate how our various operators commute with each other and behave with respect to $L_1$, $L_\infty$, and Lipschitz norms and strong orthogonality.

We begin with a set of three propositions pertaining to the induced edge-function operators that are used in the proof of Theorem~\ref{thm:stronglyorthog}. The first two, Propositions \ref{prop:epsslashcommute} and \ref{prop:epspullbackcommute}, can be viewed as stating that induced edge-function operators $(\cdot)_\pm: \bR^{V(H')} \to \bR^{E(H')}$ commute with pre-$\os$ operators $h \os (\cdot): \bR^{V(G)} \to \bR^{V(H\os G)}$ and with pullback operators $\theta^*$.

\begin{prop} \label{prop:epsslashcommute}
For every graph $H$, $s$-$t$ graph $G$, functions $h: E(H) \to \bR$, $g: V(G) \to \bR$ with $g(s(G)) = g(t(G)) = 0$, and $\eps \in \{-,+\}$,
\begin{equation*}
    (h \os g)_\eps = h \os g_\eps.
\end{equation*}
\end{prop}

\begin{proof}
Let $H,G,h,g,\eps$ be as above. Let $e_1 \os e_2 \in E(H \os G)$. Then
\begin{align*}
    (h \os g)_\eps(e_1 \os e_2) &= (h \os g)((e_1 \os e_2)^\eps) = (h \os g)(e_1 \os e_2^\eps) \\
    &= h(e_1) \cdot g(e_2^\eps) = h(e_1) \cdot g_\eps(e_2) = (h \os g_\eps)(e_1 \os e_2).
\end{align*}
\end{proof}

\begin{prop} \label{prop:epspullbackcommute}
Let $\theta: V(G) \to V(G')$ be a graph morphism between graphs. For every $f: V(G') \to \bR$ and $\eps \in \{-,+\}$,
\begin{equation*}
    \theta^*(f)_\eps = \theta^*(f_\eps).
\end{equation*}
\end{prop}

\begin{proof}
Let $f,\eps$ be as above. Let $e \in E(G)$. Then we have
\begin{align*}
    \theta^*(f)_\eps(e) = \theta^*(f)(e^\eps) = f(\theta(e^\eps)) = f(\theta(e)^\eps) = f_\eps(\theta(e)) = \theta^*(f_\eps)(e).
\end{align*}
\end{proof}

The third proposition on induced edge-function operators illustrates how certain inner-products of $\Lin(f)_{\pm}$ and $\Lin(f')_{\pm}$ can be expressed as linear combinations of $f_{\pm}f'_{\pm}$, $f_{\pm}f'_{\mp}$. This proposition easily implies the fact that the barycentric extension operator $\Lin$ preserves strong orthogonality, which is crucial to the proof of Theorem~\ref{thm:stronglyorthog}.

\begin{prop} \label{prop:linearcombo}
For every graph $H$, measure $\nu_{\Pk}$ on $E(\Pk)$, $f,f': V(H) \to \bR$, and $\eps,\eps' \in \{-,+\}$, there exist scalars $c_1,c_2,c_3,c_4 \in \bR$ such that, for every $e_1 \in E(H)$,
\begin{align*}
    \int_{E(\Pk)} (\Lin(f)_\eps \Lin(f')_{\eps'})(e_1\os e_2) d\nu_{\Pk}(e_2) = (c_1f_-f'_- + c_2f_-f'_+ + c_3f_+f'_- + c_4f_+f'_+)(e_1).
\end{align*}
\end{prop}

\begin{proof}
Let $H,\nu_{\Pk},f,f',\eps,\eps'$ be as above. We will show the proof in the case $\eps = -$ and $\eps' = +$. The other cases can be treated similarly. For any $e_1 \in E(H)$, we have
\begin{align*}
    &\int_{E(\Pk)} (\Lin(f)_- \Lin(f')_+)(e_1\os e_2) d\nu_{\Pk}(e_2) \\
    &= \sum_{i=1}^k ((1-\tfrac{i-1}{k})f(e_1^-) + \tfrac{i-1}{k}f(e_1^+)) ((1-\tfrac{i}{k})f'(e_1^-) + \tfrac{i}{k}f'(e_1^+)) \nu_{\Pk}((\tfrac{i-1}{k},\tfrac{i}{k})) \\
    &= \left(\sum_{i=1}^k (1-\tfrac{i-1}{k})(1-\tfrac{i}{k})\nu_{\Pk}((\tfrac{i-1}{k},\tfrac{i}{k}))\right)f(e_1^-)f'(e_1^-) + \left(\sum_{i=1}^k (1-\tfrac{i-1}{k})(\tfrac{i}{k})\nu_{\Pk}((\tfrac{i-1}{k},\tfrac{i}{k}))\right)f(e_1^-)f'(e_1^+) \\
    &\hspace{.3in}+ \left(\sum_{i=1}^k (\tfrac{i-1}{k})(1-\tfrac{i}{k})\nu_{\Pk}((\tfrac{i-1}{k},\tfrac{i}{k}))\right)f(e_1^+)f'(e_1^-) + \left(\sum_{i=1}^k (\tfrac{i-1}{k})(\tfrac{i}{k})\nu_{\Pk}((\tfrac{i-1}{k},\tfrac{i}{k}))\right)f(e_1^+)f'(e_1^+) \\
    &= (c_1f_-f'_- + c_2f_-f'_+ + c_3f_+f'_- + c_4f_+f'_+)(e_1).
\end{align*}
\end{proof}

We require one more proposition to be used in the proof of Theorem~\ref{thm:stronglyorthog}. It shows a commutation relation between pre-$\os$ operators and conditional expectations.

\begin{prop} \label{prop:condexpslashcommute}
Let $H$ be a graph, $\theta: V(G) \to V(G')$ an $s$-$t$ graph morphism between $s$-$t$ graphs, and $\nu_H,\nu_G$ measures on $E(H),E(G)$. Then for any $h: E(H) \to \bR$ and $g: E(G) \to \bR$,
\begin{equation*}
    \bE_{\nu_H\os\nu_G}^{id_H\os\theta}(h\os g) = h\os\bE_{\nu_G}^{\theta}(g).
\end{equation*}
\end{prop}

\begin{proof}
Let $h,g$ be as above. We need to show that $h\os\bE_{\nu_G}^{\theta}(g)$ is $\sigma(id_H\os\theta)$-measurable and satisfies
\begin{equation}\label{eq:condexp}
    \int_{E(H\os G)} \phi \cdot (h\os\bE_{\nu_G}^{\theta}(g)) d(\nu_H\os\nu_G) = \int_{E(H\os G)} \phi \cdot (h\os g) d(\nu_H\os\nu_G)
\end{equation}
for every $\sigma(id_H\os\theta)$-measurable $\phi: E(H\os G) \to \bR$. Since $\bE_{\nu_G}^{\theta}(g)$ is $\sigma(\theta)$-measurable, there exists $f': E(G') \to \bR$ such that $\bE_{\nu_G}^{\theta}(g) = \theta^*(f')$. It is immediate to check that $(id_H\os\theta)^*(h\os f') = h\os \theta^{-1}(f')$, which shows that $h\os\bE_{\nu_G}^{\theta}(g)$ is $\sigma(id_H\os\theta)$-measurable.

Finally we verify \eqref{eq:condexp}. Let $(id_H\os\theta)^*(f): E(H\os G) \to \bR$ be an arbitrary $\sigma(id_H\os\theta)$-measurable function. For each $e_1 \in E(H)$, define $f^{e_1} : E(G') \to \bR$ by $f^{e_1}(e_2) := f(e_1\os e_2)$. It is immediate to check that for every $e_1\os e_2 \in E(H\os G)$, $(id_H\os\theta)^*(f)(e_1\os e_2) = \theta^*(f^{e_1})(e_2)$. Then we have
\begin{align*}
    \int_{E(H\os G)}& (id_H\os\theta)^*(f) \cdot (h\os\bE_{\nu_G}^{\theta}(g)) d(\nu_H\os\nu_G) \\
    &\overset{\eqref{eq:prod-measure}}{=} \int_{E(H)} \int_{E(G)} ((id_H\os\theta)^*(f) \cdot (h\os\bE_{\nu_G}^{\theta}(g)))(e_1\os e_2) d\nu_G(e_2) d\nu_H(e_1) \\
    &= \int_{E(H)} h(e_1) \int_{E(G)} (\theta^*(f^{e_1}) \cdot \bE_{\nu_G}^{\theta}(g))(e_2) d\nu_G(e_2) d\nu_H(e_1) \\
    &= \int_{E(H)} h(e_1) \int_{E(G)} (\theta^*(f^{e_1}) \cdot g)(e_2) d\nu_G(e_2) d\nu_H(e_1) \\
    &= \int_{E(H)} \int_{E(G)} ((id_H\os\theta)^*(f) \cdot (h\os g))(e_1\os e_2) d\nu_G(e_2) d\nu_H(e_1) \\
    &\overset{\eqref{eq:prod-measure}}{=} \int_{E(H\os G)} (id_H\os\theta)^*(f) \cdot (h\os g) d(\nu_H\os\nu_G).
\end{align*}
\end{proof}

The second set of propositions shows how $L_1$, $L_\infty$, and Lipschitz norms are affected by $\os$-operators, $\Lin$-operators, and pullback operators. They will be used in the proofs of Theorems~\ref{thm:Lpnorms} and \ref{thm:Lipgrowth}.

\begin{prop} \label{prop:slashprops}
Let $H$ be a graph, $G$ an $s$-$t$ graph, $\nu_H,\mu_G$ measures on $E(H),V(G)$, and $\sd_H,\sd_G$ geodesic metrics on $V(H),V(G)$. Equip $V(H \os G)$ with the $\os$-measure $\nu_H\os\mu_G$, and equip $V(H \os G)$ with the $\os$-geodesic metric $\sd_{H} \os \sd_{G}$. Then for every $h: E(H) \to \bR$ and $g \colon V(G) \to \bR$ with $g(s(G)) = g(t(G)) = 0$, the following holds.
	\begin{itemize}
    		\item $\|h \os g\|_\infty = \|h\|_\infty\|g\|_\infty$.
		\item $\Lip(h \os g) = \sup_{e \in E(H)} \left|h(e)\right|\sd_H(e)^{-1}\Lip(g)$.
		\item $\|h \os g\|_{L_1(\nu_H\os\mu_G)} = \|h\|_{L_1(\nu_H)}\|g\|_{L_1(\mu_G)}$.
	\end{itemize}
\end{prop}

\begin{proof}
Let $h,g$ be as above. The first item is obvious. For the second, let $e_1 \os  e_2 \in E(H \os G)$. Then we have
\begin{align*}
    |\nabla (h \os g)(e_1 \os e_2)| & = (\sd_{H} \os \sd_{G})(e_1 \os e_2)^{-1} |(h \os g)((e_1 \os e_2)^+) - (h \os g)((e_1 \os e_2)^-)|\\
    &= \sd_H(e_1)^{-1}\sd_G(e_2)^{-1}|h(e_1)g(e_2^+) - h(e_1)g(e_2^-)| \\
    &= |h(e_1)|\sd_{H}(e_1)^{-1} |\nabla (g)(e_2)|.
\end{align*}
Since $e_1 \os e_2 \in E(H \os G)$ was arbitrary, the conclusion follows by taking the supremum of each side. The third item follows immediately from \eqref{eq:slash-measure} and the definition of $h \os g$.
\end{proof}

The next proposition on preservation of $L_p$ norms follows more or less immediately from the definition of pushforward measure $\pi_\#\nu_G$ and the defining property of graph morphisms $\theta$. It is used in the proof of Theorem~\ref{thm:Lpnorms}.

\begin{prop} \label{prop:measurepreserve}
Let $H$ be a graph, $\theta: V(G) \to V(G')$ an $s$-$t$ graph morphism between $s$-$t$ graphs, and $\nu_H,\nu_G$ measures on $E(H),E(G)$. Then for every $f: V(H \os G') \to \bR$ and $p \in [1,\infty]$,
\begin{equation*}
    \|(id_H\os\theta)^*(f)\|_{L_p(\nu_H\os\mu(\nu_G))} = \|f\|_{L_p(\nu_H\os\mu(\theta_\#\nu_G))}.
\end{equation*}
\end{prop}

\begin{proof}
Let $g: V(H \os G') \to \bR$ be any function. For each $e \in E(H)$, define the contraction of $g$ along $e$ by $g^e: V(G') \to \bR$ by $g^e(u) := g(e\os u)$. The conclusion of the proposition follows by choosing $g = |f|^p$ (for $p < \infty$, the conclusion is obvious for $p=\infty$) and applying the following calculation:
\begin{align*}
    \int_{V(H\os G)} (id_H\os\theta)^*(g) d(\nu_H\os\mu(\nu_G)) &\overset{\eqref{eq:slash-measure}}{=} \int_{E(H)} \int_{V(G)} g(e\os\theta(u)) d\mu(\nu_G)(u) d\nu_H(e) \\
    &\overset{\eqref{eq:induced-integral}}{=} \int_{E(H)} \int_{E(G)} \frac{g(e\os\theta(e_1^-)) + g(e\os\theta(e_1^+))}{2} d\nu_G(e_1) d\nu_H(e) \\
    &= \int_{E(H)} \int_{E(G)} \frac{g^e(\theta(e_1)^-) + g^e(\theta(e_1)^+)}{2} d\nu_G(e_1) d\nu_H(e) \\
    &= \int_{E(H)} \int_{E(G)} \frac{\theta^*(g^e_-)(e_1) + \theta^*(g^e_+)(e_1)}{2} d\nu_G(e_1) d\nu_H(e) \\
    &= \int_{E(H)} \int_{E(G')} \frac{g^e_-(e_1) + g^e_+(e_1)}{2} d\theta_\#\nu_G(e_1) d\nu_H(e) \\
    &\overset{\eqref{eq:induced-integral}}{=} \int_{E(H)} \int_{V(G')} g^e(u) d\mu(\theta_\#\nu_G)(u) d\nu_H(e) \\
    &= \int_{E(H)} \int_{V(G')} g(e\os u) d\mu(\theta_\#\nu_G)(u) d\nu_H(e) \\
    &\overset{\eqref{eq:slash-measure}}{=} \int_{V(H\os G')} g d(\nu_H\os\mu(\theta_\#\nu_G)).
\end{align*}
\end{proof}

The next lemma states that barycentric extension operators preserve expectations when $\nu_{\Pk}$ is reflection invariant. It is only used to prove Proposition~\ref{prop:Linprops}, which in turn is used in the proofs of Theorems~\ref{thm:Lpnorms} and \ref{thm:Lipgrowth}.

\begin{lemm} \label{lem:intLin}
Let $H$ be a graph and $\nu_H$ a measure on $E(H)$. Then for any reflection invariant probability measure $\mu_{\Pk}$ on $V(\Pk)$ and function $f \colon V(H) \to \bR$,
	\begin{equation*}
		\int_{V(H \os \Pk)} \Lin(f) d(\nu_{H}\os \mu_{\Pk}) = \int_{V(H)} f d\mu(\nu_{H}).
	\end{equation*}
\end{lemm}

\begin{proof}
Let $\nu_{\Pk},f$ be as above. It is easily verified from the definition that
\begin{equation} \label{eq:Linaverage}
    \frac{\Lin(f)(e\os \tfrac{i}{k})+\Lin(f)(e\os (1-\tfrac{i}{k}))}{2} = \frac{f(e^-)+f(e^+)}{2}
\end{equation}
for every $e \in E(H)$ and $\tfrac{i}{k} \in V(\Pk)$. Then using the reflection invariance of $\mu_{\Pk}$ we have
\begin{align*}
	\int_{V(H \os \Pk)} \Lin(f) d(\nu_{H}\os \mu_{\Pk}) &\overset{\eqref{eq:slash-measure}}{=} \int_{E(H)} \left(\int_{V(\Pk)} \Lin(f)(e \os \tfrac{i}{k}) \mu_{\Pk}(\tfrac{i}{k}) \right)d\nu_{H}(e) \\
	&= \int_{E(H)} \left(\int_{V(\Pk)} \Lin(f)(e \os \tfrac{i}{k}) \frac{\left(\mu_{\Pk}(\tfrac{i}{k})+\mu_{\Pk}(1-\tfrac{i}{k})\right)}{2} \right)d\nu_{H}(e) \\
	&\overset{\eqref{eq:Linaverage}}{=} \int_{E(H)} \left(\int_{V(\Pk)} \frac{f(e^-)+f(e^+)}{2} \mu_{\Pk}(\tfrac{i}{k}) \right)d\nu_{H}(e) \\
	&= \int_{E(H)} \frac{f(e^-)+f(e^+)}{2} d\nu_{H}(e) \\
	&\overset{\eqref{eq:induced-integral}}{=} \int_{V(H)} f d\mu(\nu_{H}).
\end{align*}
\end{proof}

Barycentric extension operators preserve $L_\infty$-norms, Lipschitz constants, and, under certain restrictions, $L_1$-norms.

\begin{prop} \label{prop:Linprops}
For any graph $H$ and function $f \colon V(H) \to \bR$,
\begin{itemize}
    \item $\norm{ \Lin(f) }_{\infty} = \norm{ f }_{\infty}$,
    \item $\Lip(\Lin(f)) = \Lip(f)$.
\end{itemize}
Moreover, if $\nu_H$ is a measure on $E(H)$, $\mu_{\Pk}$ is a reflection invariant probability measure on $V(\Pk)$, and $f$ satisfies the edge-sign property, then
	\begin{itemize}
		\item $\norm{ \Lin(f)}_{L_1(\nu_{H} \os \mu_{\Pk})} = \norm{ f }_{L_1(\mu(\nu_{ H}))}$.
	\end{itemize}
\end{prop}

\begin{proof}
The first two items are obvious and we omit their proofs. For the third, since $f$ has the edge-sign property, it is clear that $|\Lin(f)| = \Lin(|f|)$. Together with Lemma \ref{lem:intLin}, this gives us
	\begin{align*}
		\int_{V(H)} |f| d\mu(\nu_{ H})\overset{\text{Lem }\ref{lem:intLin}}{=} \int_{V(H \os \Pk)} \Lin(|f|) d(\nu_{H} \os \mu_{\Pk}) = \int_{V(H \os \Pk)} |\Lin(f)|d(\nu_{H} \os \mu_{\Pk}).
	\end{align*} 
\end{proof}

The last proposition shows how one may commute pullback operators with the gradient operator, which implies that pullback operators preserve Lipschitz constants. It is used in the proof of Theorem~\ref{thm:Lipgrowth}.

\begin{prop} \label{prop:pullbackLip}
Let $\theta: V(G) \to V(G')$ a surjective graph morphism between graphs and $\sd_{G'},\sd_G$ geodesic metrics on $V(G'),V(G)$ such that $\sd_{G'}(\theta(e)) = \sd_{G}(e)$ for every $e \in E(G)$. Then for every $f: V(G') \to \bR$,
\begin{itemize}
    \item $(\nabla_{\sd_G} \circ \theta^*)(f) = (\theta^* \circ \nabla_{\sd_{G'}})(f)$,
    \item $\Lip(\theta^*(f)) = \Lip(f)$.
\end{itemize}
\end{prop}

\begin{proof}
Let $f: V(G') \to \bR$ and $e \in E(G')$ be arbitrary. Since $\theta$ is a graph morphism, $\theta(e^\pm) = \theta(e)^\pm$. Then we have
\begin{align*}
	\nabla_{\sd_G}(\theta^*(f))(e) &= \frac{\theta^*(f)(e^+) - \theta^*(f)(e^-)}{\sd_{G}(e)} = \frac{f(\theta(e^+)) - f(\theta(e^-))}{\sd_{G}(e)} \\
	&= \frac{f(\theta(e)^+) - f(\theta(e)^-)}{\sd_{G}(e)} = \frac{f(\theta(e)^+) - f(\theta(e)^-)}{\sd_{G'}(\theta(e))} \\
	&= (\nabla_{\sd_{G'}} f)(\theta(e)) = \theta^*(\nabla_{\sd_{G'}} f)(e).
\end{align*}
	
The second item follows from the first and the fact that $\|\theta^*(g)\|_\infty = \|g\|_\infty$ since $\theta$ is surjective.
\end{proof}

\subsection{Proofs of main theorems}
In this final subsection, we provide the proofs of Theorems \ref{thm:stronglyorthog}, \ref{thm:Lpnorms}, and \ref{thm:Lipgrowth}.
We start with the proof of Theorem \ref{thm:stronglyorthog} regarding the preservation of strong orthogonality. This proof requires Propositions \ref{prop:epsslashcommute}, \ref{prop:epspullbackcommute}, \ref{prop:linearcombo}, and \ref{prop:condexpslashcommute}.

\begin{proof}[Proof of Theorem \ref{thm:stronglyorthog}]
Let $f \neq f' \in F_1$ and $\eps,\eps' \in \{-,+\}$. By Proposition~\ref{prop:linearcombo}, there are scalars $c_1,c_2,c_3,c_4 \in \bR$ such that, for every $e_1 \in E(H)$,
\begin{equation} \label{eq:linearcombo}
\int_{E(\Pk)} (\Lin(f)_\eps\Lin(f')_{\eps'})(e_1\os e_2) d\pi_\#\nu_{G}(e_2) = (c_1f_-f'_- + c_2f_-f'_+ + c_3f_+f'_- + c_4f_+f'_+)(e_1).
\end{equation}
Then we have
\begin{align*}
    \int_{E(H\os G)} ((id_{H}\os\pi)^* \circ \Lin)(f))_\eps & ((id_{H}\os\pi)^* \circ \Lin)(f'))_{\eps'} d(\nu_{H}\os\nu_G)\\
     & \overset{\text{Prop }\ref{prop:epspullbackcommute}}{=} \int_{E(H\os G)} (id_{H}\os\pi)^*(\Lin(f)_\eps) (id_{H}\os\pi)^*(\Lin(f')_{\eps'}) d(\nu_{H}\os\nu_G) \\
    &= \int_{E(H\os\Pk)} \Lin(f)_\eps \Lin(f')_{\eps'} d(\nu_{H}\os\pi_\#\nu_{G}) \\
    &= \int_{E(H)}\int_{E(\Pk)} (\Lin(f)_\eps\Lin(f')_{\eps'})(e_1\os e_2) d\pi_\#\nu_{G}(e_2)  d\nu_{H}(e_1) \\
    &\overset{\eqref{eq:linearcombo}}{=} \int_{E(H)} (c_1f_-f'_- + c_2f_-f'_+ + c_3f_+f'_- + c_4f_+f'_+) d\nu_{H} \\
    &= 0,
\end{align*}
where the last equality holds since $f,f'$ are assumed to be strongly $\nu_H$-orthogonal. This proves that $((id_{H} \os \pi)^* \circ \Lin)(F_1)$ is strongly $\nu_H\os\nu_G$-orthogonal.

Now let $f \os g \neq f' \os g' \in F_2 \os F_3$. Then we have
\begin{align*}
    \int_{E(H\os G)} (f\os g)_\eps(f'\os g')_{\eps'} d(\nu_{H}\os\nu_G) &\overset{\text{Prop }\ref{prop:epsslashcommute}}{=} \int_{E(H\os G)} (f\os g_\eps)(f'\os g'_{\eps'}) d(\nu_{H}\os\nu_G) \\
    &= \int_{E(H)} ff' d\nu_H \int_{E(G)} g_\eps g'_{\eps'} d\nu_G \\
    &=0,
\end{align*}
where the last equality holds since $F_2$ is $\nu_H$-orthogonal and $F_3$ is strongly $\nu_G$-orthogonal. This proves that $F_2 \os F_3$ is strongly $\nu_H\os\nu_G$-orthogonal.

It remains to verify strong $\nu_H\os\nu_G$-orthogonality between $((id_{H}\os\pi)^* \circ \Lin)(F_1)$ and $F_2 \os F_3$. Let $((id_{H}\os\pi)^* \circ \Lin)(f) \in ((id_{H}\os\pi)^* \circ \Lin)(F_1)$ and $f'\os g' \in F_2 \os F_3$. It follows immediately from Proposition~\ref{prop:epspullbackcommute} that $((id_{H}\os\pi)^* \circ \Lin)(f)_\eps$ is $\sigma(id_H\os\pi)$-measurable. Then if we can show $(f'\os g')_{\eps'} \in \ker(\bE_{\nu_H\os\nu_G}^{id_H\os\pi})$, we have the desired orthogonality and the proof is complete.
\begin{align*}
    \bE_{\nu_H\os\nu_G}^{id_H\os\pi}((f'\os g')_{\eps'}) \overset{\text{Prop }\ref{prop:epsslashcommute}}{=} \bE_{\nu_H\os\nu_G}^{id_H\os\pi}(f'\os g'_{\eps'}) \overset{\text{Prop }\ref{prop:condexpslashcommute}}{=} f'\os\bE_{\nu_G}^{\pi}(g'_{\eps'}) = 0,
\end{align*}
where the last equation holds by assumption on $F_3$.
\end{proof}

We now provide the details of the proof of Theorem \ref{thm:Lpnorms} pertaining to the preservation of edge-sign property and $L_1,L_\infty$-norms. This proof requires Propositions \ref{prop:slashprops}, \ref{prop:measurepreserve}, and \ref{prop:Linprops}.

\begin{proof}[Proof of Theorem \ref{thm:Lpnorms}]
Let $((id_{H}\os\pi)^* \circ \Lin)(f_1) \in ((id_{H}\os\pi)^* \circ \Lin)(F_1)$ and $f_2 \os f_3 \in F_2 \os F_3$. First we have, for $p \in \{1,\infty\}$,
\begin{align*}
    \|((id_{H}\os\pi)^* \circ \Lin)(f_1)\|_{L_p(\nu_H\os\mu(\nu_G))} &\overset{\text{Prop }\ref{prop:measurepreserve}}{=} \|\Lin(f_1)\|_{L_p(\nu_H\os\mu(\pi_\#\nu_G))} \overset{\text{Prop }\ref{prop:Linprops}}{=} \|f_1\|_{L_p(\mu(\nu_H))}, \\
    \|f_2 \os f_3\|_p &\overset{\text{Prop }\ref{prop:slashprops}}{=} \|f_2\|_p \cdot \|f_3\|_p,
\end{align*}
which proves \eqref{eq:thm2Linfty} and \eqref{eq:thm2L1}.

Furthermore, it is clear that $((id_{H}\os\pi)^* \circ \Lin)(f_1)$ has the edge-sign property since $f_1$ does and that $f_2 \os f_3$ has the edge-sign property since $f_3$ does, proving \eqref{eq:thm2edgesign}.
\end{proof}

Final, the proof of Theorem \ref{thm:Lipgrowth} is given below, thereby completing the proof of the main results. This proof requires Propositions \ref{prop:slashprops}, \ref{prop:Linprops}, and \ref{prop:pullbackLip}.
\begin{proof}[Proof of Theorem \ref{thm:Lipgrowth}]
Let $s \geq 0$. Of course, it is easy to see from the definitions that it suffices to prove
\begin{align*}
    \gamma_{((id_{H}\os\pi)^* \circ \Lin)(F_1)}(s) &= \gamma_{F_1}(s), \\
    \gamma_{F_2 \os F_3}(s) &\geq |F_2|\cdot\gamma_{F_3}\left(\frac{s}{\sup\Lip(F_2)}\right),
\end{align*}
where $\sup\Lip(F_2) := \sup_{f_2 \in F_2}\sup_{e \in E(H)} \frac{|f_2(e)|}{\sd_H(e)}$, and the above follow from
\begin{align*}
    \Lip(((id_{H}\os\pi)^* \circ \Lin)(f_1)) &= \Lip(f_1), \\
    \Lip(f_2 \os f_3) &\leq \sup_{e\in E(H)}\frac{|f_2(e)|}{\sd_H(e)}\Lip(f_3) \\
    |((id_{H}\os\pi)^* \circ \Lin)(F_1)| &= |F_1| \\
    |F_2 \os F_3| &= |F_2||F_3|
\end{align*}
for every $f_1 \in F_1$, $f_2 \in F_2$, and $f_3 \in F_3$. The first line follows from Propositions \ref{prop:Linprops} and \ref{prop:pullbackLip}, the second from Proposition \ref{prop:slashprops}, and the third and fourth are obvious.
\end{proof}

\appendix
\section{~}
In this short appendix we recall for the convenience of the reader the construction of orthogonal sets needed in the proof of Corollary \ref{cor:spec-slashpower}.

\begin{lemm} \label{lem:Hadamard}
Let $\bP$ be the uniform probability measure on a finite set $\Omega$. Then there exists a collection of functions  $\{f_j \colon  \Omega \to \bR\}_{j\in J}$ such that
\begin{itemize}
    \item $\{f_j\}_{j\in J}$ is orthogonal as a subset of $L_2(\Omega,\bP)$,
    \item $\sup_{j \in J} \|f_j\|_{L_\infty(\bP)} \leq 1$,
    \item $\inf_{j \in J} \|f_j\|_{L_1(\bP)} \geq \frac{1}{2}$, and
    \item $|J| \geq \frac{1}{2}|\Omega|$.
\end{itemize}
\end{lemm}

\begin{proof}
Let $n \in \bN$ such that $2^{n} \leq |\Omega| < 2^{n+1}$. Choose any subset $S \subset \Omega$ with $|S| = 2^{n}$, and choose an arbitrary enumeration of its elements, say $S := \{s_i\}_{i=1}^{2^n}$. Let $H = [h_{ij}]_{i,j=1}^{2^n}$ be a $2^n \times 2^n$ Hadamard matrix, meaning one whose columns (and therefore rows) are orthogonal and such that $h_{ij} \in \{-1,1\}$ for every $1\le i,j \leq 2^n$. Such a matrix exists by Sylvester's construction \cite[$\S$~2.1.1]{Hadamard}. For each $1 \leq j \leq 2^n$, we associate to the $j$th column of $H$ a function $f_j\colon \Omega \to \bR$ defined by
\begin{align*}
    f_j(\omega) \eqd \begin{cases} h_{ij} & \omega = s_i \\
    0 & \omega \not\in S. \end{cases}
\end{align*}
Then the collection $\{f_j\}_{j=1}^{2^n}$ satisfies the four desired properties.
\end{proof}

\bibliographystyle{alpha}


\end{document}